\documentclass[a4paper,american,reqno]{amsart}

\newif\ifPreprint \Preprinttrue
\newif\ifSubmission \Submissionfalse
\newif\ifLongVersion \LongVersionfalse

\usepackage{babel}
\usepackage[utf8]{inputenc}
\usepackage[T1]{fontenc}
\usepackage{todonotes}
\usepackage{booktabs}
\usepackage{csquotes}
\usepackage{microtype}
\usepackage{xcolor}
\usepackage{placeins}
\usepackage[binary-units]{siunitx}
\usepackage[style = authoryear,
            maxbibnames = 100,
            maxcitenames = 2,
            giveninits = true,
            uniquename = init,
            isbn = false,
            backend = bibtex]{biblatex}
\usepackage[colorlinks,
            citecolor=blue,
            urlcolor=blue,
            linkcolor=blue]{hyperref} %

\makeatletter
\patchcmd{\@settitle}{\uppercasenonmath\@title}{\scshape\large}{}{}
\patchcmd{\@setauthors}{\MakeUppercase}{\scshape\normalsize}{}{}
\makeatother

\vfuzz \hfuzz
\newcommand{\define}{\mathrel{{\mathop:}{=}}}

\newcommand{\set}[1]{\{#1\}}

\newcommand{\Defset}[3][\defsep]{\Set{#2#1#3}}
\newcommand{\st}{\text{s.t.}}
\newcommand{\field}{\mathbb}
\newcommand{\naturals}{\field{N}}

\newcommand{\reals}{\field{R}}

\newcommand{\N}{\naturals}

\newcommand{\R}{\reals}

\newcommand{\Card}{\Abs}
\makeatletter
\newcommand{\fcdot}{\,\cdot\,}
\newcommand{\fcarg}[1]{\def\fc@rg{#1}\ifx\fc@rg\empty\fcdot\else\fc@rg\fi}
\makeatother

\newcommand{\Abs}[1]{\left\lvert#1\right\rvert}

\newtheorem{lemma}{Lemma}
\newtheorem{remark}{Remark}

\definecolor{myblue}{HTML}{4575b4}
\definecolor{myorange}{HTML}{fc8d59}
\definecolor{myred}{HTML}{d7191c}
\newcommand{\blue}[1]{\textcolor{myblue}{#1}}
\newcommand{\orange}[1]{\textcolor{myorange}{#1}}
\newcommand{\red}[1]{\textcolor{myred}{#1}}

\newcommand{\rev}[1]{#1}

\bibliography{dd-rob-via-bilevel}

\begin{document}

\title[Solving Decision-Dependent Robust Problems as Bilevel
Problems]{A Computational Study for Solving Decision-Dependent Robust
  Problems as Bilevel Optimization Problems}

\author[H. Lefebvre, M. Schmidt, S. Stevens, J. Thürauf]%
{Henri Lefebvre, Martin Schmidt, Simon Stevens, Johannes Thürauf}

\address[M. Schmidt, S. Stevens]{%
  Trier University,
  Department of Mathematics,
  Universitätsring 15,
  54296 Trier,
  Germany}
\email{martin.schmidt@uni-trier.de}
\email{stevens@uni-trier.de}

\address[H. Lefebvre]{%
  LIRMM, University of Montpellier, CNRS, Montpellier, France
}
\email{henri.lefebvre@cnrs.fr}

\address[J. Thürauf]{%
  University of Technology Nuremberg (UTN),
  Department Liberal Arts and Social Sciences,
  Discrete Optimization,
  Dr.-Luise-Herzberg-Str. 4,
  90461 Nuremberg,
  Germany}
\email{johannes.thuerauf@utn.de}

\date{\today}

\begin{abstract}
  Both bilevel and robust optimization are established fields of mathematical
optimization and operations research. However, only until recently, the
similarities in their mathematical structure have neither been studied
theoretically nor exploited computationally. Based on the recent results by
\textcite{goerigk2025}, this paper is the first one that provides an
extensive computational study for solving strictly robust optimization
problems with decision-dependent uncertainty sets as equivalent
bilevel optimization problems. If the uncertainty
set can be dualized, the respective bilevel techniques to obtain a single-level
reformulation are very similar compared with the classic dualization techniques
used in robust optimization but lead to larger single-level problems to be
solved. Our numerical study shows that this usually leads to larger
computation times. For the more challenging case
of decision-dependent uncertainty sets represented by mixed-integer linear
models, one cannot apply classic dualization techniques from robust
optimization. Thus, we compare the
presented bilevel approach with an established method from the literature,
which is based on quantified mixed-integer linear programs. Our numerical
results indicate that, for the problem class of decision-dependent robust
optimization problems with mixed-integer linear uncertainty sets, the
bilevel approach performs better in terms of computation times.

\end{abstract}

\keywords{Robust optimization,
Decision-dependent uncertainty sets,
Endogenous uncertainty,
Bilevel optimization%
}
\subjclass[2020]{90C70, %
90C17, %
91A65%
}

\maketitle

\section{Introduction}
\label{sec:introduction}

Bilevel optimization deals with models of hierarchical decision making
in which a so-called leader acts first, while anticipating the optimal
reaction of the so-called follower, whose decision depends on the one
of the leader.
For a general overview of bilevel optimization we refer to
\textcite{Dempe:2002,Dempe-et-al:2015} and the more recent surveys by
\textcite{kleinert2021,Beck_et_al:2023}.
Hence, the overall structure is that of a nested optimization problem.
In robust optimization, one only considers a single decision-maker but
explicitly takes into consideration that this agent has to make a
decision under uncertainty.
In classic robust optimization, the decision is taken before the
uncertainty realizes and this uncertainty is assumed to realize in the
worst-case sense, which again is represented by a nested optimization
problem \parencite{soyster1973,Ben-Tal_et_al:2009}.
Hence, both types of problems---although being introduced to model
completely different aspects of real-world decision making---exhibit a
rather similar mathematical structure.

To the best of our knowledge, the first publication in which this
similarity has been observed is the one by \textcite{Stein:2013} in
the context of (generalized) semi-infinite optimization.
Nevertheless, the literature on robust and bilevel optimization has
been rather disjoint.
Besides some comments in this direction by
\textcite{Leyffer_et_al:2020}, the first systematic study of the
similarities and the differences of the mathematical structure of
bilevel and robust optimization has been recently published by
\textcite{goerigk2025}.
In particular, the authors show the equivalence of certain classes of
robust and bilevel optimization problems, paving the way for using the
computational methods from one field to solve problems from the other one.
In this paper, we will exploit the particular equivalence between
bilevel optimization and decision-dependent robust optimization to
solve the latter, and present the first extensive computational study
for this approach.

One of the main concerns in robust optimization is to reduce the
conservatism that is inherent to robust counterparts and their
solutions.
Many developments have thus focused on less conservative
approaches for modeling the uncertainty sets, including polyhedral,
ellipsoidal \parencite{ben1999} or budgeted \parencite{bertsimas2003}
uncertainty sets.
These approaches often provide a more realistic representation of
uncertainties, thereby resulting in less conservative solutions.
In addition to traditional robust optimization, adaptive robust
optimization and decision-dependent robust optimization (DDRO)
have gained more attention in recent years.
The former incorporates dynamic decision-making that adapts as
uncertainties unfold, typically modeled using a multi-stage structure
as in \textcite{bertsimas2011} and as in the survey by~\textcite{Hertog2019}.
This approach may significantly reduce the conservatism of the model.
DDRO addresses situations in which uncertainties
are modeled as being dependent on the decision variables,
enabling some control over these uncertainties within the model.

Decision-dependent uncertainties are sometimes also called endogenous
uncertainties and were initially introduced in the context of
stochastic optimization problems; see, e.g.,
\textcite{jonsbraaten1998}.
Since then, several papers on this topic have been
published in the field of stochastic optimization
\parencite{zhan2016,apap2017,hellemo2018,motamed2021}, to
name just a few.

Within robust optimization, however, the field of decision-dependent uncertainty
is less mature. Notable methodological contributions include the work of
\textcite{nohadani2018}, who consider shortest-path problems with
decision-dependent uncertainties in the arc lengths. They also show that even
linear DDRO problems with polyhedral uncertainty sets and affine
decision-dependence are NP-complete in general. Similarly,
\textcite{poss2013,poss2014} examine combinatorial optimization problems such as
the knapsack problem under budgeted uncertainty sets as well as under knapsack
uncertainties. \rev{More recently, \textcite{ley2026} study decision-dependent
robust optimization problems in which uncertainty sets can be modified
elementwise at a cost.} Furthermore, decision-dependent robust optimization is
also of interest in many application areas such as software partitioning
\parencite{spacey2012}, scheduling \parencites{lappas2016,vujanic2016}, energy
networks~\parencite{aigner2022}, or health care~\parencite{zhu2022}. Finally, in
recent years, decision-dependent uncertainties have also been explored within
more specialized fields of robust optimization, such as two-stage and multistage
robust optimization \parencites{zhang2020,
avraamidou2020,zeng2022,chen2026decision} or distributionally robust
optimization
\parencites{luo2020,feng2021,basciftci2021,yu2022,doan2022,ryu2025}.

To the best of our knowledge, most of the techniques to solve
single-level DDRO problems rely on dualization or problem-specific knowledge
to reformulate the robust optimization problem as a finite-dimensional problem
consisting of finitely many variables and constraints. Consequently, these
approaches are generally limited to cases with continuous and convex
uncertainty sets that can be dualized, i.e., for which a strong-duality
theorem is available. If such a theorem is not available, e.g., in cases in
which the decision-dependent uncertainty set is represented as a mixed-integer
linear model, the only well-maintained and publicly available code to
solve the respective decision-dependent robust problems we are aware
of is the \textsf{Yasol} solver for quantified mixed-integer
problems \parencite{yasol2017,hartisch_lorenz:2022}.
More recently, new approaches have also been
introduced by~\textcite{lozano2025}.
In particular, the authors present a column-and-constraint generation
algorithm and a method based on decision diagrams.

We use the recent result by
\textcite{goerigk2025} showing that decision-dependent robust optimization
problems can be equivalently reformulated as bilevel optimization problems.
Our main contribution is to present the first extensive computational
study that illustrates the potential of solving DDRO problems as
bilevel optimization problems.
To this end, we consider classic operations-research problems that
have already been studied in the literature on DDRO: the shortest-path
problem, also considered by \textcite{nohadani2018} in the DDRO
setting, the knapsack problem, also considered by
\textcite{poss2013,poss2014} in DDRO, and the portfolio optimization
problem in its DDRO version.
For the case of continuous and convex decision-dependent uncertainty
sets, we derive the single-level reformulation of the corresponding
bilevel problem, solve it, and compare it with solving the given
problem using classic dualization techniques from robust optimization
and standard linearization techniques à la \textcite{mccormick1976}.
Our numerical results reveal that the bilevel problem usually leads to
larger computation times due to larger model sizes.
Second, we consider the first two of the three mentioned
problems but with decision-dependent uncertainty sets for which classic
dualization techniques are not applicable. However, we use the most recent
advances in mixed-integer linear bilevel optimization and solve the
corresponding bilevel problem using the publicly available open-source solver
\textsf{MibS} \parencite{mibs,DeNegre-Ralphs:2009,Tahernejad_Ralphs:2020}.
This enables us to present numerical results for the considered class of
problems and to compare them with the results obtained by using the \textsf{Yasol}
solver.
The results demonstrate that using \textsf{MibS} for the bilevel
reformulation consistently performs better than applying \textsf{Yasol} with
respect to computation times.

\ifLongVersion
The remainder of this paper is structured as follows.
In Section~\ref{sec:general-model}, we present the DDRO models that we
discuss in this paper and review the classic robust reformulation as
well as the respective bilevel reformulations based on strong-duality
or Karush--Kuhn--Tucker conditions.
Afterward, in Section~\ref{sec:uncertainty-sets} we introduce the
uncertainty sets that we use later in our numerical study:
decision-dependent budgeted uncertainty and decision-dependent
knapsack uncertainty---both in a continuous and a mixed-integer
variant.
Section~\ref{sec:applications} then discusses the three mentioned
applications together with their uncertainty models.
Then, we present and discuss our numerical results in
Section~\ref{sec:num-results} before we conclude the paper
with some comments on future research questions in
Section~\ref{sec:conclusion}.
\fi

\section{Reformulations of a General DDRO Problem}
\label{sec:general-model}

In this section, we present equivalent reformulations of general
decision-dependent robust optimization problems. First, we will discuss the
classic robust reformulation that is commonly used in the literature in
Section~\ref{sec:general-model:classic-robust-reformulation}. Then, we derive a
bilevel reformulation in Section~\ref{sec:general-model:bilevel-reformulation}
together with its single-level reformulations using strong
duality~(Section~\ref{sec:general-model:bilevel-reformulation:sd}) and
Karush--Kuhn--Tucker (KKT)
conditions~(Section~\ref{sec:general-model:bilevel-reformulation:kkt}).

To this end, we consider the general DDRO problem
\begin{subequations}
  \label{eq:general-model}
  \begin{align}
    \min_{x \in X} \quad
    & c^\top x
    \\
    \st \quad
    &a_i^\top x + u_i^\top B_i x \leq \beta_i
      \quad
      \forall u_i \in U_i(x),
      \quad
      i \in [m_x],
      \label{eq:general-model:uncertain-constraint}
  \end{align}
\end{subequations}
where $X$ is a non-empty and compact set that may contain integrality
conditions for some or all $x$-variables.
\rev{More specifically, we will later suppose that \hbox{$X \subseteq
  \set{0,1}^{k_x} \times \R^{n_x - k_x}$} holds.}
Here, $c \in \R^{n_x}$, $a_i \in \R^{n_x}$, $B_i \in \R^{n_i \times
  n_x}$, $\beta_i \in \R, \ i \in [m_x]$, are given data.
The uncertainty sets are supposed to be polytopes, possibly
intersected with some integer lattice, i.e., we have
\begin{equation}
  \label{eq:general-model:uncertainty-set}
  U_i(x) \define
  \Defset{u_i \in \set{0,1}^{k_i} \times \R^{n_i - k_i}}
  {C_i x + D_i u_i \leq \alpha_i},
  \quad
  i \in [m_x],
\end{equation}
with given data $C_i \in \R^{m_i \times n_x}$, $D_i \in \R^{m_i \times n_i}$,
as well as $\alpha_i \in \R^{m_i}$ and where the binary assumption is
again w.l.o.g.

\begin{remark}
  \label{rem:general-model:wlog-uncertain-constraints}
  Note that it is w.l.o.g.\ to assume that the uncertainty only
  appears in the constraints since any uncertainty in the objective function can
  be moved to the constraints by introducing an auxiliary epigraph variable.
\end{remark}

\begin{remark}
  \rev{We note that assuming the integer part of~$X$ to be binary, as we later
  do, is mathematically equivalent to considering bounded integers, since
  bounded integer variables can always be represented by binary
  variables.
  However, such reformulations usually lead to significantly larger
  models so that, computationally speaking, it might make a big
  difference if one considers binary variables or bounded integers.}
\end{remark}

\subsection{Classic Robust Reformulation}
\label{sec:general-model:classic-robust-reformulation}

In this section, we present the classic robust reformulation of Problem
\eqref{eq:general-model} that is commonly used in the literature.
Thus, we first re-write the uncertain
constraints~\eqref{eq:general-model:uncertain-constraint} as
\begin{equation}
  \label{eq:general-model:classic-robust-reformulation:uncertain-constraint}
  a_i^\top x + \max_{u_i \in U_i(x)} u_i^\top B_i x
  \leq \beta_i,
  \quad
  i \in [m_x],
\end{equation}
which is well-posed due to the non-emptiness and compactness assumption for the
uncertainty sets. This is valid since if
Constraints~\eqref{eq:general-model:uncertain-constraint} have to hold for all
$u_i \in U_i(x)$, it is equivalent to requiring that they hold for the
worst-case. If $k_i = 0$ holds, i.e., $U_{i}(x)$ is polyhedral, we can dualize
the inner maximization problems
in~\eqref{eq:general-model:classic-robust-reformulation:uncertain-constraint} to
obtain the dual problems
\begin{equation}
  \label{eq:general-model:classic-robust-reformulation:dual-inner-problem}
  \begin{aligned}
    \min_{\lambda_i} \quad
    &(\alpha_i - C_i x)^\top \lambda_i
    \\
    \st \quad
    &D_i^\top \lambda_i = B_i x,
    \\
    &\lambda_i \geq 0
  \end{aligned}
\end{equation}
for all $i \in [m_x]$.
By strong duality, which is ensured to hold under the stated assumptions,
the dual problems~\eqref{eq:general-model:classic-robust-reformulation:dual-inner-problem}
have the same optimal objective value as the primal problems. In the following,
we write $u = (u_i)_{i\in[m_x]}$ and $\lambda = (\lambda_i)_{i\in[m_x]}$
for the collections of all primal and dual variables, respectively.

We can now replace the inner maximization problems by their duals to obtain the
classic robust reformulation of Problem \eqref{eq:general-model}:
\begin{subequations}
  \label{eq:general-model:classic-robust-reformulation}
  \begin{align}
    \min_{x \in X, \lambda} \quad
    &c^\top x
    \\
    \st \quad
    &a_i^\top x + (\alpha_i - C_i x)^\top \lambda_i
      \leq \beta_i,
      \quad
      i \in [m_x],
      \label{eq:general-model:classic-robust-reformulation:reformulated-uncertain-constraint}
    \\
    &D_i^\top \lambda_i = B_i x,
      \quad
      i \in [m_x],
    \\
    &\lambda_i \geq 0,
      \quad
      i \in [m_x].
  \end{align}
\end{subequations}
Using McCormick envelopes~\parencite{mccormick1976}, the bilinear terms $x^\top
C_i^\top \lambda_i$ in
Constraints~\eqref{eq:general-model:classic-robust-reformulation:reformulated-uncertain-constraint}
can be linearized if the variables $x$ and $\lambda_i$ are bounded. Assuming
that~$X$ is described by linear constraints and integrality restrictions only,
this leads to an equivalent mixed-integer linear programming (MILP)
reformulation of Problem~\eqref{eq:general-model:classic-robust-reformulation},
provided that the variables~$x$ are integer. In contrast, for robust
optimization problems with decision-independent uncertainty, these bilinear
terms do not appear and the robust reformulation is directly given as an MILP.

\subsection{Bilevel Reformulation}
\label{sec:general-model:bilevel-reformulation}

Following the recent results by \textcite{goerigk2025}, we can also reformulate
Problem~\eqref{eq:general-model} as the bilevel optimization problem
\begin{equation}
  \label{eq:general-model:bilevel-reformulation}
  \begin{aligned}
    \min_{x \in X, u} \quad
    &c^\top x
    \\
    \st \quad
    &a_i^\top x + u_i^\top B_i x
      \leq \beta_i,
      \quad
      i \in [m_x],
    \\
    &u \in S(x),
  \end{aligned}
\end{equation}
where $S(x)$ is the set of optimal solutions to the $x$-parameterized
lower-level problem
\begin{equation}
  \label{eq:general-model:bilevel-reformulation:lower-level}
  \begin{aligned}
    \max_{u} \quad
    &\sum_{i\in[m_x]} u_i^\top B_i x
    \\
    \st \quad
    &C_i x + D_i u_i \leq \alpha_i,
      \quad
      i \in [m_x].
  \end{aligned}
\end{equation}
Here, the lower-level problem models the worst case scenario over the
uncertainty sets $U_i(x)$ defined in \eqref{eq:general-model:uncertainty-set}.
Note that we do not need the uncertainty sets to be given as continuous
polyhedra to use this reformulation, i.e., we can also allow for $k_i > 0$.
\begin{remark}
  \label{rem:general-model:bilevel-reformulation:solvability}
  Assume that $X$ is given by linear constraints and integrality
  restrictions only.
  Moreover assume that the bilinearities $u_i^\top B_i x$ can be
  linearized.
  Then, the bilevel
  problem~\eqref{eq:general-model:bilevel-reformulation} is solvable
  with off-the-shelf mixed-integer bilevel optimization (MIBLP)
  solvers like
  \textnormal{\textsf{MibS}~\parencite{mibs}}. Results on this will be shown in
  Section~\ref{sec:num-results}.
\end{remark}
If the uncertainty sets $U_i(x)$ are continuous, i.e., if $k_i = 0$ holds for
all $i \in [m_x]$, we can reformulate the bilevel problem
\eqref{eq:general-model:bilevel-reformulation} as a single-level optimization
problem using the strong-duality or KKT conditions of the lower-level
problem\rev{; see \textcite{Fortuny-Amat-McCarl:1981,Zare_et_al:2019} or
the recent book by \textcite{Beck_Ljubic_Schmidt:2026} for more
information about these single-level reformulations.}
This is described in the following section.

\subsubsection{Strong-Duality-Based Single-Level Reformulation}
\label{sec:general-model:bilevel-reformulation:sd}

In the following, we assume that the uncertainty sets $U_i(x)$ are
given by continuous variables only, i.e., $k_i = 0$ holds for all $i
\in [m_x]$.
Then, we can dualize the linear lower-level
problem~\eqref{eq:general-model:bilevel-reformulation:lower-level} and obtain
the dual problem
\begin{equation}
  \label{eq:general-model:bilevel-reformulation:sd:dual-ll}
  \begin{aligned}
    \min_{\lambda} \quad
    &\sum_{i\in[m_x]} (\alpha_i - C_ix)^\top \lambda_i
    \\
    \st \quad
    &D_i^\top \lambda_i = B_i x,
      \quad
      i \in [m_x],
    \\
    &\lambda_i \geq 0,
      \quad
      i \in [m_x].
  \end{aligned}
\end{equation}
Strong duality holds if and only if
\begin{equation*}
  \label{eq:general-model:bilevel-reformulation:sd:strong-duality}
  \sum_{i\in[m_x]} u_i^\top B_i x
  \geq \sum_{i\in[m_x]} (\alpha_i - C_i x)^\top \lambda_i,
\end{equation*}
where $u$ and $\lambda$ are solutions for the primal lower-level
problem~\eqref{eq:general-model:bilevel-reformulation:lower-level} and the dual
problem~\eqref{eq:general-model:bilevel-reformulation:sd:dual-ll}, respectively.
Using this, we can reformulate the bilevel
problem~\eqref{eq:general-model:bilevel-reformulation} as the single-level
mixed-integer optimization problem
\begin{subequations}
  \label{eq:general-model:bilevel-reformulation:sd:single-level}
  \begin{align}
    \min_{x \in X, u, \lambda} \quad
    &c^\top x
    \\
    \st \quad
    &a_i^\top x + u_i^\top B_i x
      \leq \beta_i,
      \quad
      i \in [m_x],
    \\
    &C_i x + D_i u_i \leq \alpha_i,
      \quad
      i \in [m_x],
      \label{eq:general-model:bilevel-reformulation:sd:single-level:primal-ll}
    \\
    &D_i^\top \lambda_i = B_i x,
      \quad
      i \in [m_x],
    \\
    &\lambda_i \geq 0,
      \quad
      i \in [m_x],
    \\
    &\sum_{i\in[m_x]} u_i^\top B_i x
      \geq \sum_{i\in[m_x]} (\alpha_i - C_i x)^\top \lambda_i.
      \label{eq:general-model:bilevel-reformulation:sd:single-level:sd-constraint}
  \end{align}
\end{subequations}
Note that this is the same problem as the classic robust
reformulation~\eqref{eq:general-model:classic-robust-reformulation} with the
addition of the strong duality
constraint~\eqref{eq:general-model:bilevel-reformulation:sd:single-level:sd-constraint}
as well as the primal lower-level variables $u$ and
constraints~\eqref{eq:general-model:bilevel-reformulation:sd:single-level:primal-ll}.

\subsubsection{KKT-Based Single-Level Reformulation}
\label{sec:general-model:bilevel-reformulation:kkt}

We now also present the single-level reformulation of the bilevel
problem~\eqref{eq:general-model:bilevel-reformulation} based on the KKT
conditions of the lower-level
problem~\eqref{eq:general-model:bilevel-reformulation:lower-level}. The KKT
conditions of the lower-level
problem~\eqref{eq:general-model:bilevel-reformulation:lower-level} are given as
\begin{equation*}
  \label{eq:general-model:bilevel-reformulation:kkt:kkt-conditions}
  \begin{aligned}
    C_i x + D_i u_i \leq \alpha_i,
    & \quad
      i \in [m_x],
    \\
    D_i^\top \lambda_i = B_i x,
    & \quad
      i \in [m_x],
    \\
    \lambda_i \geq 0,
    & \quad
      i \in [m_x],
    \\
    \lambda_i^\top (\alpha_i - C_i x - D_i u_i) = 0,
    & \quad
      i \in [m_x].
  \end{aligned}
\end{equation*}
Since the lower-level
problem~\eqref{eq:general-model:bilevel-reformulation:lower-level} is linear,
the KKT conditions are necessary and sufficient for optimality. Using these, we
can reformulate the bilevel
problem~\eqref{eq:general-model:bilevel-reformulation} as the single-level
optimization problem
\begin{subequations}
  \label{eq:general-model:bilevel-reformulation:kkt:single-level}
  \begin{align}
    \min_{x \in X, u, \lambda} \quad
    &c^\top x
    \\
    \st \quad
    &a_i^\top x + u_i^\top B_i x
      \leq \beta_i,
      \quad
      i \in [m_x],
    \\
    &C_i x + D_i u_i \leq \alpha_i,
      \quad
      i \in [m_x],
    \\
    &D_i^\top \lambda_i = B_i x,
      \quad
      i \in [m_x],
    \\
    &\lambda_i \geq 0,
      \quad
      i \in [m_x],
    \\
    &\lambda_i^\top (\alpha_i - C_i x - D_i u_i) = 0,
      \quad
      i \in [m_x].
      \label{eq:general-model:bilevel-reformulation:kkt:single-level:complementarity}
  \end{align}
\end{subequations}
Note that this is equivalent to the strong-duality-based reformulation
\eqref{eq:general-model:bilevel-reformulation:sd:single-level} with the
replacement of the strong-duality constraint
\eqref{eq:general-model:bilevel-reformulation:sd:single-level:sd-constraint} by
the complementarity-slackness
conditions~\eqref{eq:general-model:bilevel-reformulation:kkt:single-level:complementarity}.
The bilinearities in both single-level reformulations based on bilevel
optimization can be linearized using McCormick envelopes if the variables~$x$,
$u$, and $\lambda$ are bounded, analogously to the classic robust
reformulation~\eqref{eq:general-model:classic-robust-reformulation}.

\begin{remark}
  \label{rem:general-model:bilevel-reformulation:kkt:general-uncertainty}
  Note that the uncertainty sets $U_i(x)$ in
  \eqref{eq:general-model:uncertainty-set} can be more general sets as
  long as the lower-level
  problem~\eqref{eq:general-model:bilevel-reformulation:lower-level}
  allows for applying a strong-duality theorem.
  This, in general, is possible for convex sets that satisfy Slater's
  constraint qualification so that the KKT conditions are
  necessary and sufficient.
\end{remark}

\section{Uncertainty Sets}
\label{sec:uncertainty-sets}

In this section, we present the two most commonly used uncertainty sets known
from the literature on robust optimization with decision-dependent uncertainty
that we also consider in the computational experiments in
Section~\ref{sec:num-results}.
\rev{General introductions to uncertainty sets in robust optimization
  can, e.g., be found in the books by
  \textcite{Ben-Tal_et_al:2009,Goerigk_Hartisch:2024}.}
We also derive bounds on the dual
variables for each uncertainty set that can be used for linearization in the
resulting single-level reformulations.

We start with decision-dependent budgeted uncertainty sets in
Section~\ref{sec:uncertainty-sets:budgeted}, followed by decision-dependent
knapsack uncertainty sets in Section~\ref{sec:uncertainty-sets:knapsack}.

\subsection{Decision-Dependent Budgeted Uncertainty}
\label{sec:uncertainty-sets:budgeted}

The first type of uncertainty sets that we consider are decision-dependent budgeted
uncertainty sets known from, e.g., \textcite{nohadani2018}. The intuition behind
budgeted uncertainty sets is that not all uncertain parameters deviate from
their nominal values at the same time. Instead, only a limited number of
parameters can take their worst-case values simultaneously. This limitation is
modeled by an uncertainty budget. Via hedging variables, the decision-maker can
reduce the uncertainty in each parameter individually by paying additional
hedging costs.

We consider decision-dependent budgeted uncertainty sets in their continuous
form
\begin{equation*} \label{eq:uncertainty-sets:budgeted:continuous}
  U^{\text{cb}}(x) \define \Defset{u \in \R^n}
  {\sum_{i = 1}^n u_i \leq \Gamma,
    \ u_i \leq 1-\gamma_i x_i,
    \ u_i \geq 0, \ i \in [n]},
\end{equation*}
where $\Gamma \geq 0$ is the uncertainty budget and $\gamma \in [0,1]^{n}$ is
the vector of fractions of uncertainty reduction.

In their discrete form, the decision-dependent budgeted uncertainty sets are
given by
\begin{equation*} \label{eq:uncertainty-sets:budgeted:discrete}
  U^{\text{db}}(x) \define \Defset{u \in \set{0,1}^{n}}
  {\sum_{i = 1}^n u_i \leq \Gamma,
    \ u_i \leq 1 - x_i, \ i \in [n]}.
\end{equation*}
Note that in this case, the fraction of uncertainty reduction $\gamma$ is
assumed to be equal to 1 for all parameters since the uncertainty parameters are
binary. Thus, we are only able to fully protect an uncertain parameter or leave
it unprotected.

To linearize the bilinearities in the single-level reformulations derived in
Section~\ref{sec:general-model}, we need upper bounds on the dual variables of
the inner maximization problems. The dual problem of
$\max_{u} \Defset{u^\top B x}{u \in U^{\text{cb}}(x)}$
is given by
\begin{subequations}
  \label{eq:uncertainty-sets:budgeted:dual}
  \begin{align}
    \min_{\pi, \lambda} \quad & \Gamma \pi + \sum_{i \in [n]} \lambda_i (1 - \gamma_i x_i)
    \\
    \st \quad & \pi + \lambda_i \geq (B x)_i, \quad i \in [n],
    \label{eq:uncertainty-sets:budgeted:dual:constr}
    \\
    & \pi, \lambda_i \geq 0, \quad i \in [n].
  \end{align}
\end{subequations}
Constraints~\eqref{eq:uncertainty-sets:budgeted:dual:constr} hold for all $\pi
\geq \max_i (B x)_i$ and $\lambda_i \geq 0$. Thus, a valid upper bound on $\pi$
is given by
\begin{equation*}
  \overline{\pi} = \max \Defset{ (B x)_i }{ i \in [n] }.
\end{equation*}
Moreover, with the same type of argument,
a valid upper bound on $\lambda_i$ is given by
\begin{equation*}
  \overline{\lambda}_i = (B x)_i, \quad i \in [n].
\end{equation*}

\subsection{Decision-Dependent Knapsack Uncertainty}
\label{sec:uncertainty-sets:knapsack}

Another type of uncertainty set that we consider is the decision-dependent knapsack
uncertainty set. Here, the decision-dependence lies in the uncertainty budget
$b(x)$ instead of the fraction of uncertainty reduction. The uncertainty budget
is defined as a function of $x$ given by $b(x) = b + w^\top x$ with $w \in
\R^n_{\geq 0}$ and $b \in \R_{\geq 0}$. Whenever we consider this uncertainty,
we further assume that $b(x) \geq 0$ for all $x \in X$. Moreover, the
constraints in the uncertainty set are knapsack constraints with
weights $f_i \geq 0$. The decision-dependent knapsack uncertainty set is
therefore given by
\begin{equation*} \label{eq:uncertainty-sets:cont-knapsack}
  U^{\text{ck}}(x) \define \Defset{u \in [0,1]^n}
  {f^\top u \leq b + w^\top x}.
\end{equation*}
Here, the decision-maker influences the overall uncertainty budget but cannot
hedge against uncertainty in individual components. This uncertainty set was
introduced by~\textcite{poss2013} who highlights that, for given binary
variables~$x$, sparse vectors are inherently more protected from the uncertainty
than dense ones. The uncertainty budget itself grows with $\lVert x \rVert$, so
that more activated components of $x$ expose the model to more uncertainty.  The
discrete version of this uncertainty set is given by
\begin{equation*} \label{eq:uncertainty-sets:disc-knapsack}
  U^{\text{dk}}(x) \define U^{\text{ck}}(x) \cap \set{0,1}^{n}.
\end{equation*}

Similarly to the previous section, we need upper bounds on the dual variables of
the inner maximization problems to linearize the bilinearities in the
single-level reformulations derived in Section~\ref{sec:general-model}. The dual
problem of
$\max_{u} \Defset{u^\top B x}{u \in U^{\text{ck}}(x)}$
is given by
\begin{subequations}
  \label{eq:uncertainty-sets:knapsack:dual}
  \begin{align}
    \min_{\pi, \lambda} \quad & \pi (b + w^\top x) + \sum_{i \in [n]} \lambda_i
    \\
    \st \quad & \pi f_i + \lambda_i \geq (B x)_i, \quad i \in [n],
    \label{eq:uncertainty-sets:knapsack:dual:const}
    \\
    & \pi, \lambda_i \geq 0, \quad i \in [n].
  \end{align}
\end{subequations}
Assume that $f_i > 0$ for at least one $i \in [n]$.
Otherwise, the knapsack constraints in $U^{\text{ck}}(x)$ are
redundant.
Moreover, note that $b + w^\top x \geq 0$ holds by assumption.
Thus, a valid upper bound on $\pi$ is given by
\begin{equation*}
  \overline{\pi} = \max \Defset{ \frac{(B x)_i}{f_i} }{ i \in [n] },
\end{equation*}
since Constraints~\eqref{eq:uncertainty-sets:knapsack:dual:const} hold for all
$\pi \geq \bar{\pi}$. Moreover, with the same argument, a valid upper bound on
$\lambda_i$ is again given by
\begin{equation*}
  \overline{\lambda}_i = (B x)_i, \quad i \in [n].
\end{equation*}

\section{Applications}
\label{sec:applications}

In this section, we present three applications from the literature that can be
modeled as decision-dependent robust optimization problems. For each
application, we describe the nominal problem and derive the corresponding robust
problems. We start with the shortest-path problem in
Section~\ref{sec:applications:shortest-path}, followed by the knapsack problem
in Section~\ref{sec:applications:knapsack}, and end with the portfolio
optimization problem in Section~\ref{sec:applications:portfolio}.

\subsection{Shortest-Path Problem}
\label{sec:applications:shortest-path}

To model the nominal shortest-path problem, we consider a directed graph~$G =
(V,A)$, where $V$ is the set of all nodes and $A$ is the set of all arcs. The
objective is to find a shortest path from a given source node~$s \in V$ to a
target node~$t \in V$ with $t \neq s$. For an arc~$a \in A$, the variable~$y_a$
models whether the arc~$a$ is to be chosen in the shortest path and is
therefore binary. The nominal costs for using arc~$a \in A$ are given by $d_a
\geq 0$. The only constraints of the model are the flow~conservation~constraints
\begin{equation}
  \label{eq:applications:shortest-path:flow-conservation}
  \sum_{a \in \delta^{\text{in}}(v)} y_a - \sum_{a
    \in \delta^{\text{out}}(v)} y_a =
  \begin{cases}
    1, \ &v = t, \\\
    -1, \ &v = s, \\
    0, \ &\text{else},
  \end{cases}
  \ v \in V,
\end{equation}
in which we denote the ingoing and outgoing arcs of the node~$v\in V$ by the
sets $\delta^{\text{in}}(v) \define \Defset{(u,v) \in A}{u \in V \setminus
\set{v}}$ and $\delta^{\text{out}}(v) \define \Defset{(v,u) \in A}{u \in V
\setminus \set{v}}$. Since we minimize the costs, we obtain the nominal model
\begin{equation}
  \label{eq:applications:shortest-path:nominal-model}
  \min_{y \in \set{0,1}^{\Card{A}}} \quad
  \sum_{a \in A} d_a y_a \quad \st
  \quad~\eqref{eq:applications:shortest-path:flow-conservation}.
\end{equation}

\subsubsection{Budgeted Uncertainty}
\label{sec:applications:shortest-path:budgeted-uncertainty}

In the following, we assume decision-dependent uncertainties in the costs $d_a$.
To this end, we introduce the nominal cost $\bar{d}_a \geq 0$ on arc $a$ and the
maximum deviation of the cost $\hat{d}_a \geq 0$. Reducing the uncertainty for
arc $a$ by, e.g., investing in development of the road network, leads to hedging
costs represented by $h_a \geq 0$. The new variable~$x_a$ models the decision
whether to reduce the uncertainty on arc~$a$ or not. The objective function in
Problem~\eqref{eq:applications:shortest-path:nominal-model} is therefore
modified to obtain the robust shortest-path problem
\begin{equation}
  \label{eq:applications:shortest-path:cont-dd-rob-model}
  \begin{aligned}
    \min_{x,y} \quad
    &\sum_{a \in A} h_a x_a
      + \sum_{a \in A} \bar{d}_a y_a
      + \max_{u \in U_{\text{sp}}(x)} \
      \sum_{a \in A} u_a\hat{d}_a y_a
    \\
    \st \quad &\eqref{eq:applications:shortest-path:flow-conservation}, \
                x, y \in \set{0,1}^{\Card{A}}
  \end{aligned}
\end{equation}
with the decision-dependent budgeted uncertainty sets $U_{\text{sp}} =
U^{\text{cb}}$ or $U_{\text{sp}} = U^{\text{db}}$ as defined in
Section~\ref{sec:uncertainty-sets:budgeted}.

\subsubsection{Knapsack Uncertainty}
\label{sec:applications:shortest-path:knapsack-uncertainty}

We also consider decision-dependent uncertainties in the costs $d_a$ in form of
a knapsack uncertainty set. In contrast to the budgeted uncertainty set, the
knapsack uncertainty set is parameterized by the decision variables directly.
The uncertainty budget $b(y)$ grows with the number of chosen arcs. We obtain
the robust shortest-path problem
\begin{equation}
  \label{eq:applications:shortest-path:cont-dd-rob-model:knapsack}
  \begin{aligned}
    \min_{y} \quad
    &\sum_{a \in A} \bar{d}_a y_a
      + \max_{u \in U_{\text{sp}}(y)} \
      \sum_{a \in A} u_a \hat{d}_a y_a
    \\
    \st \quad &\eqref{eq:applications:shortest-path:flow-conservation}, \
                y \in \set{0,1}^{\Card{A}}
  \end{aligned}
\end{equation}
with the decision-dependent knapsack uncertainty sets $U_{\text{sp}} =
U^{\text{ck}}$ or $U_{\text{sp}} = U^{\text{dk}}$ as defined in
Section~\ref{sec:uncertainty-sets:knapsack}.

\subsection{Knapsack Problem}
\label{sec:applications:knapsack}

Consider a set of $n$~items, each with value~$c_i \geq 0$ and weight~$a_i \geq
0$, to be given. The goal is to choose a subset of these items so that the
available capacity~$d \geq 0$ of the knapsack is not exceeded and the overall
value is maximized. This results in the nominal knapsack~problem
\begin{equation*}
  \label{eq:applications:knapsack:nominal-model}
  \max_{y \in \set{0,1}^n} \quad
  c^\top y \quad \st \quad
  a^\top y \leq d.
\end{equation*}

\subsubsection{Budgeted Uncertainty}
\label{sec:applications:knapsack:budgeted-uncertainty}

We now consider decision-dependent uncertainties in the weights $a_i$ of the
items. To this end, we introduce the nominal weight $\bar{a}_i \geq 0$ of item
$i$ and the maximum deviation~$\hat{a}_i \geq 0$ of the weight. Reducing the
uncertainty for item~$i$ by, e.g., investing in better packaging, leads to
hedging costs represented by $h_i \geq 0$. The new variable~$x_i$ models the
decision whether to reduce the uncertainty on item~$i$ or not. Thus, the robust
knapsack problem with decision-dependent budgeted uncertainty is given by
\begin{equation}
  \label{eq:applications:knapsack:budgeted-dd-rob-model:budgeted}
  \begin{aligned}
    \max_{x,y} \quad
    &c^\top y - h^\top x
    \\
    \st \quad
    &\sum_{i \in [n]} \bar{a}_i y_i
      + \max_{u \in U_\text{k}(x)} \
      \sum_{i \in [n]} u_i \hat{a}_i y_i \leq d,
    \\
    &x, y \in \set{0,1}^n,
  \end{aligned}
\end{equation}
with the decision-dependent budgeted uncertainty sets $U_\text{k} =
U^{\text{cb}}$ or $U_\text{k} = U^{\text{db}}$ as defined in
Section~\ref{sec:uncertainty-sets:budgeted}.

\subsubsection{Knapsack Uncertainty}
\label{sec:applications:knapsack:knapsack-uncertainty}

We also consider decision-dependent uncertainties in the weights~$a_i$ in form
of a continuous knapsack uncertainty set. Similar to the shortest-path problem,
the uncertainty budget $b(y)$ grows with the number of chosen items.
We obtain the uncertain knapsack problem
\begin{equation}
  \label{eq:applications:knapsack:cont-dd-rob-model:knapsack}
  \begin{aligned}
    \max_{y \in \set{0,1}^n} \quad
    &c^\top y \quad
    \\
    \st \quad
    &\sum_{i \in [n]} \bar{a}_i y_i
      + \max_{u \in U_\text{k}(y)} \
      \sum_{i \in [n]} u_i \hat{a}_i y_i \leq d,
  \end{aligned}
\end{equation}
with the decision-dependent knapsack uncertainty sets $U_\text{k} =
U^{\text{ck}}$ or $U_\text{k} = U^{\text{dk}}$ as defined in
Section~\ref{sec:uncertainty-sets:knapsack}.

\subsection{Portfolio Optimization}
\label{sec:applications:portfolio}

We consider the cardinality-constrained portfolio
optimization problem. In this setting, investors allocate their budget
across \hbox{$N \in \N$}~different assets by assigning weights~$y_i \in [0,1]$ to each
asset~$i \in [N]$. These weights have to add up to one, i.e., the entire budget
has to be invested. We do not allow negative weights, thereby prohibiting short
positions. The expected return of asset $i$ is given by $\mu_i \in \R$ and
the covariance matrix is given by $\Sigma \in \R^{N \times N}$, which is
symmetric and positive semi-definite.

The goal is now to maximize the expected return of the portfolio, while the
variance of the portfolio does not exceed a certain value~$V_0$. To ensure that
only a reasonable amount of different assets is chosen, we introduce a
cardinality constraint according to \textcite{jin2016} of the form
\begin{equation*}
  ||y||_0 = \Card{\Defset{i \in [N]}{y_i > 0}} \leq k,
\end{equation*}
in which $||y||_0$ represents the number of entries of $y$ that are non-zero.
This number is bounded from above by a fixed number $k \in \N$ with $1 \leq k
\leq N$. Consequently, only up to $k$ assets can be included in the portfolio.
We can reformulate this constraint using binary variables~$s_i$ ($i\in[N]$) and obtain the
nominal cardinality-constrained portfolio~optimization~problem
\begin{subequations}
\label{eq:applications:portfolio:nominal-model}
  \begin{align}
    \max_{y, s} \quad &\mu^\top y
    \\
    \st \quad &y^\top \Sigma y \leq V_0,
    \quad \sum_{i \in [N]} y_i = 1, \label{eq:applications:portfolio:nominal-model:con1}
    \\
    &\sum_{i \in [N]} s_i \leq k,
    \quad y_i \leq s_i,
    \quad i \in [N], \label{eq:applications:portfolio:nominal-model:con2}
    \\
    &s \in \set{0,1}^N,
    \ y \geq 0. \label{eq:applications:portfolio:nominal-model:con3}
  \end{align}
\end{subequations}

\begin{remark}
  Note that Problem~\eqref{eq:applications:portfolio:nominal-model} is a
  mixed-integer quadratic optimization problem due to
  Constraint~\eqref{eq:applications:portfolio:nominal-model:con1}. Thus, the
  bilevel reformulation derived in
  Section~\ref{sec:general-model:bilevel-reformulation} becomes a mixed-integer
  quadratic bilevel problem, which cannot be handled by mixed-integer
  linear bilevel optimization solvers like
  \textnormal{\textsf{MibS}}. Consequently, we cannot present numerical results
  for this problem with discrete uncertainty sets in
  Section~\ref{sec:num-results}.
\end{remark}

\subsubsection{Budgeted Uncertainty}
\label{sec:applications:portfolio:budgeted-uncertainty}

We introduce budgeted uncertainties in the expected returns $\mu$ that depend on
a second vector of decision-variables~\mbox{$x \in [0,1]^N$}, which is, in
contrast to the previous sections, continuous. The variable~$x_i$ models the
decision whether and to what extent to hedge, i.e., to insure oneself, against
the uncertainty in the expected return of asset~$i$, e.g., by acquiring
additional information about the asset.

For this, we modify the objective function of
Problem~\eqref{eq:applications:portfolio:nominal-model} to include the costs for
hedging~$h^\top x$ as well as the uncertainties. We obtain the uncertain
portfolio selection problem
\begin{equation}
  \label{eq:applications:portfolio:cont-dd-rob-model}
  \begin{aligned}
    \max_{y,s,x} \quad &\bar{\mu}^\top y - h^\top x -
    \max_{u \in U_\text{p}(x)} \
    \sum_{i \in [N]} u_i \hat{\mu}_i y_i
    \\
    \st \quad & \eqref{eq:applications:portfolio:nominal-model:con1}\text{--}
    \eqref{eq:applications:portfolio:nominal-model:con3},
  \end{aligned}
\end{equation}
with the continuous decision-dependent budgeted uncertainty set $U_\text{p} =
U^{\text{cb}}$ defined in
Section~\ref{sec:uncertainty-sets:budgeted}.
If we choose $x_i \in (0,1]$, we
hedge against the uncertainty in the expected returns for asset~$i$. For
$x_i=0$, we do not hedge against the uncertainties for asset~$i$ at all.

\subsubsection{Knapsack Uncertainty}
\label{sec:applications:portfolio:knapsack-uncertainty}

We also consider decision-dependent uncertainties in the expected returns
$\mu_i$ in form of a continuous knapsack uncertainty set. In this case, the
uncertainty budget $b(s)$ grows with the number of chosen assets. We obtain the
uncertain portfolio selection problem
\begin{equation}
  \label{eq:applications:portfolio:cont-dd-rob-model:knapsack}
  \begin{aligned}
    \max_{y,s} \quad &\bar{\mu}^\top y -
    \max_{u \in U_\text{p}(s)} \
    \sum_{i \in [N]}   u_i \hat{\mu}_i y_i
    \\
    \st \quad & \eqref{eq:applications:portfolio:nominal-model:con1}\text{--}
    \eqref{eq:applications:portfolio:nominal-model:con3},
  \end{aligned}
\end{equation}
with the continuous decision-dependent knapsack uncertainty set $U_\text{p} =
U^{\text{ck}}$ defined in
Section~\ref{sec:uncertainty-sets:knapsack}.

\section{Computational Study}
\label{sec:num-results}

In this computational study, we consider the decision-dependent robust versions
of the shortest-path problem, the knapsack problem, and the portfolio selection
problem as described in the previous section. We start with continuous
uncertainty sets from Section~\ref{sec:num-results:cont}. The respective problems
can be tackled both by the classic (duality-based) robust approach stated in
Problem~\eqref{eq:general-model:classic-robust-reformulation} and by the
single-level reformulations of the bilevel reformulation stated in
Problems~\eqref{eq:general-model:bilevel-reformulation:sd:single-level}
and~\eqref{eq:general-model:bilevel-reformulation:kkt:single-level}.
We consider the continuous budgeted uncertainty set from
Section~\ref{sec:num-results:cont:budgeted} and the continuous knapsack
uncertainty set from Section~\ref{sec:num-results:cont:knapsack}. The goal is to
compare the computational performance of robust and bilevel approaches derived
from Section~\ref{sec:general-model} w.r.t.\ runtimes and branch-and-bound nodes.
We omit the results for the KKT-based single-level
reformulation~\eqref{eq:general-model:bilevel-reformulation:kkt:single-level}
because
preliminary experiments revealed that it is not competitive with the
duality-based approach. In the second part of the computational study, we consider
the discrete uncertainty sets, namely
the discrete budgeted uncertainty set from
Section~\ref{sec:num-results:discrete:budgeted} and the discrete knapsack
uncertainty set from Section~\ref{sec:num-results:discrete:knapsack}. Here, we
compare the results for the bilevel reformulation solved with the mixed-integer
linear bilevel solver \textsf{MibS} \parencite{mibs} and the results obtained
using the quantified integer programming solver~\textsf{Yasol}
\parencite{yasol2017}. \textsf{MibS} solves the bilevel
reformulation~\eqref{eq:general-model:bilevel-reformulation} while
\textsf{Yasol} solves the resulting quantified integer programs (QIPs).
For details on QIPs in the context of decision-dependent robust optimization and
the two different formulations considered in this computational study, we refer
to the Appendix~\ref{sec:qip-models}.

Our open-source implementation and all instances used in this computational
study are publicly available
at~\url{https://github.com/simstevens/ddro-via-bilevel}.

\subsection{Hardware and Software Setup}
\label{sec:num-results:setup}
All numerical experiments were conducted on a single core Intel
Xeon Gold 6126 at \SI{2.6}{\giga\hertz} with \SI{64}{\giga\byte} RAM and with
a time limit of \SI{2}{\hour}.

The single-level reformulations of the models with continuous uncertainty sets in
Section~\ref{sec:num-results:cont} were implemented in \textsf{Python}~3.12.2
using \textsf{Gurobi}~11.0.3 through its \textsf{gurobipy} interface.
The bilevel reformulations of the models with discrete uncertainty sets in
Section~\ref{sec:num-results:discrete} were solved with the mixed-integer linear
bilevel solver~\textsf{MibS}~1.2.2, which internally uses the mixed-integer
linear optimization solver \textsf{CPLEX}~22.1.1. We compare the results of the
bilevel approach with those obtained solving the QIP formulations with the
\textsf{Yasol} solver in version~4.0.1.5, which also uses \textsf{CPLEX}~22.1.1
as the underlying LP solver.

\subsection{Instance Generation}
\label{sec:num-results:instance-gen}

In the following, we describe the instance generation procedure for the three
considered applications. The data for the knapsack uncertainty sets is generated
with the knapsack instance generator~\textsf{gen2} from~\textcite{pisinger1999}
with coefficients in the range $[1, 100]$. The constant part of the uncertainty
budget $b$ is then given as the knapsack budget, $f$ is given as the knapsack
weights and $w$ is chosen uniformly at random from $[0, b/n]$ to ensure that
$b(x) \geq 0$ for all $x$. %

\rev{For the budgeted uncertainty sets, we identified application-specific
parameter configurations in preliminary experiments, under which hedging
decisions are typically made. The specific parameters are described in the
respective sections below. For more details on the impact of the choice of
hedging costs, we refer to Section~\ref{sec:num-results:hedging-costs}.}

\subsubsection{Shortest-Path Problem}
\label{sec:num-results:instance-gen:shortest-path}

The shortest-path instances are generated according to the procedure described
in \textcite{nohadani2018}. To this end, the number of nodes is taken as input.
For continuous uncertainty sets, we consider instances with a number of nodes
ranging from 75 to 300 with a step size of 25. For the discrete uncertainty
sets, we consider instances with a number of nodes ranging from \rev{2 to 20
with a step size of 2}. For each instance size, 10 instances are generated to
obtain 100~instances in total for each uncertainty set.

Initially, each node is uniformly assigned to a point on a $100\times 100$
grid to define travel costs between two nodes by their Euclidean distance. Then,
the source and the target nodes are chosen to be the two furthest nodes in the
respective complete graph. To avoid direct connections between the source and
the target, we then remove $60\,\%$ of the longest arcs in the cases of
continuous uncertainty sets. Hence, we obtain a graph $G = (V,A)$ with $|A| =
|V|\times(|V|-1)\times 0.4$ arcs. Since the graphs for the discrete uncertainty
sets are significantly smaller, we only remove $20\,\%$ of the longest arcs to
ensure that the graph is still connected. Finally, the maximum deviation of the
travel costs $\hat d_a$ is set to \rev{50\,\% of} the nominal value $\bar{d_a}$.

When considering Problem~\eqref{eq:applications:shortest-path:cont-dd-rob-model}
with budgeted uncertainty sets, the hedging costs are assumed to be the same for
all arcs and are fixed to $h_a = \rev{4}$. We use an uncertainty budget $\Gamma$
equal to \rev{1\,\% of the total number of arcs for the continuous case and
10\,\% for the discrete case}, while the fraction of uncertainty
reduction~$\gamma_a$ is set to $0.2$ for all arcs $a$. \rev{Note that, in
comparison to \textcite{nohadani2018}, we used different parameters than the
original ones to avoid a test set containing many instances without any
hedging decisions in the optimal solutions.}

The data for
Problem~\eqref{eq:applications:shortest-path:cont-dd-rob-model:knapsack} with a
knapsack uncertainty set is generated with the knapsack instance
generator~\textsf{gen2} as stated at the beginning of
Section~\ref{sec:num-results:instance-gen}. For the continuous version of the
uncertainty sets, we use strongly correlated weights and a capacity of $0.35\,
W_\text{sum}$, where $W_\text{sum}$ is the sum of all weights. For the discrete
version, we use uncorrelated weights and a capacity of $0.1\, W_\text{sum}$.

\subsubsection{Knapsack Problem}
\label{sec:num-results:instance-gen:knapsack}

The knapsack instances are generated again using the~\textsf{gen2} instance generator
with weights ranging from 1 to 100. In case of
continuous uncertainty sets, the capacity is set to $0.35\, W_\text{sum}$ and the
profits are strongly correlated with the weights. For discrete uncertainty sets,
we use uncorrelated weights and a capacity of $0.1\, W_\text{sum}$. The maximum
deviation of the weights $\hat{a}_i$ is set to $10$\,\% of the nominal value
$\bar{a}_i$ for all items~$i$.

In Figure~\ref{fig:num-results:instance-gen:knapsack:heatmap}, a heatmap putting
in relation the number of items, the knapsack's capacity, and the computation
time for the knapsack problems with a continuous knapsack uncertainty set is
presented. From these preliminary experiments, we see that both approaches
struggle to solve the same class of instances. In particular, instances with a
capacity which is around 50\,\% of the sum of the weights seem to be the most
challenging ones. This fact was already stated by \textcite{pisinger2005hard}
for general knapsack problems. Obviously, larger instances also tend to be more
difficult to solve than smaller ones. Thus, a capacity of $0.35\, W_\text{sum}$
yields harder instances than a capacity of $0.1\, W_\text{sum}$.

\begin{figure}
  \centering
  \makebox[\textwidth][c]{%
    \scalebox{1.0}{\input{plots/knapsack_cont_capacities.tex}}%
  } \caption{Heatmap of the computation times depending on the knapsack's
    capacity for the knapsack problem with a continuous knapsack uncertainty set.}
  \label{fig:num-results:instance-gen:knapsack:heatmap}
\end{figure}

When considering
Problem~\eqref{eq:applications:knapsack:budgeted-dd-rob-model:budgeted} with
budgeted uncertainty sets, the hedging costs~$h_i$ are \rev{set to $1/1000$} for
all items $i$. The uncertainty budget $\Gamma$ is set to $\rev{0.1} n$, while
the fraction of uncertainty reduction $\gamma_i$ is set to $0.2$ for all
items~$i$.

The data for the knapsack uncertainty set in
Problem~\eqref{eq:applications:knapsack:cont-dd-rob-model:knapsack} is generated
similarly to that for the shortest-path problem, as described above.

For the continuous \rev{knapsack} uncertainty set, we consider instances with a
number of items~$n$ ranging from $1000$ to $10\,000$ with a step size of $1000$,
\rev{while we consider instances with a number of items~$n$ ranging from $100$
to $1000$ with a step size of $100$ for the continuous budgeted uncertainty set.}
For the discrete \rev{knapsack} uncertainty set, we consider instances with a
number of items $n$ ranging from~$40$ to~$130$ with a step size of $10$ \rev{and
for the discrete budgeted uncertainty set the number of items ranges from $10$
to $100$.} For each instance size, 10~instances are generated to obtain
100~instances in total for each test set.

\subsubsection{Portfolio Selection Problem}
\label{sec:num-results:instance-gen:portfolio-selection}

For the portfolio selection instances, the covariance matrix~$\Sigma$ and the
vector of nominal expected returns~$\bar\mu$ are generated using the instance
generator from~\textcite{pardalos1990} with default settings. The maximum
variance~$V_0$ is uniformly drawn from~$[\sigma_{\text{mean}},
\sigma_{\text{max}}]$ with $\sigma_{\text{mean}}$ being the arithmetic mean of
all entries in~$\Sigma$ and $\sigma_{\text{max}}$ being the maximum entry
of~$\Sigma$. The cardinality parameter~$k$ is set to $\rev{50}$.
We consider instances with $N$ assets for $N \in \set{50, 100, \dotsc,
500}$. For each instance size, we generate 10~instances resulting in
100~instances for each uncertainty set.

The budgeted uncertainty sets for
Problem~\eqref{eq:applications:portfolio:cont-dd-rob-model} are defined by
setting $\Gamma$ to \rev{10\,\%} of the number of assets and $\gamma_i = 0.2$
for all assets~$i$. The maximum return deviation~$\hat{\mu}_i$ \rev{is set to
10\,\% of the nominal value $\bar{\mu}_i$} while the hedging cost~$h_i$ is
\rev{set to $1/1000$} for all assets~$i$.

The data for the knapsack uncertainty set in
Problem~\eqref{eq:applications:portfolio:cont-dd-rob-model:knapsack} is
again generated with~\textsf{gen2} with strongly correlated weights
and a capacity of $0.35\, W_\text{sum}$.

Note that we cannot present results for the
portfolio selection problem with discrete uncertainty sets since neither
\textsf{MibS} nor \textsf{Yasol} are able to solve problems with quadratic
constraints.

\rev{%
\subsection{Impact of Hedging Costs}
\label{sec:num-results:hedging-costs}
For the budgeted uncertainty sets, the choice of the hedging costs $h$ is
essential. Indeed, a too small (resp.\ too high) cost would make the hedging
decision trivially equal to $1$ (resp.\ $0$). We therefore present some
illustrative results on the shortest-path problem with continuous budgeted
uncertainty sets to demonstrate how the choice of $h$ affects both solutions and
computation times.\footnote{The experiments in
this subsection were conducted on a different machine than the remaining
experiments, namely an Intel Xeon E5-2640 v3 with \SI{2.6}{\giga\hertz}. All other
settings are as described in Section~\ref{sec:num-results:setup}. Because these
are mainly qualitative results, this difference in the hardware setup
does not affect comparability.} To do so, we consider the instances from
Section~\ref{sec:num-results:instance-gen:shortest-path} with varying hedging
costs $h \in \{1, 2, 3, 4, 5\}$.
\begin{figure}
  \centering
  \makebox[\textwidth][c]{%
    \scalebox{1.0}{\input{plots/optl_rev/sp_hedging_solutions_scatter.tex}}%
  } \caption{Scatter plot for the shortest-path problem with continuous budgeted
  uncertainty sets. The $x$-axis shows the hedging costs~$h$ and the $y$-axis shows
  the fraction of arcs in the optimal shortest path that are hedged.}
  \label{fig:num-results:instance-gen:hedging:scatter}
\end{figure}
Figure~\ref{fig:num-results:instance-gen:hedging:scatter} shows the fraction of
arcs in the optimal shortest path that are hedged for different hedging costs
$h$, i.e.,
\begin{equation*}
\frac{\text{\# arcs with hedging}}{\text{\# arcs in shortest path}}.
\end{equation*}
We observe that the number of hedging decisions decreases with increasing
hedging costs. For $h=1$, almost all arcs in the optimal shortest path are
hedged, while for $h=5$, no arcs are hedged at all. We illustrate the
corresponding impact on the optimal objective function values in
Figure~\ref{fig:num-results:hedging:objective:scatter} in
Appendix~\ref{sec:appendix:plots}. We see that the optimal objective function
value deteriorates with increasing hedging costs because fewer arcs are hedged
and thus the uncertainty set becomes larger; see also \textcite{nohadani2018}.
}

\rev{
\begin{figure}
  \centering
  \makebox[\textwidth][c]{%
    \scalebox{1.0}{\begin{tikzpicture}[x=1pt,y=1pt]
\definecolor{fillColor}{RGB}{255,255,255}
\path[use as bounding box,fill=fillColor,fill opacity=0.00] (35,0) rectangle (108.54,144.54);
\begin{scope}
\path[clip] ( 38.33, 32.16) rectangle (139.54,139.54);
\definecolor{drawColor}{gray}{0.92}

\path[draw=drawColor,line width= 0.3pt,line join=round] ( 38.33, 49.25) --
	(139.54, 49.25);

\path[draw=drawColor,line width= 0.3pt,line join=round] ( 38.33, 73.65) --
	(139.54, 73.65);

\path[draw=drawColor,line width= 0.3pt,line join=round] ( 38.33, 98.05) --
	(139.54, 98.05);

\path[draw=drawColor,line width= 0.3pt,line join=round] ( 38.33,122.46) --
	(139.54,122.46);

\path[draw=drawColor,line width= 0.3pt,line join=round] (119.60, 32.16) --
	(119.60,139.54);

\path[draw=drawColor,line width= 0.3pt,line join=round] ( 94.05, 32.16) --
	( 94.05,139.54);

\path[draw=drawColor,line width= 0.3pt,line join=round] ( 68.49, 32.16) --
	( 68.49,139.54);

\path[draw=drawColor,line width= 0.5pt,line join=round] ( 38.33, 37.04) --
	(139.54, 37.04);

\path[draw=drawColor,line width= 0.5pt,line join=round] ( 38.33, 61.45) --
	(139.54, 61.45);

\path[draw=drawColor,line width= 0.5pt,line join=round] ( 38.33, 85.85) --
	(139.54, 85.85);

\path[draw=drawColor,line width= 0.5pt,line join=round] ( 38.33,110.26) --
	(139.54,110.26);

\path[draw=drawColor,line width= 0.5pt,line join=round] ( 38.33,134.66) --
	(139.54,134.66);

\path[draw=drawColor,line width= 0.5pt,line join=round] (132.38, 32.16) --
	(132.38,139.54);

\path[draw=drawColor,line width= 0.5pt,line join=round] (106.82, 32.16) --
	(106.82,139.54);

\path[draw=drawColor,line width= 0.5pt,line join=round] ( 81.27, 32.16) --
	( 81.27,139.54);

\path[draw=drawColor,line width= 0.5pt,line join=round] ( 55.71, 32.16) --
	( 55.71,139.54);
\definecolor{drawColor}{RGB}{69,117,180}

\path[draw=drawColor,line width= 1.4pt,line join=round] ( 43.26, 38.02) --
	( 43.39, 38.02) --
	( 43.39, 39.00) --
	( 43.45, 39.00) --
	( 43.45, 39.97) --
	( 43.80, 39.97) --
	( 43.80, 40.95) --
	( 43.86, 40.95) --
	( 43.86, 41.92) --
	( 44.13, 41.92) --
	( 44.13, 42.90) --
	( 44.14, 42.90) --
	( 44.14, 43.88) --
	( 44.33, 43.88) --
	( 44.33, 44.85) --
	( 44.51, 44.85) --
	( 44.51, 45.83) --
	( 44.70, 45.83) --
	( 44.70, 46.80) --
	( 44.99, 46.80) --
	( 44.99, 47.78) --
	( 45.94, 47.78) --
	( 45.94, 48.76) --
	( 47.54, 48.76) --
	( 47.54, 49.73) --
	( 48.27, 49.73) --
	( 48.27, 50.71) --
	( 48.32, 50.71) --
	( 48.32, 51.69) --
	( 48.74, 51.69) --
	( 48.74, 52.66) --
	( 50.14, 52.66) --
	( 50.14, 53.64) --
	( 51.35, 53.64) --
	( 51.35, 54.61) --
	( 52.89, 54.61) --
	( 52.89, 55.59) --
	( 53.33, 55.59) --
	( 53.33, 56.57) --
	( 53.36, 56.57) --
	( 53.36, 57.54) --
	( 54.84, 57.54) --
	( 54.84, 58.52) --
	( 56.43, 58.52) --
	( 56.43, 59.49) --
	( 58.33, 59.49) --
	( 58.33, 60.47) --
	( 58.90, 60.47) --
	( 58.90, 61.45) --
	( 61.38, 61.45) --
	( 61.38, 62.42) --
	( 63.59, 62.42) --
	( 63.59, 63.40) --
	( 64.51, 63.40) --
	( 64.51, 64.38) --
	( 65.81, 64.38) --
	( 65.81, 65.35) --
	( 69.97, 65.35) --
	( 69.97, 66.33) --
	( 71.91, 66.33) --
	( 71.91, 67.30) --
	( 72.04, 67.30) --
	( 72.04, 68.28) --
	( 78.66, 68.28) --
	( 78.66, 69.26) --
	( 91.34, 69.26) --
	( 91.34, 70.23) --
	( 99.83, 70.23) --
	( 99.83, 71.21) --
	(105.28, 71.21) --
	(105.28, 72.18) --
	(120.98, 72.18) --
	(120.98, 73.16) --
	(134.94, 73.16) --
	(134.94, 73.16);
\definecolor{drawColor}{RGB}{252,141,89}

\path[draw=drawColor,line width= 1.4pt,dash pattern=on 2pt off 2pt ,line join=round] ( 43.00, 38.02) --
	( 43.00, 38.02) --
	( 43.00, 39.00) --
	( 43.02, 39.00) --
	( 43.02, 39.97) --
	( 43.12, 39.97) --
	( 43.12, 40.95) --
	( 43.17, 40.95) --
	( 43.17, 41.92) --
	( 43.20, 41.92) --
	( 43.20, 42.90) --
	( 43.29, 42.90) --
	( 43.29, 43.88) --
	( 43.30, 43.88) --
	( 43.30, 44.85) --
	( 43.32, 44.85) --
	( 43.32, 45.83) --
	( 43.37, 45.83) --
	( 43.37, 46.80) --
	( 43.49, 46.80) --
	( 43.49, 47.78) --
	( 43.49, 47.78) --
	( 43.49, 48.76) --
	( 43.60, 48.76) --
	( 43.60, 49.73) --
	( 43.74, 49.73) --
	( 43.74, 50.71) --
	( 43.91, 50.71) --
	( 43.91, 51.69) --
	( 43.96, 51.69) --
	( 43.96, 52.66) --
	( 44.05, 52.66) --
	( 44.05, 53.64) --
	( 44.06, 53.64) --
	( 44.06, 54.61) --
	( 44.30, 54.61) --
	( 44.30, 55.59) --
	( 44.39, 55.59) --
	( 44.39, 56.57) --
	( 44.41, 56.57) --
	( 44.41, 57.54) --
	( 45.06, 57.54) --
	( 45.06, 58.52) --
	( 45.07, 58.52) --
	( 45.07, 59.49) --
	( 45.09, 59.49) --
	( 45.09, 60.47) --
	( 45.42, 60.47) --
	( 45.42, 61.45) --
	( 46.75, 61.45) --
	( 46.75, 62.42) --
	( 46.95, 62.42) --
	( 46.95, 63.40) --
	( 47.53, 63.40) --
	( 47.53, 64.38) --
	( 47.58, 64.38) --
	( 47.58, 65.35) --
	( 47.71, 65.35) --
	( 47.71, 66.33) --
	( 47.94, 66.33) --
	( 47.94, 67.30) --
	( 49.21, 67.30) --
	( 49.21, 68.28) --
	( 51.58, 68.28) --
	( 51.58, 69.26) --
	( 56.12, 69.26) --
	( 56.12, 70.23) --
	( 57.93, 70.23) --
	( 57.93, 71.21) --
	( 58.73, 71.21) --
	( 58.73, 72.18) --
	( 62.02, 72.18) --
	( 62.02, 73.16) --
	( 73.97, 73.16) --
	( 73.97, 74.14) --
	( 75.11, 74.14) --
	( 75.11, 75.11) --
	( 85.02, 75.11) --
	( 85.02, 76.09) --
	( 95.42, 76.09) --
	( 95.42, 77.07) --
	( 97.32, 77.07) --
	( 97.32, 78.04) --
	( 97.83, 78.04) --
	( 97.83, 79.02) --
	(100.39, 79.02) --
	(100.39, 79.99) --
	(107.41, 79.99) --
	(107.41, 80.97) --
	(107.63, 80.97) --
	(107.63, 81.95) --
	(110.33, 81.95) --
	(110.33, 82.92) --
	(110.99, 82.92) --
	(110.99, 83.90) --
	(134.94, 83.90) --
	(134.94, 83.90);
\definecolor{drawColor}{RGB}{69,117,180}

\node[text=drawColor,anchor=base west,inner sep=0pt, outer sep=0pt, scale=  1.00] at ( 42.93,120.94) {37/100};
\definecolor{drawColor}{RGB}{252,141,89}

\node[text=drawColor,anchor=base west,inner sep=0pt, outer sep=0pt, scale=  1.00] at ( 42.93,110.61) {48/100};
\end{scope}
\begin{scope}
\path[clip] (  0.00,  0.00) rectangle (144.54,144.54);
\definecolor{drawColor}{gray}{0.30}

\node[text=drawColor,anchor=base east,inner sep=0pt, outer sep=0pt, scale=  1.00] at ( 33.83, 33.60) {0};

\node[text=drawColor,anchor=base east,inner sep=0pt, outer sep=0pt, scale=  1.00] at ( 33.83, 58.00) {25};

\node[text=drawColor,anchor=base east,inner sep=0pt, outer sep=0pt, scale=  1.00] at ( 33.83, 82.41) {50};

\node[text=drawColor,anchor=base east,inner sep=0pt, outer sep=0pt, scale=  1.00] at ( 33.83,106.81) {75};

\node[text=drawColor,anchor=base east,inner sep=0pt, outer sep=0pt, scale=  1.00] at ( 33.83,131.22) {100};
\end{scope}
\begin{scope}
\path[clip] (  0.00,  0.00) rectangle (144.54,144.54);
\definecolor{drawColor}{gray}{0.30}

\node[text=drawColor,anchor=base,inner sep=0pt, outer sep=0pt, scale=  1.00] at (132.38, 20.78) {7000};

\node[text=drawColor,anchor=base,inner sep=0pt, outer sep=0pt, scale=  1.00] at (106.82, 20.78) {5000};

\node[text=drawColor,anchor=base,inner sep=0pt, outer sep=0pt, scale=  1.00] at ( 81.27, 20.78) {3000};

\node[text=drawColor,anchor=base,inner sep=0pt, outer sep=0pt, scale=  1.00] at ( 55.71, 20.78) {1000};
\end{scope}
\begin{scope}
\path[clip] (  0.00,  0.00) rectangle (144.54,144.54);
\definecolor{drawColor}{RGB}{0,0,0}

\node[text=drawColor,rotate= 90.00,anchor=base,inner sep=0pt, outer sep=0pt, scale=  1.00] at ( 15, 85.85) {\# of solved instances};
\end{scope}
\end{tikzpicture}}%
    \scalebox{1.0}{\begin{tikzpicture}[x=1pt,y=1pt]
\definecolor{fillColor}{RGB}{255,255,255}
\path[use as bounding box,fill=fillColor,fill opacity=0.00] (0,0) rectangle (108.54,144.54);
\begin{scope}
\path[clip] ( 38.33, 32.16) rectangle (139.54,139.54);
\definecolor{drawColor}{gray}{0.92}

\path[draw=drawColor,line width= 0.3pt,line join=round] ( 38.33, 49.25) --
	(139.54, 49.25);

\path[draw=drawColor,line width= 0.3pt,line join=round] ( 38.33, 73.65) --
	(139.54, 73.65);

\path[draw=drawColor,line width= 0.3pt,line join=round] ( 38.33, 98.05) --
	(139.54, 98.05);

\path[draw=drawColor,line width= 0.3pt,line join=round] ( 38.33,122.46) --
	(139.54,122.46);

\path[draw=drawColor,line width= 0.3pt,line join=round] (119.60, 32.16) --
	(119.60,139.54);

\path[draw=drawColor,line width= 0.3pt,line join=round] ( 94.05, 32.16) --
	( 94.05,139.54);

\path[draw=drawColor,line width= 0.3pt,line join=round] ( 68.49, 32.16) --
	( 68.49,139.54);

\path[draw=drawColor,line width= 0.5pt,line join=round] ( 38.33, 37.04) --
	(139.54, 37.04);

\path[draw=drawColor,line width= 0.5pt,line join=round] ( 38.33, 61.45) --
	(139.54, 61.45);

\path[draw=drawColor,line width= 0.5pt,line join=round] ( 38.33, 85.85) --
	(139.54, 85.85);

\path[draw=drawColor,line width= 0.5pt,line join=round] ( 38.33,110.26) --
	(139.54,110.26);

\path[draw=drawColor,line width= 0.5pt,line join=round] ( 38.33,134.66) --
	(139.54,134.66);

\path[draw=drawColor,line width= 0.5pt,line join=round] (132.38, 32.16) --
	(132.38,139.54);

\path[draw=drawColor,line width= 0.5pt,line join=round] (106.82, 32.16) --
	(106.82,139.54);

\path[draw=drawColor,line width= 0.5pt,line join=round] ( 81.27, 32.16) --
	( 81.27,139.54);

\path[draw=drawColor,line width= 0.5pt,line join=round] ( 55.71, 32.16) --
	( 55.71,139.54);
\definecolor{drawColor}{RGB}{69,117,180}

\path[draw=drawColor,line width= 1.4pt,line join=round] ( 43.28, 38.02) --
	( 43.45, 38.02) --
	( 43.45, 39.00) --
	( 43.48, 39.00) --
	( 43.48, 39.97) --
	( 44.13, 39.97) --
	( 44.13, 40.95) --
	( 44.15, 40.95) --
	( 44.15, 41.92) --
	( 44.63, 41.92) --
	( 44.63, 42.90) --
	( 44.73, 42.90) --
	( 44.73, 43.88) --
	( 44.93, 43.88) --
	( 44.93, 44.85) --
	( 45.14, 44.85) --
	( 45.14, 45.83) --
	( 45.15, 45.83) --
	( 45.15, 46.80) --
	( 45.43, 46.80) --
	( 45.43, 47.78) --
	( 47.96, 47.78) --
	( 47.96, 48.76) --
	( 48.96, 48.76) --
	( 48.96, 49.73) --
	( 49.54, 49.73) --
	( 49.54, 50.71) --
	( 51.82, 50.71) --
	( 51.82, 51.69) --
	( 51.85, 51.69) --
	( 51.85, 52.66) --
	( 55.87, 52.66) --
	( 55.87, 53.64) --
	( 55.95, 53.64) --
	( 55.95, 54.61) --
	( 56.22, 54.61) --
	( 56.22, 55.59) --
	( 57.14, 55.59) --
	( 57.14, 56.57) --
	( 57.88, 56.57) --
	( 57.88, 57.54) --
	( 57.98, 57.54) --
	( 57.98, 58.52) --
	( 58.86, 58.52) --
	( 58.86, 59.49) --
	( 65.73, 59.49) --
	( 65.73, 60.47) --
	( 71.00, 60.47) --
	( 71.00, 61.45) --
	( 80.48, 61.45) --
	( 80.48, 62.42) --
	( 88.67, 62.42) --
	( 88.67, 63.40) --
	( 89.87, 63.40) --
	( 89.87, 64.38) --
	( 89.89, 64.38) --
	( 89.89, 65.35) --
	( 90.60, 65.35) --
	( 90.60, 66.33) --
	( 91.82, 66.33) --
	( 91.82, 67.30) --
	(104.83, 67.30) --
	(104.83, 68.28) --
	(105.08, 68.28) --
	(105.08, 69.26) --
	(113.32, 69.26) --
	(113.32, 70.23) --
	(115.88, 70.23) --
	(115.88, 71.21) --
	(134.94, 71.21) --
	(134.94, 71.21);
\definecolor{drawColor}{RGB}{252,141,89}

\path[draw=drawColor,line width= 1.4pt,dash pattern=on 2pt off 2pt ,line join=round] ( 43.01, 38.02) --
	( 43.03, 38.02) --
	( 43.03, 39.00) --
	( 43.12, 39.00) --
	( 43.12, 39.97) --
	( 43.13, 39.97) --
	( 43.13, 40.95) --
	( 43.13, 40.95) --
	( 43.13, 41.92) --
	( 43.20, 41.92) --
	( 43.20, 42.90) --
	( 43.22, 42.90) --
	( 43.22, 43.88) --
	( 43.24, 43.88) --
	( 43.24, 44.85) --
	( 43.38, 44.85) --
	( 43.38, 45.83) --
	( 43.39, 45.83) --
	( 43.39, 46.80) --
	( 43.45, 46.80) --
	( 43.45, 47.78) --
	( 43.49, 47.78) --
	( 43.49, 48.76) --
	( 43.57, 48.76) --
	( 43.57, 49.73) --
	( 43.74, 49.73) --
	( 43.74, 50.71) --
	( 44.04, 50.71) --
	( 44.04, 51.69) --
	( 44.12, 51.69) --
	( 44.12, 52.66) --
	( 44.16, 52.66) --
	( 44.16, 53.64) --
	( 44.29, 53.64) --
	( 44.29, 54.61) --
	( 44.66, 54.61) --
	( 44.66, 55.59) --
	( 45.04, 55.59) --
	( 45.04, 56.57) --
	( 45.27, 56.57) --
	( 45.27, 57.54) --
	( 45.35, 57.54) --
	( 45.35, 58.52) --
	( 45.42, 58.52) --
	( 45.42, 59.49) --
	( 46.55, 59.49) --
	( 46.55, 60.47) --
	( 46.78, 60.47) --
	( 46.78, 61.45) --
	( 46.91, 61.45) --
	( 46.91, 62.42) --
	( 47.45, 62.42) --
	( 47.45, 63.40) --
	( 48.01, 63.40) --
	( 48.01, 64.38) --
	( 48.17, 64.38) --
	( 48.17, 65.35) --
	( 50.03, 65.35) --
	( 50.03, 66.33) --
	( 50.49, 66.33) --
	( 50.49, 67.30) --
	( 52.43, 67.30) --
	( 52.43, 68.28) --
	( 52.76, 68.28) --
	( 52.76, 69.26) --
	( 54.90, 69.26) --
	( 54.90, 70.23) --
	( 56.56, 70.23) --
	( 56.56, 71.21) --
	( 56.75, 71.21) --
	( 56.75, 72.18) --
	( 58.52, 72.18) --
	( 58.52, 73.16) --
	( 61.38, 73.16) --
	( 61.38, 74.14) --
	( 65.61, 74.14) --
	( 65.61, 75.11) --
	( 66.59, 75.11) --
	( 66.59, 76.09) --
	( 70.21, 76.09) --
	( 70.21, 77.07) --
	( 77.91, 77.07) --
	( 77.91, 78.04) --
	( 87.60, 78.04) --
	( 87.60, 79.02) --
	( 89.63, 79.02) --
	( 89.63, 79.99) --
	( 98.90, 79.99) --
	( 98.90, 80.97) --
	( 99.74, 80.97) --
	( 99.74, 81.95) --
	(105.55, 81.95) --
	(105.55, 82.92) --
	(106.01, 82.92) --
	(106.01, 83.90) --
	(129.92, 83.90) --
	(129.92, 84.88) --
	(130.89, 84.88) --
	(130.89, 85.85) --
	(132.51, 85.85) --
	(132.51, 86.83) --
	(134.94, 86.83) --
	(134.94, 86.83);
\definecolor{drawColor}{RGB}{69,117,180}

\node[text=drawColor,anchor=base west,inner sep=0pt, outer sep=0pt, scale=  1.00] at ( 40,120) {35/100};
\definecolor{drawColor}{RGB}{252,141,89}

\node[text=drawColor,anchor=base west,inner sep=0pt, outer sep=0pt, scale=  1.00] at ( 40,110) {51/100};
\end{scope}
\begin{scope}
\path[clip] (  0.00,  0.00) rectangle (144.54,144.54);
\definecolor{drawColor}{gray}{0.30}

\node[text=drawColor,anchor=base,inner sep=0pt, outer sep=0pt, scale=  1.00] at (132.38, 20.78) {7000};

\node[text=drawColor,anchor=base,inner sep=0pt, outer sep=0pt, scale=  1.00] at (106.82, 20.78) {5000};

\node[text=drawColor,anchor=base,inner sep=0pt, outer sep=0pt, scale=  1.00] at ( 81.27, 20.78) {3000};

\node[text=drawColor,anchor=base,inner sep=0pt, outer sep=0pt, scale=  1.00] at ( 55.71, 20.78) {1000};
\end{scope}
\begin{scope}
\path[clip] (  0.00,  0.00) rectangle (144.54,144.54);
\definecolor{drawColor}{RGB}{0,0,0}

\node[text=drawColor,anchor=base,inner sep=0pt, outer sep=0pt, scale=  1.00] at ( 88.93,  6.94) {Computation time (s)};
\end{scope}
\end{tikzpicture}}%
    \scalebox{1.0}{\begin{tikzpicture}[x=1pt,y=1pt]
\definecolor{fillColor}{RGB}{255,255,255}
\path[use as bounding box,fill=fillColor,fill opacity=0.00] (0,0) rectangle (108.54,144.54);
\begin{scope}
\path[clip] ( 38.33, 32.16) rectangle (139.54,139.54);
\definecolor{drawColor}{gray}{0.92}

\path[draw=drawColor,line width= 0.3pt,line join=round] ( 38.33, 49.25) --
	(139.54, 49.25);

\path[draw=drawColor,line width= 0.3pt,line join=round] ( 38.33, 73.65) --
	(139.54, 73.65);

\path[draw=drawColor,line width= 0.3pt,line join=round] ( 38.33, 98.05) --
	(139.54, 98.05);

\path[draw=drawColor,line width= 0.3pt,line join=round] ( 38.33,122.46) --
	(139.54,122.46);

\path[draw=drawColor,line width= 0.3pt,line join=round] (119.60, 32.16) --
	(119.60,139.54);

\path[draw=drawColor,line width= 0.3pt,line join=round] ( 94.05, 32.16) --
	( 94.05,139.54);

\path[draw=drawColor,line width= 0.3pt,line join=round] ( 68.49, 32.16) --
	( 68.49,139.54);

\path[draw=drawColor,line width= 0.5pt,line join=round] ( 38.33, 37.04) --
	(139.54, 37.04);

\path[draw=drawColor,line width= 0.5pt,line join=round] ( 38.33, 61.45) --
	(139.54, 61.45);

\path[draw=drawColor,line width= 0.5pt,line join=round] ( 38.33, 85.85) --
	(139.54, 85.85);

\path[draw=drawColor,line width= 0.5pt,line join=round] ( 38.33,110.26) --
	(139.54,110.26);

\path[draw=drawColor,line width= 0.5pt,line join=round] ( 38.33,134.66) --
	(139.54,134.66);

\path[draw=drawColor,line width= 0.5pt,line join=round] (132.38, 32.16) --
	(132.38,139.54);

\path[draw=drawColor,line width= 0.5pt,line join=round] (106.82, 32.16) --
	(106.82,139.54);

\path[draw=drawColor,line width= 0.5pt,line join=round] ( 81.27, 32.16) --
	( 81.27,139.54);

\path[draw=drawColor,line width= 0.5pt,line join=round] ( 55.71, 32.16) --
	( 55.71,139.54);
\definecolor{drawColor}{RGB}{69,117,180}

\path[draw=drawColor,line width= 1.4pt,line join=round] ( 43.26, 38.02) --
	( 43.44, 38.02) --
	( 43.44, 39.00) --
	( 43.73, 39.00) --
	( 43.73, 39.97) --
	( 44.07, 39.97) --
	( 44.07, 40.95) --
	( 44.11, 40.95) --
	( 44.11, 41.92) --
	( 44.34, 41.92) --
	( 44.34, 42.90) --
	( 45.17, 42.90) --
	( 45.17, 43.88) --
	( 45.18, 43.88) --
	( 45.18, 44.85) --
	( 45.77, 44.85) --
	( 45.77, 45.83) --
	( 46.45, 45.83) --
	( 46.45, 46.80) --
	( 46.90, 46.80) --
	( 46.90, 47.78) --
	( 46.94, 47.78) --
	( 46.94, 48.76) --
	( 50.94, 48.76) --
	( 50.94, 49.73) --
	( 51.73, 49.73) --
	( 51.73, 50.71) --
	( 52.60, 50.71) --
	( 52.60, 51.69) --
	( 55.06, 51.69) --
	( 55.06, 52.66) --
	( 55.94, 52.66) --
	( 55.94, 53.64) --
	( 57.08, 53.64) --
	( 57.08, 54.61) --
	( 58.11, 54.61) --
	( 58.11, 55.59) --
	( 60.41, 55.59) --
	( 60.41, 56.57) --
	( 66.10, 56.57) --
	( 66.10, 57.54) --
	( 70.69, 57.54) --
	( 70.69, 58.52) --
	( 71.01, 58.52) --
	( 71.01, 59.49) --
	( 71.19, 59.49) --
	( 71.19, 60.47) --
	( 82.76, 60.47) --
	( 82.76, 61.45) --
	( 82.81, 61.45) --
	( 82.81, 62.42) --
	( 83.18, 62.42) --
	( 83.18, 63.40) --
	( 84.50, 63.40) --
	( 84.50, 64.38) --
	( 86.55, 64.38) --
	( 86.55, 65.35) --
	( 91.05, 65.35) --
	( 91.05, 66.33) --
	( 99.17, 66.33) --
	( 99.17, 67.30) --
	(101.94, 67.30) --
	(101.94, 68.28) --
	(106.80, 68.28) --
	(106.80, 69.26) --
	(121.85, 69.26) --
	(121.85, 70.23) --
	(134.94, 70.23) --
	(134.94, 70.23);
\definecolor{drawColor}{RGB}{252,141,89}

\path[draw=drawColor,line width= 1.4pt,dash pattern=on 2pt off 2pt ,line join=round] ( 42.96, 38.02) --
	( 42.96, 38.02) --
	( 42.96, 39.00) --
	( 42.96, 39.00) --
	( 42.96, 39.97) --
	( 42.97, 39.97) --
	( 42.97, 40.95) --
	( 42.97, 40.95) --
	( 42.97, 41.92) --
	( 42.97, 41.92) --
	( 42.97, 42.90) --
	( 42.98, 42.90) --
	( 42.98, 43.88) --
	( 42.99, 43.88) --
	( 42.99, 44.85) --
	( 42.99, 44.85) --
	( 42.99, 45.83) --
	( 42.99, 45.83) --
	( 42.99, 46.80) --
	( 42.99, 46.80) --
	( 42.99, 47.78) --
	( 43.00, 47.78) --
	( 43.00, 48.76) --
	( 43.00, 48.76) --
	( 43.00, 49.73) --
	( 43.00, 49.73) --
	( 43.00, 50.71) --
	( 43.00, 50.71) --
	( 43.00, 51.69) --
	( 43.01, 51.69) --
	( 43.01, 52.66) --
	( 43.01, 52.66) --
	( 43.01, 53.64) --
	( 43.01, 53.64) --
	( 43.01, 54.61) --
	( 43.01, 54.61) --
	( 43.01, 55.59) --
	( 43.01, 55.59) --
	( 43.01, 56.57) --
	( 43.02, 56.57) --
	( 43.02, 57.54) --
	( 43.02, 57.54) --
	( 43.02, 58.52) --
	( 43.02, 58.52) --
	( 43.02, 59.49) --
	( 43.03, 59.49) --
	( 43.03, 60.47) --
	( 43.03, 60.47) --
	( 43.03, 61.45) --
	( 43.03, 61.45) --
	( 43.03, 62.42) --
	( 43.03, 62.42) --
	( 43.03, 63.40) --
	( 43.03, 63.40) --
	( 43.03, 64.38) --
	( 43.03, 64.38) --
	( 43.03, 65.35) --
	( 43.04, 65.35) --
	( 43.04, 66.33) --
	( 43.04, 66.33) --
	( 43.04, 67.30) --
	( 43.05, 67.30) --
	( 43.05, 68.28) --
	( 43.06, 68.28) --
	( 43.06, 69.26) --
	( 43.07, 69.26) --
	( 43.07, 70.23) --
	( 43.08, 70.23) --
	( 43.08, 71.21) --
	( 43.09, 71.21) --
	( 43.09, 72.18) --
	( 43.09, 72.18) --
	( 43.09, 73.16) --
	( 43.10, 73.16) --
	( 43.10, 74.14) --
	( 43.10, 74.14) --
	( 43.10, 75.11) --
	( 43.10, 75.11) --
	( 43.10, 76.09) --
	( 43.11, 76.09) --
	( 43.11, 77.07) --
	( 43.12, 77.07) --
	( 43.12, 78.04) --
	( 43.14, 78.04) --
	( 43.14, 79.02) --
	( 43.14, 79.02) --
	( 43.14, 79.99) --
	( 43.14, 79.99) --
	( 43.14, 80.97) --
	( 43.15, 80.97) --
	( 43.15, 81.95) --
	( 43.15, 81.95) --
	( 43.15, 82.92) --
	( 43.16, 82.92) --
	( 43.16, 83.90) --
	( 43.18, 83.90) --
	( 43.18, 84.88) --
	( 43.18, 84.88) --
	( 43.18, 85.85) --
	( 43.18, 85.85) --
	( 43.18, 86.83) --
	( 43.18, 86.83) --
	( 43.18, 87.80) --
	( 43.20, 87.80) --
	( 43.20, 88.78) --
	( 43.21, 88.78) --
	( 43.21, 89.76) --
	( 43.21, 89.76) --
	( 43.21, 90.73) --
	( 43.25, 90.73) --
	( 43.25, 91.71) --
	( 43.42, 91.71) --
	( 43.42, 92.68) --
	( 43.55, 92.68) --
	( 43.55, 93.66) --
	( 43.56, 93.66) --
	( 43.56, 94.64) --
	( 43.58, 94.64) --
	( 43.58, 95.61) --
	( 43.65, 95.61) --
	( 43.65, 96.59) --
	( 43.65, 96.59) --
	( 43.65, 97.57) --
	( 43.67, 97.57) --
	( 43.67, 98.54) --
	( 43.74, 98.54) --
	( 43.74, 99.52) --
	( 43.83, 99.52) --
	( 43.83,100.49) --
	( 43.85,100.49) --
	( 43.85,101.47) --
	( 43.86,101.47) --
	( 43.86,102.45) --
	( 43.95,102.45) --
	( 43.95,103.42) --
	( 43.97,103.42) --
	( 43.97,104.40) --
	( 43.98,104.40) --
	( 43.98,105.37) --
	( 44.03,105.37) --
	( 44.03,106.35) --
	( 44.68,106.35) --
	( 44.68,107.33) --
	( 45.61,107.33) --
	( 45.61,108.30) --
	( 45.80,108.30) --
	( 45.80,109.28) --
	( 58.37,109.28) --
	( 58.37,110.26) --
	( 67.37,110.26) --
	( 67.37,111.23) --
	( 75.73,111.23) --
	( 75.73,112.21) --
	( 82.14,112.21) --
	( 82.14,113.18) --
	( 85.02,113.18) --
	( 85.02,114.16) --
	(100.77,114.16) --
	(100.77,115.14) --
	(108.51,115.14) --
	(108.51,116.11) --
	(110.21,116.11) --
	(110.21,117.09) --
	(112.10,117.09) --
	(112.10,118.06) --
	(122.87,118.06) --
	(122.87,119.04) --
	(124.93,119.04) --
	(124.93,120.02) --
	(133.55,120.02) --
	(133.55,120.99) --
	(134.94,120.99) --
	(134.94,120.99);
\definecolor{drawColor}{RGB}{69,117,180}

\node[text=drawColor,anchor=base west,inner sep=0pt, outer sep=0pt, scale=  1.00] at ( 100,90) {34/100};
\definecolor{drawColor}{RGB}{252,141,89}

\node[text=drawColor,anchor=base west,inner sep=0pt, outer sep=0pt, scale=  1.00] at ( 100,80) {86/100};
\end{scope}
\begin{scope}
\path[clip] (  0.00,  0.00) rectangle (144.54,144.54);
\definecolor{drawColor}{gray}{0.30}

\node[text=drawColor,anchor=base,inner sep=0pt, outer sep=0pt, scale=  1.00] at (132.38, 20.78) {7000};

\node[text=drawColor,anchor=base,inner sep=0pt, outer sep=0pt, scale=  1.00] at (106.82, 20.78) {5000};

\node[text=drawColor,anchor=base,inner sep=0pt, outer sep=0pt, scale=  1.00] at ( 81.27, 20.78) {3000};

\node[text=drawColor,anchor=base,inner sep=0pt, outer sep=0pt, scale=  1.00] at ( 55.71, 20.78) {1000};
\end{scope}
\end{tikzpicture}}%
  } \caption{ECDF of computation time for the shortest-path problem with
  continuous budgeted uncertainty sets and hedging costs $h = 3$ (left), $h =
  4$ (middle), and $h = 5$ (right). The numbers in the legend indicate the
  number of instances solved within the time limit of 2 hours out of 100
  instances. Solid \blue{blue}: \blue{bilevel approach}. Dotted \orange{orange}:
  \orange{robust approach}.}
  \label{fig:num-results:instance-gen:hedging:ecdf}
\end{figure}
Figure~\ref{fig:num-results:instance-gen:hedging:ecdf} shows the empirical
cumulative distribution functions (ECDFs) of computation times under a time
limit of 2 hours for the shortest-path problem with continuous budgeted
uncertainty sets and hedging costs $h = 3$ (left plot), $h = 4$ (middle
plot), and $h = 5$ (right plot). We observe that the robust approach
outperforms the bilevel approach in
terms of computation time for all considered hedging costs. However, the
difference between the two approaches becomes more significant for $h=5$, where
the optimal solutions typically contain no hedging decisions. }

\subsection{Numerical Results for Continuous Uncertainty Sets}
\label{sec:num-results:cont}

In this section, we analyze the continuous versions of the budgeted and the
knapsack uncertainty sets for the shortest-path problem, the knapsack problem,
and the portfolio selection problem. We compare the classic robust
approach~\eqref{eq:general-model:classic-robust-reformulation} with
the bilevel approach using the strong-duality based single-level
reformulation~\eqref{eq:general-model:bilevel-reformulation:sd:single-level}.

\subsubsection{Budgeted Uncertainty}
\label{sec:num-results:cont:budgeted}
We first consider all three applications with the respective continuous
decision-dependent budgeted uncertainty set.  The complete derivation of the
reformulations can be found in
Appendix~\ref{sec:appendix:applications:shortest-path:budgeted} for the
shortest-path problem, in
Appendix~\ref{sec:appendix:applications:knapsack:budgeted} for the knapsack
problem, and in Appendix~\ref{sec:appendix:applications:portfolio:budgeted} for
the portfolio selection problem.

\begin{figure}
  \centering
  \makebox[\textwidth][c]{%
    \scalebox{1.0}{\begin{tikzpicture}[x=1pt,y=1pt]
\definecolor{fillColor}{RGB}{255,255,255}
\path[use as bounding box,fill=fillColor,fill opacity=0.00] (35,0) rectangle (108.54,144.54);
\begin{scope}
\path[clip] ( 38.33, 32.16) rectangle (139.54,139.54);
\definecolor{drawColor}{gray}{0.92}

\path[draw=drawColor,line width= 0.3pt,line join=round] ( 38.33, 49.25) --
	(139.54, 49.25);

\path[draw=drawColor,line width= 0.3pt,line join=round] ( 38.33, 73.65) --
	(139.54, 73.65);

\path[draw=drawColor,line width= 0.3pt,line join=round] ( 38.33, 98.05) --
	(139.54, 98.05);

\path[draw=drawColor,line width= 0.3pt,line join=round] ( 38.33,122.46) --
	(139.54,122.46);

\path[draw=drawColor,line width= 0.3pt,line join=round] (119.60, 32.16) --
	(119.60,139.54);

\path[draw=drawColor,line width= 0.3pt,line join=round] ( 94.05, 32.16) --
	( 94.05,139.54);

\path[draw=drawColor,line width= 0.3pt,line join=round] ( 68.49, 32.16) --
	( 68.49,139.54);

\path[draw=drawColor,line width= 0.5pt,line join=round] ( 38.33, 37.04) --
	(139.54, 37.04);

\path[draw=drawColor,line width= 0.5pt,line join=round] ( 38.33, 61.45) --
	(139.54, 61.45);

\path[draw=drawColor,line width= 0.5pt,line join=round] ( 38.33, 85.85) --
	(139.54, 85.85);

\path[draw=drawColor,line width= 0.5pt,line join=round] ( 38.33,110.26) --
	(139.54,110.26);

\path[draw=drawColor,line width= 0.5pt,line join=round] ( 38.33,134.66) --
	(139.54,134.66);

\path[draw=drawColor,line width= 0.5pt,line join=round] (132.38, 32.16) --
	(132.38,139.54);

\path[draw=drawColor,line width= 0.5pt,line join=round] (106.82, 32.16) --
	(106.82,139.54);

\path[draw=drawColor,line width= 0.5pt,line join=round] ( 81.27, 32.16) --
	( 81.27,139.54);

\path[draw=drawColor,line width= 0.5pt,line join=round] ( 55.71, 32.16) --
	( 55.71,139.54);
\definecolor{drawColor}{RGB}{69,117,180}

\path[draw=drawColor,line width= 1.4pt,line join=round] ( 43.20, 38.02) --
	( 43.33, 38.02) --
	( 43.33, 39.00) --
	( 43.36, 39.00) --
	( 43.36, 39.97) --
	( 43.91, 39.97) --
	( 43.91, 40.95) --
	( 43.92, 40.95) --
	( 43.92, 41.92) --
	( 44.30, 41.92) --
	( 44.30, 42.90) --
	( 44.37, 42.90) --
	( 44.37, 43.88) --
	( 44.57, 43.88) --
	( 44.57, 44.85) --
	( 44.63, 44.85) --
	( 44.63, 45.83) --
	( 44.74, 45.83) --
	( 44.74, 46.80) --
	( 45.04, 46.80) --
	( 45.04, 47.78) --
	( 47.09, 47.78) --
	( 47.09, 48.76) --
	( 48.42, 48.76) --
	( 48.42, 49.73) --
	( 48.76, 49.73) --
	( 48.76, 50.71) --
	( 49.98, 50.71) --
	( 49.98, 51.69) --
	( 50.33, 51.69) --
	( 50.33, 52.66) --
	( 53.64, 52.66) --
	( 53.64, 53.64) --
	( 53.78, 53.64) --
	( 53.78, 54.61) --
	( 54.12, 54.61) --
	( 54.12, 55.59) --
	( 54.57, 55.59) --
	( 54.57, 56.57) --
	( 55.36, 56.57) --
	( 55.36, 57.54) --
	( 55.51, 57.54) --
	( 55.51, 58.52) --
	( 56.29, 58.52) --
	( 56.29, 59.49) --
	( 62.05, 59.49) --
	( 62.05, 60.47) --
	( 66.59, 60.47) --
	( 66.59, 61.45) --
	( 75.41, 61.45) --
	( 75.41, 62.42) --
	( 82.76, 62.42) --
	( 82.76, 63.40) --
	( 83.15, 63.40) --
	( 83.15, 64.38) --
	( 83.57, 64.38) --
	( 83.57, 65.35) --
	( 84.76, 65.35) --
	( 84.76, 66.33) --
	( 84.97, 66.33) --
	( 84.97, 67.30) --
	( 94.31, 67.30) --
	( 94.31, 68.28) --
	( 94.80, 68.28) --
	( 94.80, 69.26) --
	(105.00, 69.26) --
	(105.00, 70.23) --
	(107.13, 70.23) --
	(107.13, 71.21) --
	(134.94, 71.21) --
	(134.94, 71.21);
\definecolor{drawColor}{RGB}{252,141,89}

\path[draw=drawColor,line width= 1.4pt,dash pattern=on 2pt off 2pt ,line join=round] ( 42.99, 38.02) --
	( 43.02, 38.02) --
	( 43.02, 39.00) --
	( 43.09, 39.00) --
	( 43.09, 39.97) --
	( 43.09, 39.97) --
	( 43.09, 40.95) --
	( 43.09, 40.95) --
	( 43.09, 41.92) --
	( 43.15, 41.92) --
	( 43.15, 42.90) --
	( 43.18, 42.90) --
	( 43.18, 43.88) --
	( 43.18, 43.88) --
	( 43.18, 44.85) --
	( 43.30, 44.85) --
	( 43.30, 45.83) --
	( 43.31, 45.83) --
	( 43.31, 46.80) --
	( 43.35, 46.80) --
	( 43.35, 47.78) --
	( 43.38, 47.78) --
	( 43.38, 48.76) --
	( 43.45, 48.76) --
	( 43.45, 49.73) --
	( 43.60, 49.73) --
	( 43.60, 50.71) --
	( 43.83, 50.71) --
	( 43.83, 51.69) --
	( 43.91, 51.69) --
	( 43.91, 52.66) --
	( 44.04, 52.66) --
	( 44.04, 53.64) --
	( 44.05, 53.64) --
	( 44.05, 54.61) --
	( 44.32, 54.61) --
	( 44.32, 55.59) --
	( 44.56, 55.59) --
	( 44.56, 56.57) --
	( 44.83, 56.57) --
	( 44.83, 57.54) --
	( 44.90, 57.54) --
	( 44.90, 58.52) --
	( 44.97, 58.52) --
	( 44.97, 59.49) --
	( 45.85, 59.49) --
	( 45.85, 60.47) --
	( 46.04, 60.47) --
	( 46.04, 61.45) --
	( 46.10, 61.45) --
	( 46.10, 62.42) --
	( 46.65, 62.42) --
	( 46.65, 63.40) --
	( 47.08, 63.40) --
	( 47.08, 64.38) --
	( 47.19, 64.38) --
	( 47.19, 65.35) --
	( 48.73, 65.35) --
	( 48.73, 66.33) --
	( 49.03, 66.33) --
	( 49.03, 67.30) --
	( 50.73, 67.30) --
	( 50.73, 68.28) --
	( 50.97, 68.28) --
	( 50.97, 69.26) --
	( 52.77, 69.26) --
	( 52.77, 70.23) --
	( 54.14, 70.23) --
	( 54.14, 71.21) --
	( 54.16, 71.21) --
	( 54.16, 72.18) --
	( 55.74, 72.18) --
	( 55.74, 73.16) --
	( 58.23, 73.16) --
	( 58.23, 74.14) --
	( 62.17, 74.14) --
	( 62.17, 75.11) --
	( 62.96, 75.11) --
	( 62.96, 76.09) --
	( 65.41, 76.09) --
	( 65.41, 77.07) --
	( 73.36, 77.07) --
	( 73.36, 78.04) --
	( 82.20, 78.04) --
	( 82.20, 79.02) --
	( 83.33, 79.02) --
	( 83.33, 79.99) --
	( 89.58, 79.99) --
	( 89.58, 80.97) --
	( 93.06, 80.97) --
	( 93.06, 81.95) --
	( 95.83, 81.95) --
	( 95.83, 82.92) --
	( 99.04, 82.92) --
	( 99.04, 83.90) --
	(117.98, 83.90) --
	(117.98, 84.88) --
	(119.22, 84.88) --
	(119.22, 85.85) --
	(120.18, 85.85) --
	(120.18, 86.83) --
	(128.44, 86.83) --
	(128.44, 87.80) --
	(134.94, 87.80) --
	(134.94, 87.80);
\definecolor{drawColor}{RGB}{69,117,180}

\node[text=drawColor,anchor=base west,inner sep=0pt, outer sep=0pt, scale=  1.00] at ( 40,120) {35/100};
\definecolor{drawColor}{RGB}{252,141,89}

\node[text=drawColor,anchor=base west,inner sep=0pt, outer sep=0pt, scale=  1.00] at ( 40,110) {52/100};
\end{scope}
\begin{scope}
\path[clip] (  0.00,  0.00) rectangle (144.54,144.54);
\definecolor{drawColor}{gray}{0.30}

\node[text=drawColor,anchor=base east,inner sep=0pt, outer sep=0pt, scale=  1.00] at ( 33.83, 33.60) {0};

\node[text=drawColor,anchor=base east,inner sep=0pt, outer sep=0pt, scale=  1.00] at ( 33.83, 58.00) {25};

\node[text=drawColor,anchor=base east,inner sep=0pt, outer sep=0pt, scale=  1.00] at ( 33.83, 82.41) {50};

\node[text=drawColor,anchor=base east,inner sep=0pt, outer sep=0pt, scale=  1.00] at ( 33.83,106.81) {75};

\node[text=drawColor,anchor=base east,inner sep=0pt, outer sep=0pt, scale=  1.00] at ( 33.83,131.22) {100};
\end{scope}
\begin{scope}
\path[clip] (  0.00,  0.00) rectangle (144.54,144.54);
\definecolor{drawColor}{gray}{0.30}

\node[text=drawColor,anchor=base,inner sep=0pt, outer sep=0pt, scale=  1.00] at (132.38, 20.78) {7000};

\node[text=drawColor,anchor=base,inner sep=0pt, outer sep=0pt, scale=  1.00] at (106.82, 20.78) {5000};

\node[text=drawColor,anchor=base,inner sep=0pt, outer sep=0pt, scale=  1.00] at ( 81.27, 20.78) {3000};

\node[text=drawColor,anchor=base,inner sep=0pt, outer sep=0pt, scale=  1.00] at ( 55.71, 20.78) {1000};
\end{scope}
\begin{scope}
\path[clip] (  0.00,  0.00) rectangle (144.54,144.54);
\definecolor{drawColor}{RGB}{0,0,0}

\end{scope}
\begin{scope}
\path[clip] (  0.00,  0.00) rectangle (144.54,144.54);
\definecolor{drawColor}{RGB}{0,0,0}

\node[text=drawColor,rotate= 90.00,anchor=base,inner sep=0pt, outer sep=0pt, scale=  1.00] at ( 15, 85.85) {\# of solved instances};
\end{scope}
\end{tikzpicture}}%
    \scalebox{1.0}{\begin{tikzpicture}[x=1pt,y=1pt]
\definecolor{fillColor}{RGB}{255,255,255}
\path[use as bounding box,fill=fillColor,fill opacity=0.00] (0,0) rectangle (108.54,144.54);
\begin{scope}
\path[clip] ( 38.33, 32.16) rectangle (139.54,139.54);
\definecolor{drawColor}{gray}{0.92}

\path[draw=drawColor,line width= 0.3pt,line join=round] ( 38.33, 49.25) --
	(139.54, 49.25);

\path[draw=drawColor,line width= 0.3pt,line join=round] ( 38.33, 73.65) --
	(139.54, 73.65);

\path[draw=drawColor,line width= 0.3pt,line join=round] ( 38.33, 98.05) --
	(139.54, 98.05);

\path[draw=drawColor,line width= 0.3pt,line join=round] ( 38.33,122.46) --
	(139.54,122.46);

\path[draw=drawColor,line width= 0.3pt,line join=round] (119.60, 32.16) --
	(119.60,139.54);

\path[draw=drawColor,line width= 0.3pt,line join=round] ( 94.05, 32.16) --
	( 94.05,139.54);

\path[draw=drawColor,line width= 0.3pt,line join=round] ( 68.49, 32.16) --
	( 68.49,139.54);

\path[draw=drawColor,line width= 0.5pt,line join=round] ( 38.33, 37.04) --
	(139.54, 37.04);

\path[draw=drawColor,line width= 0.5pt,line join=round] ( 38.33, 61.45) --
	(139.54, 61.45);

\path[draw=drawColor,line width= 0.5pt,line join=round] ( 38.33, 85.85) --
	(139.54, 85.85);

\path[draw=drawColor,line width= 0.5pt,line join=round] ( 38.33,110.26) --
	(139.54,110.26);

\path[draw=drawColor,line width= 0.5pt,line join=round] ( 38.33,134.66) --
	(139.54,134.66);

\path[draw=drawColor,line width= 0.5pt,line join=round] (132.38, 32.16) --
	(132.38,139.54);

\path[draw=drawColor,line width= 0.5pt,line join=round] (106.82, 32.16) --
	(106.82,139.54);

\path[draw=drawColor,line width= 0.5pt,line join=round] ( 81.27, 32.16) --
	( 81.27,139.54);

\path[draw=drawColor,line width= 0.5pt,line join=round] ( 55.71, 32.16) --
	( 55.71,139.54);
\definecolor{drawColor}{RGB}{69,117,180}

\path[draw=drawColor,line width= 1.4pt,line join=round] ( 42.93, 38.02) --
	( 42.93, 38.02) --
	( 42.93, 39.00) --
	( 42.93, 39.00) --
	( 42.93, 39.97) --
	( 42.94, 39.97) --
	( 42.94, 40.95) --
	( 42.94, 40.95) --
	( 42.94, 41.92) --
	( 42.95, 41.92) --
	( 42.95, 42.90) --
	( 42.96, 42.90) --
	( 42.96, 43.88) --
	( 42.96, 43.88) --
	( 42.96, 44.85) --
	( 42.97, 44.85) --
	( 42.97, 45.83) --
	( 42.98, 45.83) --
	( 42.98, 46.80) --
	( 42.99, 46.80) --
	( 42.99, 47.78) --
	( 43.00, 47.78) --
	( 43.00, 48.76) --
	( 43.00, 48.76) --
	( 43.00, 49.73) --
	( 43.01, 49.73) --
	( 43.01, 50.71) --
	( 43.03, 50.71) --
	( 43.03, 51.69) --
	( 43.03, 51.69) --
	( 43.03, 52.66) --
	( 43.04, 52.66) --
	( 43.04, 53.64) --
	( 43.08, 53.64) --
	( 43.08, 54.61) --
	( 43.11, 54.61) --
	( 43.11, 55.59) --
	( 43.12, 55.59) --
	( 43.12, 56.57) --
	( 43.12, 56.57) --
	( 43.12, 57.54) --
	( 43.16, 57.54) --
	( 43.16, 58.52) --
	( 43.26, 58.52) --
	( 43.26, 59.49) --
	( 43.27, 59.49) --
	( 43.27, 60.47) --
	( 43.31, 60.47) --
	( 43.31, 61.45) --
	( 43.32, 61.45) --
	( 43.32, 62.42) --
	( 43.34, 62.42) --
	( 43.34, 63.40) --
	( 43.47, 63.40) --
	( 43.47, 64.38) --
	( 43.71, 64.38) --
	( 43.71, 65.35) --
	( 43.81, 65.35) --
	( 43.81, 66.33) --
	( 44.06, 66.33) --
	( 44.06, 67.30) --
	( 44.63, 67.30) --
	( 44.63, 68.28) --
	( 44.66, 68.28) --
	( 44.66, 69.26) --
	( 45.80, 69.26) --
	( 45.80, 70.23) --
	( 49.28, 70.23) --
	( 49.28, 71.21) --
	( 49.85, 71.21) --
	( 49.85, 72.18) --
	( 51.26, 72.18) --
	( 51.26, 73.16) --
	( 52.99, 73.16) --
	( 52.99, 74.14) --
	( 53.17, 74.14) --
	( 53.17, 75.11) --
	( 56.08, 75.11) --
	( 56.08, 76.09) --
	( 56.42, 76.09) --
	( 56.42, 77.07) --
	( 59.34, 77.07) --
	( 59.34, 78.04) --
	( 59.91, 78.04) --
	( 59.91, 79.02) --
	( 61.82, 79.02) --
	( 61.82, 79.99) --
	( 62.46, 79.99) --
	( 62.46, 80.97) --
	( 62.69, 80.97) --
	( 62.69, 81.95) --
	( 80.99, 81.95) --
	( 80.99, 82.92) --
	(113.29, 82.92) --
	(113.29, 83.90) --
	(116.43, 83.90) --
	(116.43, 84.88) --
	(133.83, 84.88) --
	(133.83, 85.85) --
	(134.94, 85.85) --
	(134.94, 85.85);
\definecolor{drawColor}{RGB}{252,141,89}

\path[draw=drawColor,line width= 1.4pt,dash pattern=on 2pt off 2pt ,line join=round] ( 42.93, 38.02) --
	( 42.93, 38.02) --
	( 42.93, 39.00) --
	( 42.94, 39.00) --
	( 42.94, 39.97) --
	( 42.94, 39.97) --
	( 42.94, 40.95) --
	( 42.94, 40.95) --
	( 42.94, 41.92) --
	( 42.95, 41.92) --
	( 42.95, 42.90) --
	( 42.95, 42.90) --
	( 42.95, 43.88) --
	( 42.96, 43.88) --
	( 42.96, 44.85) --
	( 42.96, 44.85) --
	( 42.96, 45.83) --
	( 42.98, 45.83) --
	( 42.98, 46.80) --
	( 42.99, 46.80) --
	( 42.99, 47.78) --
	( 43.00, 47.78) --
	( 43.00, 48.76) --
	( 43.00, 48.76) --
	( 43.00, 49.73) --
	( 43.01, 49.73) --
	( 43.01, 50.71) --
	( 43.02, 50.71) --
	( 43.02, 51.69) --
	( 43.04, 51.69) --
	( 43.04, 52.66) --
	( 43.04, 52.66) --
	( 43.04, 53.64) --
	( 43.05, 53.64) --
	( 43.05, 54.61) --
	( 43.07, 54.61) --
	( 43.07, 55.59) --
	( 43.08, 55.59) --
	( 43.08, 56.57) --
	( 43.08, 56.57) --
	( 43.08, 57.54) --
	( 43.17, 57.54) --
	( 43.17, 58.52) --
	( 43.18, 58.52) --
	( 43.18, 59.49) --
	( 43.20, 59.49) --
	( 43.20, 60.47) --
	( 43.21, 60.47) --
	( 43.21, 61.45) --
	( 43.41, 61.45) --
	( 43.41, 62.42) --
	( 43.45, 62.42) --
	( 43.45, 63.40) --
	( 43.64, 63.40) --
	( 43.64, 64.38) --
	( 43.65, 64.38) --
	( 43.65, 65.35) --
	( 43.80, 65.35) --
	( 43.80, 66.33) --
	( 43.87, 66.33) --
	( 43.87, 67.30) --
	( 44.09, 67.30) --
	( 44.09, 68.28) --
	( 44.40, 68.28) --
	( 44.40, 69.26) --
	( 44.46, 69.26) --
	( 44.46, 70.23) --
	( 47.40, 70.23) --
	( 47.40, 71.21) --
	( 47.86, 71.21) --
	( 47.86, 72.18) --
	( 49.53, 72.18) --
	( 49.53, 73.16) --
	( 50.82, 73.16) --
	( 50.82, 74.14) --
	( 51.44, 74.14) --
	( 51.44, 75.11) --
	( 52.99, 75.11) --
	( 52.99, 76.09) --
	( 54.49, 76.09) --
	( 54.49, 77.07) --
	( 56.55, 77.07) --
	( 56.55, 78.04) --
	( 56.92, 78.04) --
	( 56.92, 79.02) --
	( 61.75, 79.02) --
	( 61.75, 79.99) --
	( 63.29, 79.99) --
	( 63.29, 80.97) --
	( 67.68, 80.97) --
	( 67.68, 81.95) --
	( 75.84, 81.95) --
	( 75.84, 82.92) --
	( 82.19, 82.92) --
	( 82.19, 83.90) --
	( 91.41, 83.90) --
	( 91.41, 84.88) --
	( 98.68, 84.88) --
	( 98.68, 85.85) --
	(100.07, 85.85) --
	(100.07, 86.83) --
	(103.17, 86.83) --
	(103.17, 87.80) --
	(108.01, 87.80) --
	(108.01, 88.78) --
	(128.17, 88.78) --
	(128.17, 89.76) --
	(132.06, 89.76) --
	(132.06, 90.73) --
	(134.94, 90.73) --
	(134.94, 90.73);
\definecolor{drawColor}{RGB}{69,117,180}

\node[text=drawColor,anchor=base west,inner sep=0pt, outer sep=0pt, scale=  1.00] at ( 40,120) {50/100};
\definecolor{drawColor}{RGB}{252,141,89}

\node[text=drawColor,anchor=base west,inner sep=0pt, outer sep=0pt, scale=  1.00] at ( 40,110) {55/100};
\end{scope}
\begin{scope}
\path[clip] (  0.00,  0.00) rectangle (144.54,144.54);
\definecolor{drawColor}{gray}{0.30}

\node[text=drawColor,anchor=base,inner sep=0pt, outer sep=0pt, scale=  1.00] at (132.38, 20.78) {7000};

\node[text=drawColor,anchor=base,inner sep=0pt, outer sep=0pt, scale=  1.00] at (106.82, 20.78) {5000};

\node[text=drawColor,anchor=base,inner sep=0pt, outer sep=0pt, scale=  1.00] at ( 81.27, 20.78) {3000};

\node[text=drawColor,anchor=base,inner sep=0pt, outer sep=0pt, scale=  1.00] at ( 55.71, 20.78) {1000};
\end{scope}
\begin{scope}
\path[clip] (  0.00,  0.00) rectangle (144.54,144.54);
\definecolor{drawColor}{RGB}{0,0,0}

\node[text=drawColor,anchor=base,inner sep=0pt, outer sep=0pt, scale=  1.00] at ( 88.93,  6.94) {Computation time (s)};
\end{scope}
\end{tikzpicture}}%
    \scalebox{1.0}{\begin{tikzpicture}[x=1pt,y=1pt]
\definecolor{fillColor}{RGB}{255,255,255}
\path[use as bounding box,fill=fillColor,fill opacity=0.00] (0,0) rectangle (108.54,144.54);
\begin{scope}
\path[clip] ( 38.33, 32.16) rectangle (139.54,139.54);
\definecolor{drawColor}{gray}{0.92}

\path[draw=drawColor,line width= 0.3pt,line join=round] ( 38.33, 49.25) --
	(139.54, 49.25);

\path[draw=drawColor,line width= 0.3pt,line join=round] ( 38.33, 73.65) --
	(139.54, 73.65);

\path[draw=drawColor,line width= 0.3pt,line join=round] ( 38.33, 98.05) --
	(139.54, 98.05);

\path[draw=drawColor,line width= 0.3pt,line join=round] ( 38.33,122.46) --
	(139.54,122.46);

\path[draw=drawColor,line width= 0.3pt,line join=round] (119.60, 32.16) --
	(119.60,139.54);

\path[draw=drawColor,line width= 0.3pt,line join=round] ( 94.05, 32.16) --
	( 94.05,139.54);

\path[draw=drawColor,line width= 0.3pt,line join=round] ( 68.49, 32.16) --
	( 68.49,139.54);

\path[draw=drawColor,line width= 0.5pt,line join=round] ( 38.33, 37.04) --
	(139.54, 37.04);

\path[draw=drawColor,line width= 0.5pt,line join=round] ( 38.33, 61.45) --
	(139.54, 61.45);

\path[draw=drawColor,line width= 0.5pt,line join=round] ( 38.33, 85.85) --
	(139.54, 85.85);

\path[draw=drawColor,line width= 0.5pt,line join=round] ( 38.33,110.26) --
	(139.54,110.26);

\path[draw=drawColor,line width= 0.5pt,line join=round] ( 38.33,134.66) --
	(139.54,134.66);

\path[draw=drawColor,line width= 0.5pt,line join=round] (132.38, 32.16) --
	(132.38,139.54);

\path[draw=drawColor,line width= 0.5pt,line join=round] (106.82, 32.16) --
	(106.82,139.54);

\path[draw=drawColor,line width= 0.5pt,line join=round] ( 81.27, 32.16) --
	( 81.27,139.54);

\path[draw=drawColor,line width= 0.5pt,line join=round] ( 55.71, 32.16) --
	( 55.71,139.54);
\definecolor{drawColor}{RGB}{69,117,180}

\path[draw=drawColor,line width= 1.4pt,line join=round] ( 42.93, 38.02) --
	( 42.93, 38.02) --
	( 42.93, 39.00) --
	( 42.93, 39.00) --
	( 42.93, 39.97) --
	( 42.93, 39.97) --
	( 42.93, 40.95) --
	( 42.93, 40.95) --
	( 42.93, 41.92) --
	( 42.93, 41.92) --
	( 42.93, 42.90) --
	( 42.93, 42.90) --
	( 42.93, 43.88) --
	( 42.93, 43.88) --
	( 42.93, 44.85) --
	( 42.93, 44.85) --
	( 42.93, 45.83) --
	( 42.93, 45.83) --
	( 42.93, 46.80) --
	( 42.93, 46.80) --
	( 42.93, 47.78) --
	( 42.93, 47.78) --
	( 42.93, 48.76) --
	( 42.93, 48.76) --
	( 42.93, 49.73) --
	( 42.93, 49.73) --
	( 42.93, 50.71) --
	( 42.93, 50.71) --
	( 42.93, 51.69) --
	( 42.93, 51.69) --
	( 42.93, 52.66) --
	( 42.94, 52.66) --
	( 42.94, 53.64) --
	( 42.94, 53.64) --
	( 42.94, 54.61) --
	( 42.94, 54.61) --
	( 42.94, 55.59) --
	( 42.94, 55.59) --
	( 42.94, 56.57) --
	( 42.94, 56.57) --
	( 42.94, 57.54) --
	( 42.94, 57.54) --
	( 42.94, 58.52) --
	( 42.94, 58.52) --
	( 42.94, 59.49) --
	( 42.94, 59.49) --
	( 42.94, 60.47) --
	( 42.94, 60.47) --
	( 42.94, 61.45) --
	( 42.94, 61.45) --
	( 42.94, 62.42) --
	( 42.95, 62.42) --
	( 42.95, 63.40) --
	( 42.95, 63.40) --
	( 42.95, 64.38) --
	( 42.95, 64.38) --
	( 42.95, 65.35) --
	( 42.95, 65.35) --
	( 42.95, 66.33) --
	( 42.96, 66.33) --
	( 42.96, 67.30) --
	( 42.96, 67.30) --
	( 42.96, 68.28) --
	( 42.97, 68.28) --
	( 42.97, 69.26) --
	( 42.97, 69.26) --
	( 42.97, 70.23) --
	( 42.97, 70.23) --
	( 42.97, 71.21) --
	( 42.98, 71.21) --
	( 42.98, 72.18) --
	( 42.98, 72.18) --
	( 42.98, 73.16) --
	( 42.99, 73.16) --
	( 42.99, 74.14) --
	( 43.01, 74.14) --
	( 43.01, 75.11) --
	( 43.01, 75.11) --
	( 43.01, 76.09) --
	( 43.05, 76.09) --
	( 43.05, 77.07) --
	( 43.05, 77.07) --
	( 43.05, 78.04) --
	( 43.05, 78.04) --
	( 43.05, 79.02) --
	( 43.07, 79.02) --
	( 43.07, 79.99) --
	( 43.09, 79.99) --
	( 43.09, 80.97) --
	( 43.10, 80.97) --
	( 43.10, 81.95) --
	( 43.10, 81.95) --
	( 43.10, 82.92) --
	( 43.11, 82.92) --
	( 43.11, 83.90) --
	( 43.12, 83.90) --
	( 43.12, 84.88) --
	( 43.13, 84.88) --
	( 43.13, 85.85) --
	( 43.13, 85.85) --
	( 43.13, 86.83) --
	( 43.15, 86.83) --
	( 43.15, 87.80) --
	( 43.16, 87.80) --
	( 43.16, 88.78) --
	( 43.17, 88.78) --
	( 43.17, 89.76) --
	( 43.18, 89.76) --
	( 43.18, 90.73) --
	( 43.18, 90.73) --
	( 43.18, 91.71) --
	( 43.21, 91.71) --
	( 43.21, 92.68) --
	( 43.21, 92.68) --
	( 43.21, 93.66) --
	( 43.22, 93.66) --
	( 43.22, 94.64) --
	( 43.24, 94.64) --
	( 43.24, 95.61) --
	( 43.30, 95.61) --
	( 43.30, 96.59) --
	( 43.30, 96.59) --
	( 43.30, 97.57) --
	( 43.36, 97.57) --
	( 43.36, 98.54) --
	( 43.41, 98.54) --
	( 43.41, 99.52) --
	( 43.44, 99.52) --
	( 43.44,100.49) --
	( 43.45,100.49) --
	( 43.45,101.47) --
	( 43.57,101.47) --
	( 43.57,102.45) --
	( 43.96,102.45) --
	( 43.96,103.42) --
	( 44.03,103.42) --
	( 44.03,104.40) --
	( 44.19,104.40) --
	( 44.19,105.37) --
	( 45.07,105.37) --
	( 45.07,106.35) --
	( 45.73,106.35) --
	( 45.73,107.33) --
	( 64.10,107.33) --
	( 64.10,108.30) --
	( 70.67,108.30) --
	( 70.67,109.28) --
	(108.32,109.28) --
	(108.32,110.26) --
	(134.94,110.26) --
	(134.94,110.26);
\definecolor{drawColor}{RGB}{252,141,89}

\path[draw=drawColor,line width= 1.4pt,dash pattern=on 2pt off 2pt ,line join=round] ( 42.93, 38.02) --
	( 42.93, 38.02) --
	( 42.93, 39.00) --
	( 42.93, 39.00) --
	( 42.93, 39.97) --
	( 42.93, 39.97) --
	( 42.93, 40.95) --
	( 42.93, 40.95) --
	( 42.93, 41.92) --
	( 42.93, 41.92) --
	( 42.93, 42.90) --
	( 42.93, 42.90) --
	( 42.93, 43.88) --
	( 42.93, 43.88) --
	( 42.93, 44.85) --
	( 42.93, 44.85) --
	( 42.93, 45.83) --
	( 42.93, 45.83) --
	( 42.93, 46.80) --
	( 42.93, 46.80) --
	( 42.93, 47.78) --
	( 42.93, 47.78) --
	( 42.93, 48.76) --
	( 42.93, 48.76) --
	( 42.93, 49.73) --
	( 42.93, 49.73) --
	( 42.93, 50.71) --
	( 42.93, 50.71) --
	( 42.93, 51.69) --
	( 42.93, 51.69) --
	( 42.93, 52.66) --
	( 42.93, 52.66) --
	( 42.93, 53.64) --
	( 42.94, 53.64) --
	( 42.94, 54.61) --
	( 42.94, 54.61) --
	( 42.94, 55.59) --
	( 42.94, 55.59) --
	( 42.94, 56.57) --
	( 42.94, 56.57) --
	( 42.94, 57.54) --
	( 42.94, 57.54) --
	( 42.94, 58.52) --
	( 42.94, 58.52) --
	( 42.94, 59.49) --
	( 42.94, 59.49) --
	( 42.94, 60.47) --
	( 42.94, 60.47) --
	( 42.94, 61.45) --
	( 42.94, 61.45) --
	( 42.94, 62.42) --
	( 42.94, 62.42) --
	( 42.94, 63.40) --
	( 42.94, 63.40) --
	( 42.94, 64.38) --
	( 42.94, 64.38) --
	( 42.94, 65.35) --
	( 42.94, 65.35) --
	( 42.94, 66.33) --
	( 42.94, 66.33) --
	( 42.94, 67.30) --
	( 42.95, 67.30) --
	( 42.95, 68.28) --
	( 42.95, 68.28) --
	( 42.95, 69.26) --
	( 42.95, 69.26) --
	( 42.95, 70.23) --
	( 42.95, 70.23) --
	( 42.95, 71.21) --
	( 42.95, 71.21) --
	( 42.95, 72.18) --
	( 42.95, 72.18) --
	( 42.95, 73.16) --
	( 42.95, 73.16) --
	( 42.95, 74.14) --
	( 42.95, 74.14) --
	( 42.95, 75.11) --
	( 42.95, 75.11) --
	( 42.95, 76.09) --
	( 42.95, 76.09) --
	( 42.95, 77.07) --
	( 42.95, 77.07) --
	( 42.95, 78.04) --
	( 42.95, 78.04) --
	( 42.95, 79.02) --
	( 42.95, 79.02) --
	( 42.95, 79.99) --
	( 42.95, 79.99) --
	( 42.95, 80.97) --
	( 42.96, 80.97) --
	( 42.96, 81.95) --
	( 42.96, 81.95) --
	( 42.96, 82.92) --
	( 42.96, 82.92) --
	( 42.96, 83.90) --
	( 42.96, 83.90) --
	( 42.96, 84.88) --
	( 42.96, 84.88) --
	( 42.96, 85.85) --
	( 42.96, 85.85) --
	( 42.96, 86.83) --
	( 42.96, 86.83) --
	( 42.96, 87.80) --
	( 42.97, 87.80) --
	( 42.97, 88.78) --
	( 42.97, 88.78) --
	( 42.97, 89.76) --
	( 42.97, 89.76) --
	( 42.97, 90.73) --
	( 42.97, 90.73) --
	( 42.97, 91.71) --
	( 42.98, 91.71) --
	( 42.98, 92.68) --
	( 42.99, 92.68) --
	( 42.99, 93.66) --
	( 42.99, 93.66) --
	( 42.99, 94.64) --
	( 42.99, 94.64) --
	( 42.99, 95.61) --
	( 43.00, 95.61) --
	( 43.00, 96.59) --
	( 43.00, 96.59) --
	( 43.00, 97.57) --
	( 43.01, 97.57) --
	( 43.01, 98.54) --
	( 43.01, 98.54) --
	( 43.01, 99.52) --
	( 43.01, 99.52) --
	( 43.01,100.49) --
	( 43.02,100.49) --
	( 43.02,101.47) --
	( 43.02,101.47) --
	( 43.02,102.45) --
	( 43.03,102.45) --
	( 43.03,103.42) --
	( 43.03,103.42) --
	( 43.03,104.40) --
	( 43.05,104.40) --
	( 43.05,105.37) --
	( 43.09,105.37) --
	( 43.09,106.35) --
	( 43.09,106.35) --
	( 43.09,107.33) --
	( 43.10,107.33) --
	( 43.10,108.30) --
	( 43.10,108.30) --
	( 43.10,109.28) --
	( 43.12,109.28) --
	( 43.12,110.26) --
	( 43.13,110.26) --
	( 43.13,111.23) --
	( 43.14,111.23) --
	( 43.14,112.21) --
	( 43.15,112.21) --
	( 43.15,113.18) --
	( 43.15,113.18) --
	( 43.15,114.16) --
	( 43.15,114.16) --
	( 43.15,115.14) --
	( 43.15,115.14) --
	( 43.15,116.11) --
	( 43.17,116.11) --
	( 43.17,117.09) --
	( 43.17,117.09) --
	( 43.17,118.06) --
	( 43.19,118.06) --
	( 43.19,119.04) --
	( 43.30,119.04) --
	( 43.30,120.02) --
	( 43.38,120.02) --
	( 43.38,120.99) --
	( 43.41,120.99) --
	( 43.41,121.97) --
	( 43.53,121.97) --
	( 43.53,122.95) --
	( 43.56,122.95) --
	( 43.56,123.92) --
	( 43.81,123.92) --
	( 43.81,124.90) --
	( 44.09,124.90) --
	( 44.09,125.87) --
	( 44.74,125.87) --
	( 44.74,126.85) --
	( 44.77,126.85) --
	( 44.77,127.83) --
	( 48.04,127.83) --
	( 48.04,128.80) --
	( 60.04,128.80) --
	( 60.04,129.78) --
	( 70.22,129.78) --
	( 70.22,130.75) --
	( 72.90,130.75) --
	( 72.90,131.73) --
	( 77.27,131.73) --
	( 77.27,132.71) --
	( 96.34,132.71) --
	( 96.34,133.68) --
	(134.94,133.68) --
	(134.94,133.68);
\definecolor{drawColor}{RGB}{69,117,180}

\node[text=drawColor,anchor=base west,inner sep=0pt, outer sep=0pt, scale=  1.00] at ( 100,50) {75/100};
\definecolor{drawColor}{RGB}{252,141,89}

\node[text=drawColor,anchor=base west,inner sep=0pt, outer sep=0pt, scale=  1.00] at ( 100,40) {99/100};
\end{scope}
\begin{scope}
\path[clip] (  0.00,  0.00) rectangle (144.54,144.54);
\definecolor{drawColor}{gray}{0.30}

\node[text=drawColor,anchor=base,inner sep=0pt, outer sep=0pt, scale=  1.00] at (132.38, 20.78) {7000};

\node[text=drawColor,anchor=base,inner sep=0pt, outer sep=0pt, scale=  1.00] at (106.82, 20.78) {5000};

\node[text=drawColor,anchor=base,inner sep=0pt, outer sep=0pt, scale=  1.00] at ( 81.27, 20.78) {3000};

\node[text=drawColor,anchor=base,inner sep=0pt, outer sep=0pt, scale=  1.00] at ( 55.71, 20.78) {1000};
\end{scope}
\end{tikzpicture}}%
  } \caption{ECDF of computation time for the shortest-path problem (left), the
  knapsack problem (middle), and the portfolio selection problem (right) with
  continuous budgeted uncertainty sets. The numbers in the legend indicate the
  number of instances solved within the time limit of 2 hours out of 100
  instances. Solid \blue{blue}: \blue{bilevel approach}. Dotted \orange{orange}:
  \orange{robust approach}.}
  \label{fig:num-results:cont:budgeted:ecdf-time}
\end{figure}

For a time limit of 2 hours, the plots of
Figure~\ref{fig:num-results:cont:budgeted:ecdf-time} depict the ECDFs of
computation times regarding all instances from the test sets. It can be seen
that the robust approach outperforms the bilevel approach in terms of
computation time. Indeed, while it solves only \rev{$5$ more instances of the
knapsack problem (middle plot of
Figure~\ref{fig:num-results:cont:budgeted:ecdf-time}), it solves $17$ more
instances of the shortest-path problem (left plot) and $24$ more instances of
the portfolio selection problem (right plot).} The difference can also be
observed in the ECDFs of the number of branch-and-bound nodes as depicted in
Figure~\ref{fig:num-results:cont:ecdf-nodes} \rev{in
Appendix~\ref{sec:appendix:plots}}, where the robust approach tends to produce
smaller search trees than the bilevel approach for the \rev{shortest-path} and
the portfolio selection problem, while for the \rev{knapsack} problem, the
search trees are of similar size. Besides this, increasing the instance size
impacts the computation time for both approaches. For example, the robust
approach only solves $1$ instance of the \rev{shortest-path} problem with
\rev{$275$ nodes}, while the bilevel approach is not able to solve any instance
of a size larger than \rev{$200$ nodes} within the time limit. For the two
remaining applications, the runtimes increase similarly with increasing instance
size for both approaches. More details can be found in the scatter plots given
in Appendix~\ref{sec:appendix:plots}, where the runtimes of the individual
instances are shown.

\FloatBarrier

\subsubsection{Knapsack Uncertainty}
\label{sec:num-results:cont:knapsack}
As a second step, we consider all three applications with their respective
continuous decision-dependent knapsack uncertainty set. We again compare the
classic robust approach with the bilevel approach using the strong-duality based
single-level reformulation. The complete derivation of the reformulations can be
found in Appendix~\ref{sec:appendix:applications:shortest-path:knapsack} for the
shortest-path problem, in
Appendix~\ref{sec:appendix:applications:knapsack:knapsack} for the knapsack
problem, and in Appendix~\ref{sec:appendix:applications:portfolio:knapsack} for
the portfolio selection problem.

\begin{figure}
  \centering
  \makebox[\textwidth][c]{%
    \scalebox{1.0}{\input{plots/sp_cont_knap_ecdf_time.tex}}%
    \scalebox{1.0}{\begin{tikzpicture}[x=1pt,y=1pt]
\definecolor{fillColor}{RGB}{255,255,255}
\path[use as bounding box,fill=fillColor,fill opacity=0.00] (0,0) rectangle (108.54,144.54);
\begin{scope}
\path[clip] ( 38.33, 32.16) rectangle (139.54,139.54);
\definecolor{drawColor}{gray}{0.92}

\path[draw=drawColor,line width= 0.3pt,line join=round] ( 38.33, 49.25) --
	(139.54, 49.25);

\path[draw=drawColor,line width= 0.3pt,line join=round] ( 38.33, 73.65) --
	(139.54, 73.65);

\path[draw=drawColor,line width= 0.3pt,line join=round] ( 38.33, 98.05) --
	(139.54, 98.05);

\path[draw=drawColor,line width= 0.3pt,line join=round] ( 38.33,122.46) --
	(139.54,122.46);

\path[draw=drawColor,line width= 0.3pt,line join=round] (119.60, 32.16) --
	(119.60,139.54);

\path[draw=drawColor,line width= 0.3pt,line join=round] ( 94.05, 32.16) --
	( 94.05,139.54);

\path[draw=drawColor,line width= 0.3pt,line join=round] ( 68.49, 32.16) --
	( 68.49,139.54);

\path[draw=drawColor,line width= 0.5pt,line join=round] ( 38.33, 37.04) --
	(139.54, 37.04);

\path[draw=drawColor,line width= 0.5pt,line join=round] ( 38.33, 61.45) --
	(139.54, 61.45);

\path[draw=drawColor,line width= 0.5pt,line join=round] ( 38.33, 85.85) --
	(139.54, 85.85);

\path[draw=drawColor,line width= 0.5pt,line join=round] ( 38.33,110.26) --
	(139.54,110.26);

\path[draw=drawColor,line width= 0.5pt,line join=round] ( 38.33,134.66) --
	(139.54,134.66);

\path[draw=drawColor,line width= 0.5pt,line join=round] (132.38, 32.16) --
	(132.38,139.54);

\path[draw=drawColor,line width= 0.5pt,line join=round] (106.82, 32.16) --
	(106.82,139.54);

\path[draw=drawColor,line width= 0.5pt,line join=round] ( 81.27, 32.16) --
	( 81.27,139.54);

\path[draw=drawColor,line width= 0.5pt,line join=round] ( 55.71, 32.16) --
	( 55.71,139.54);
\definecolor{drawColor}{RGB}{69,117,180}

\path[draw=drawColor,line width= 1.4pt,line join=round] ( 43.01, 38.02) --
	( 43.03, 38.02) --
	( 43.03, 39.00) --
	( 43.07, 39.00) --
	( 43.07, 39.97) --
	( 43.09, 39.97) --
	( 43.09, 40.95) --
	( 43.10, 40.95) --
	( 43.10, 41.92) --
	( 43.12, 41.92) --
	( 43.12, 42.90) --
	( 43.21, 42.90) --
	( 43.21, 43.88) --
	( 43.35, 43.88) --
	( 43.35, 44.85) --
	( 43.46, 44.85) --
	( 43.46, 45.83) --
	( 43.47, 45.83) --
	( 43.47, 46.80) --
	( 43.48, 46.80) --
	( 43.48, 47.78) --
	( 43.53, 47.78) --
	( 43.53, 48.76) --
	( 43.60, 48.76) --
	( 43.60, 49.73) --
	( 43.65, 49.73) --
	( 43.65, 50.71) --
	( 43.79, 50.71) --
	( 43.79, 51.69) --
	( 43.87, 51.69) --
	( 43.87, 52.66) --
	( 43.92, 52.66) --
	( 43.92, 53.64) --
	( 44.26, 53.64) --
	( 44.26, 54.61) --
	( 44.40, 54.61) --
	( 44.40, 55.59) --
	( 45.19, 55.59) --
	( 45.19, 56.57) --
	( 45.50, 56.57) --
	( 45.50, 57.54) --
	( 45.91, 57.54) --
	( 45.91, 58.52) --
	( 46.29, 58.52) --
	( 46.29, 59.49) --
	( 46.36, 59.49) --
	( 46.36, 60.47) --
	( 47.03, 60.47) --
	( 47.03, 61.45) --
	( 47.73, 61.45) --
	( 47.73, 62.42) --
	( 48.67, 62.42) --
	( 48.67, 63.40) --
	( 49.32, 63.40) --
	( 49.32, 64.38) --
	( 50.40, 64.38) --
	( 50.40, 65.35) --
	( 51.12, 65.35) --
	( 51.12, 66.33) --
	( 51.71, 66.33) --
	( 51.71, 67.30) --
	( 52.12, 67.30) --
	( 52.12, 68.28) --
	( 53.19, 68.28) --
	( 53.19, 69.26) --
	( 54.75, 69.26) --
	( 54.75, 70.23) --
	( 54.75, 70.23) --
	( 54.75, 71.21) --
	( 56.06, 71.21) --
	( 56.06, 72.18) --
	( 58.96, 72.18) --
	( 58.96, 73.16) --
	( 62.25, 73.16) --
	( 62.25, 74.14) --
	( 63.80, 74.14) --
	( 63.80, 75.11) --
	( 65.19, 75.11) --
	( 65.19, 76.09) --
	( 67.60, 76.09) --
	( 67.60, 77.07) --
	( 68.04, 77.07) --
	( 68.04, 78.04) --
	( 68.05, 78.04) --
	( 68.05, 79.02) --
	( 68.49, 79.02) --
	( 68.49, 79.99) --
	( 70.67, 79.99) --
	( 70.67, 80.97) --
	( 72.21, 80.97) --
	( 72.21, 81.95) --
	( 72.51, 81.95) --
	( 72.51, 82.92) --
	( 76.89, 82.92) --
	( 76.89, 83.90) --
	( 78.15, 83.90) --
	( 78.15, 84.88) --
	( 78.88, 84.88) --
	( 78.88, 85.85) --
	( 80.18, 85.85) --
	( 80.18, 86.83) --
	( 81.10, 86.83) --
	( 81.10, 87.80) --
	( 81.84, 87.80) --
	( 81.84, 88.78) --
	( 83.78, 88.78) --
	( 83.78, 89.76) --
	( 87.18, 89.76) --
	( 87.18, 90.73) --
	( 88.05, 90.73) --
	( 88.05, 91.71) --
	( 91.20, 91.71) --
	( 91.20, 92.68) --
	( 96.22, 92.68) --
	( 96.22, 93.66) --
	(100.14, 93.66) --
	(100.14, 94.64) --
	(100.91, 94.64) --
	(100.91, 95.61) --
	(104.87, 95.61) --
	(104.87, 96.59) --
	(104.91, 96.59) --
	(104.91, 97.57) --
	(106.79, 97.57) --
	(106.79, 98.54) --
	(107.35, 98.54) --
	(107.35, 99.52) --
	(116.76, 99.52) --
	(116.76,100.49) --
	(116.95,100.49) --
	(116.95,101.47) --
	(121.14,101.47) --
	(121.14,102.45) --
	(122.14,102.45) --
	(122.14,103.42) --
	(128.08,103.42) --
	(128.08,104.40) --
	(134.94,104.40) --
	(134.94,104.40);
\definecolor{drawColor}{RGB}{252,141,89}

\path[draw=drawColor,line width= 1.4pt,dash pattern=on 2pt off 2pt ,line join=round] ( 43.02, 38.02) --
	( 43.04, 38.02) --
	( 43.04, 39.00) --
	( 43.04, 39.00) --
	( 43.04, 39.97) --
	( 43.06, 39.97) --
	( 43.06, 40.95) --
	( 43.07, 40.95) --
	( 43.07, 41.92) --
	( 43.10, 41.92) --
	( 43.10, 42.90) --
	( 43.23, 42.90) --
	( 43.23, 43.88) --
	( 43.31, 43.88) --
	( 43.31, 44.85) --
	( 43.37, 44.85) --
	( 43.37, 45.83) --
	( 43.41, 45.83) --
	( 43.41, 46.80) --
	( 43.43, 46.80) --
	( 43.43, 47.78) --
	( 43.44, 47.78) --
	( 43.44, 48.76) --
	( 43.44, 48.76) --
	( 43.44, 49.73) --
	( 43.45, 49.73) --
	( 43.45, 50.71) --
	( 43.54, 50.71) --
	( 43.54, 51.69) --
	( 43.66, 51.69) --
	( 43.66, 52.66) --
	( 43.68, 52.66) --
	( 43.68, 53.64) --
	( 44.48, 53.64) --
	( 44.48, 54.61) --
	( 44.81, 54.61) --
	( 44.81, 55.59) --
	( 45.05, 55.59) --
	( 45.05, 56.57) --
	( 45.38, 56.57) --
	( 45.38, 57.54) --
	( 45.69, 57.54) --
	( 45.69, 58.52) --
	( 46.24, 58.52) --
	( 46.24, 59.49) --
	( 46.25, 59.49) --
	( 46.25, 60.47) --
	( 46.42, 60.47) --
	( 46.42, 61.45) --
	( 47.53, 61.45) --
	( 47.53, 62.42) --
	( 48.14, 62.42) --
	( 48.14, 63.40) --
	( 49.32, 63.40) --
	( 49.32, 64.38) --
	( 49.38, 64.38) --
	( 49.38, 65.35) --
	( 50.42, 65.35) --
	( 50.42, 66.33) --
	( 50.68, 66.33) --
	( 50.68, 67.30) --
	( 50.77, 67.30) --
	( 50.77, 68.28) --
	( 53.64, 68.28) --
	( 53.64, 69.26) --
	( 54.93, 69.26) --
	( 54.93, 70.23) --
	( 56.04, 70.23) --
	( 56.04, 71.21) --
	( 56.89, 71.21) --
	( 56.89, 72.18) --
	( 58.21, 72.18) --
	( 58.21, 73.16) --
	( 58.50, 73.16) --
	( 58.50, 74.14) --
	( 61.40, 74.14) --
	( 61.40, 75.11) --
	( 73.07, 75.11) --
	( 73.07, 76.09) --
	( 73.35, 76.09) --
	( 73.35, 77.07) --
	( 73.85, 77.07) --
	( 73.85, 78.04) --
	( 75.69, 78.04) --
	( 75.69, 79.02) --
	( 78.47, 79.02) --
	( 78.47, 79.99) --
	( 78.50, 79.99) --
	( 78.50, 80.97) --
	( 95.48, 80.97) --
	( 95.48, 81.95) --
	( 95.98, 81.95) --
	( 95.98, 82.92) --
	( 96.40, 82.92) --
	( 96.40, 83.90) --
	(103.55, 83.90) --
	(103.55, 84.88) --
	(104.51, 84.88) --
	(104.51, 85.85) --
	(105.70, 85.85) --
	(105.70, 86.83) --
	(106.45, 86.83) --
	(106.45, 87.80) --
	(110.77, 87.80) --
	(110.77, 88.78) --
	(115.47, 88.78) --
	(115.47, 89.76) --
	(121.63, 89.76) --
	(121.63, 90.73) --
	(122.76, 90.73) --
	(122.76, 91.71) --
	(123.70, 91.71) --
	(123.70, 92.68) --
	(125.32, 92.68) --
	(125.32, 93.66) --
	(130.23, 93.66) --
	(130.23, 94.64) --
	(134.94, 94.64) --
	(134.94, 94.64);
\definecolor{drawColor}{RGB}{69,117,180}

\node[text=drawColor,anchor=base west,inner sep=0pt, outer sep=0pt, scale=  1.00] at ( 40, 120) {69/100};
\definecolor{drawColor}{RGB}{252,141,89}

\node[text=drawColor,anchor=base west,inner sep=0pt, outer sep=0pt, scale=  1.00] at ( 40, 110) {59/100};
\end{scope}
\begin{scope}
\path[clip] (  0.00,  0.00) rectangle (144.54,144.54);
\definecolor{drawColor}{gray}{0.30}

\end{scope}
\begin{scope}
\path[clip] (  0.00,  0.00) rectangle (144.54,144.54);
\definecolor{drawColor}{gray}{0.30}

\node[text=drawColor,anchor=base,inner sep=0pt, outer sep=0pt, scale=  1.00] at (132.38, 20.78) {7000};

\node[text=drawColor,anchor=base,inner sep=0pt, outer sep=0pt, scale=  1.00] at (106.82, 20.78) {5000};

\node[text=drawColor,anchor=base,inner sep=0pt, outer sep=0pt, scale=  1.00] at ( 81.27, 20.78) {3000};

\node[text=drawColor,anchor=base,inner sep=0pt, outer sep=0pt, scale=  1.00] at ( 55.71, 20.78) {1000};
\end{scope}
\begin{scope}
\path[clip] (  0.00,  0.00) rectangle (144.54,144.54);
\definecolor{drawColor}{RGB}{0,0,0}

\node[text=drawColor,anchor=base,inner sep=0pt, outer sep=0pt, scale=  1.00] at ( 88.93,  6.94) {Computation time (s)};
\end{scope}
\begin{scope}
\path[clip] (  0.00,  0.00) rectangle (144.54,144.54);
\definecolor{drawColor}{RGB}{0,0,0}

\end{scope}
\end{tikzpicture}}%
    \scalebox{1.0}{\begin{tikzpicture}[x=1pt,y=1pt]
\definecolor{fillColor}{RGB}{255,255,255}
\path[use as bounding box,fill=fillColor,fill opacity=0.00] (0,0) rectangle (108.54,144.54);
\begin{scope}
\path[clip] ( 38.33, 32.16) rectangle (139.54,139.54);
\definecolor{drawColor}{gray}{0.92}

\path[draw=drawColor,line width= 0.3pt,line join=round] ( 38.33, 49.25) --
	(139.54, 49.25);

\path[draw=drawColor,line width= 0.3pt,line join=round] ( 38.33, 73.65) --
	(139.54, 73.65);

\path[draw=drawColor,line width= 0.3pt,line join=round] ( 38.33, 98.05) --
	(139.54, 98.05);

\path[draw=drawColor,line width= 0.3pt,line join=round] ( 38.33,122.46) --
	(139.54,122.46);

\path[draw=drawColor,line width= 0.3pt,line join=round] (119.60, 32.16) --
	(119.60,139.54);

\path[draw=drawColor,line width= 0.3pt,line join=round] ( 94.05, 32.16) --
	( 94.05,139.54);

\path[draw=drawColor,line width= 0.3pt,line join=round] ( 68.49, 32.16) --
	( 68.49,139.54);

\path[draw=drawColor,line width= 0.5pt,line join=round] ( 38.33, 37.04) --
	(139.54, 37.04);

\path[draw=drawColor,line width= 0.5pt,line join=round] ( 38.33, 61.45) --
	(139.54, 61.45);

\path[draw=drawColor,line width= 0.5pt,line join=round] ( 38.33, 85.85) --
	(139.54, 85.85);

\path[draw=drawColor,line width= 0.5pt,line join=round] ( 38.33,110.26) --
	(139.54,110.26);

\path[draw=drawColor,line width= 0.5pt,line join=round] ( 38.33,134.66) --
	(139.54,134.66);

\path[draw=drawColor,line width= 0.5pt,line join=round] (132.38, 32.16) --
	(132.38,139.54);

\path[draw=drawColor,line width= 0.5pt,line join=round] (106.82, 32.16) --
	(106.82,139.54);

\path[draw=drawColor,line width= 0.5pt,line join=round] ( 81.27, 32.16) --
	( 81.27,139.54);

\path[draw=drawColor,line width= 0.5pt,line join=round] ( 55.71, 32.16) --
	( 55.71,139.54);
\definecolor{drawColor}{RGB}{69,117,180}

\path[draw=drawColor,line width= 1.4pt,line join=round] ( 42.93, 38.02) --
	( 42.93, 38.02) --
	( 42.93, 39.00) --
	( 42.93, 39.00) --
	( 42.93, 39.97) --
	( 42.93, 39.97) --
	( 42.93, 40.95) --
	( 42.93, 40.95) --
	( 42.93, 41.92) --
	( 42.93, 41.92) --
	( 42.93, 42.90) --
	( 42.93, 42.90) --
	( 42.93, 43.88) --
	( 42.93, 43.88) --
	( 42.93, 44.85) --
	( 42.93, 44.85) --
	( 42.93, 45.83) --
	( 42.93, 45.83) --
	( 42.93, 46.80) --
	( 42.93, 46.80) --
	( 42.93, 47.78) --
	( 42.93, 47.78) --
	( 42.93, 48.76) --
	( 42.93, 48.76) --
	( 42.93, 49.73) --
	( 42.93, 49.73) --
	( 42.93, 50.71) --
	( 42.93, 50.71) --
	( 42.93, 51.69) --
	( 42.93, 51.69) --
	( 42.93, 52.66) --
	( 42.93, 52.66) --
	( 42.93, 53.64) --
	( 42.93, 53.64) --
	( 42.93, 54.61) --
	( 42.93, 54.61) --
	( 42.93, 55.59) --
	( 42.93, 55.59) --
	( 42.93, 56.57) --
	( 42.93, 56.57) --
	( 42.93, 57.54) --
	( 42.93, 57.54) --
	( 42.93, 58.52) --
	( 42.93, 58.52) --
	( 42.93, 59.49) --
	( 42.93, 59.49) --
	( 42.93, 60.47) --
	( 42.93, 60.47) --
	( 42.93, 61.45) --
	( 42.93, 61.45) --
	( 42.93, 62.42) --
	( 42.93, 62.42) --
	( 42.93, 63.40) --
	( 42.93, 63.40) --
	( 42.93, 64.38) --
	( 42.93, 64.38) --
	( 42.93, 65.35) --
	( 42.93, 65.35) --
	( 42.93, 66.33) --
	( 42.93, 66.33) --
	( 42.93, 67.30) --
	( 42.93, 67.30) --
	( 42.93, 68.28) --
	( 42.93, 68.28) --
	( 42.93, 69.26) --
	( 42.93, 69.26) --
	( 42.93, 70.23) --
	( 42.93, 70.23) --
	( 42.93, 71.21) --
	( 42.93, 71.21) --
	( 42.93, 72.18) --
	( 42.93, 72.18) --
	( 42.93, 73.16) --
	( 42.93, 73.16) --
	( 42.93, 74.14) --
	( 42.93, 74.14) --
	( 42.93, 75.11) --
	( 42.93, 75.11) --
	( 42.93, 76.09) --
	( 42.93, 76.09) --
	( 42.93, 77.07) --
	( 42.93, 77.07) --
	( 42.93, 78.04) --
	( 42.93, 78.04) --
	( 42.93, 79.02) --
	( 42.94, 79.02) --
	( 42.94, 79.99) --
	( 42.94, 79.99) --
	( 42.94, 80.97) --
	( 42.94, 80.97) --
	( 42.94, 81.95) --
	( 42.94, 81.95) --
	( 42.94, 82.92) --
	( 42.94, 82.92) --
	( 42.94, 83.90) --
	( 42.94, 83.90) --
	( 42.94, 84.88) --
	( 42.94, 84.88) --
	( 42.94, 85.85) --
	( 42.94, 85.85) --
	( 42.94, 86.83) --
	( 42.94, 86.83) --
	( 42.94, 87.80) --
	( 42.94, 87.80) --
	( 42.94, 88.78) --
	( 42.94, 88.78) --
	( 42.94, 89.76) --
	( 42.94, 89.76) --
	( 42.94, 90.73) --
	( 42.95, 90.73) --
	( 42.95, 91.71) --
	( 42.95, 91.71) --
	( 42.95, 92.68) --
	( 42.95, 92.68) --
	( 42.95, 93.66) --
	( 42.96, 93.66) --
	( 42.96, 94.64) --
	( 42.96, 94.64) --
	( 42.96, 95.61) --
	( 42.96, 95.61) --
	( 42.96, 96.59) --
	( 42.97, 96.59) --
	( 42.97, 97.57) --
	( 42.98, 97.57) --
	( 42.98, 98.54) --
	( 42.98, 98.54) --
	( 42.98, 99.52) --
	( 42.99, 99.52) --
	( 42.99,100.49) --
	( 42.99,100.49) --
	( 42.99,101.47) --
	( 42.99,101.47) --
	( 42.99,102.45) --
	( 43.01,102.45) --
	( 43.01,103.42) --
	( 43.25,103.42) --
	( 43.25,104.40) --
	( 43.42,104.40) --
	( 43.42,105.37) --
	( 43.93,105.37) --
	( 43.93,106.35) --
	( 44.83,106.35) --
	( 44.83,107.33) --
	( 50.69,107.33) --
	( 50.69,108.30) --
	( 51.89,108.30) --
	( 51.89,109.28) --
	( 58.58,109.28) --
	( 58.58,110.26) --
	( 58.90,110.26) --
	( 58.90,111.23) --
	(100.79,111.23) --
	(100.79,112.21) --
	(109.14,112.21) --
	(109.14,113.18) --
	(134.94,113.18) --
	(134.94,113.18);
\definecolor{drawColor}{RGB}{252,141,89}

\path[draw=drawColor,line width= 1.4pt,dash pattern=on 2pt off 2pt ,line join=round] ( 42.93, 38.02) --
	( 42.93, 38.02) --
	( 42.93, 39.00) --
	( 42.93, 39.00) --
	( 42.93, 39.97) --
	( 42.93, 39.97) --
	( 42.93, 40.95) --
	( 42.93, 40.95) --
	( 42.93, 41.92) --
	( 42.93, 41.92) --
	( 42.93, 42.90) --
	( 42.93, 42.90) --
	( 42.93, 43.88) --
	( 42.93, 43.88) --
	( 42.93, 44.85) --
	( 42.93, 44.85) --
	( 42.93, 45.83) --
	( 42.93, 45.83) --
	( 42.93, 46.80) --
	( 42.93, 46.80) --
	( 42.93, 47.78) --
	( 42.93, 47.78) --
	( 42.93, 48.76) --
	( 42.93, 48.76) --
	( 42.93, 49.73) --
	( 42.93, 49.73) --
	( 42.93, 50.71) --
	( 42.93, 50.71) --
	( 42.93, 51.69) --
	( 42.93, 51.69) --
	( 42.93, 52.66) --
	( 42.93, 52.66) --
	( 42.93, 53.64) --
	( 42.93, 53.64) --
	( 42.93, 54.61) --
	( 42.93, 54.61) --
	( 42.93, 55.59) --
	( 42.93, 55.59) --
	( 42.93, 56.57) --
	( 42.93, 56.57) --
	( 42.93, 57.54) --
	( 42.93, 57.54) --
	( 42.93, 58.52) --
	( 42.93, 58.52) --
	( 42.93, 59.49) --
	( 42.93, 59.49) --
	( 42.93, 60.47) --
	( 42.93, 60.47) --
	( 42.93, 61.45) --
	( 42.93, 61.45) --
	( 42.93, 62.42) --
	( 42.93, 62.42) --
	( 42.93, 63.40) --
	( 42.93, 63.40) --
	( 42.93, 64.38) --
	( 42.93, 64.38) --
	( 42.93, 65.35) --
	( 42.93, 65.35) --
	( 42.93, 66.33) --
	( 42.93, 66.33) --
	( 42.93, 67.30) --
	( 42.93, 67.30) --
	( 42.93, 68.28) --
	( 42.93, 68.28) --
	( 42.93, 69.26) --
	( 42.93, 69.26) --
	( 42.93, 70.23) --
	( 42.93, 70.23) --
	( 42.93, 71.21) --
	( 42.93, 71.21) --
	( 42.93, 72.18) --
	( 42.93, 72.18) --
	( 42.93, 73.16) --
	( 42.93, 73.16) --
	( 42.93, 74.14) --
	( 42.93, 74.14) --
	( 42.93, 75.11) --
	( 42.93, 75.11) --
	( 42.93, 76.09) --
	( 42.93, 76.09) --
	( 42.93, 77.07) --
	( 42.93, 77.07) --
	( 42.93, 78.04) --
	( 42.93, 78.04) --
	( 42.93, 79.02) --
	( 42.93, 79.02) --
	( 42.93, 79.99) --
	( 42.93, 79.99) --
	( 42.93, 80.97) --
	( 42.93, 80.97) --
	( 42.93, 81.95) --
	( 42.93, 81.95) --
	( 42.93, 82.92) --
	( 42.93, 82.92) --
	( 42.93, 83.90) --
	( 42.93, 83.90) --
	( 42.93, 84.88) --
	( 42.94, 84.88) --
	( 42.94, 85.85) --
	( 42.94, 85.85) --
	( 42.94, 86.83) --
	( 42.94, 86.83) --
	( 42.94, 87.80) --
	( 42.94, 87.80) --
	( 42.94, 88.78) --
	( 42.94, 88.78) --
	( 42.94, 89.76) --
	( 42.94, 89.76) --
	( 42.94, 90.73) --
	( 42.94, 90.73) --
	( 42.94, 91.71) --
	( 42.94, 91.71) --
	( 42.94, 92.68) --
	( 42.94, 92.68) --
	( 42.94, 93.66) --
	( 42.94, 93.66) --
	( 42.94, 94.64) --
	( 42.94, 94.64) --
	( 42.94, 95.61) --
	( 42.94, 95.61) --
	( 42.94, 96.59) --
	( 42.94, 96.59) --
	( 42.94, 97.57) --
	( 42.94, 97.57) --
	( 42.94, 98.54) --
	( 42.94, 98.54) --
	( 42.94, 99.52) --
	( 42.95, 99.52) --
	( 42.95,100.49) --
	( 42.95,100.49) --
	( 42.95,101.47) --
	( 42.95,101.47) --
	( 42.95,102.45) --
	( 42.96,102.45) --
	( 42.96,103.42) --
	( 43.01,103.42) --
	( 43.01,104.40) --
	( 43.15,104.40) --
	( 43.15,105.37) --
	( 43.24,105.37) --
	( 43.24,106.35) --
	( 43.29,106.35) --
	( 43.29,107.33) --
	( 43.32,107.33) --
	( 43.32,108.30) --
	( 43.47,108.30) --
	( 43.47,109.28) --
	( 45.18,109.28) --
	( 45.18,110.26) --
	( 45.36,110.26) --
	( 45.36,111.23) --
	( 46.26,111.23) --
	( 46.26,112.21) --
	( 47.59,112.21) --
	( 47.59,113.18) --
	( 48.60,113.18) --
	( 48.60,114.16) --
	( 48.62,114.16) --
	( 48.62,115.14) --
	( 49.10,115.14) --
	( 49.10,116.11) --
	( 53.33,116.11) --
	( 53.33,117.09) --
	( 55.16,117.09) --
	( 55.16,118.06) --
	( 55.94,118.06) --
	( 55.94,119.04) --
	( 57.45,119.04) --
	( 57.45,120.02) --
	( 67.35,120.02) --
	( 67.35,120.99) --
	( 74.87,120.99) --
	( 74.87,121.97) --
	( 77.98,121.97) --
	( 77.98,122.95) --
	(134.94,122.95) --
	(134.94,122.95);
\definecolor{drawColor}{RGB}{69,117,180}

\node[text=drawColor,anchor=base west,inner sep=0pt, outer sep=0pt, scale=  1.00] at ( 100,50) {78/100};
\definecolor{drawColor}{RGB}{252,141,89}

\node[text=drawColor,anchor=base west,inner sep=0pt, outer sep=0pt, scale=  1.00] at ( 100,40) {88/100};
\end{scope}
\begin{scope}
\path[clip] (  0.00,  0.00) rectangle (144.54,144.54);
\definecolor{drawColor}{gray}{0.30}

\end{scope}
\begin{scope}
\path[clip] (  0.00,  0.00) rectangle (144.54,144.54);
\definecolor{drawColor}{gray}{0.30}

\node[text=drawColor,anchor=base,inner sep=0pt, outer sep=0pt, scale=  1.00] at (132.38, 20.78) {7000};

\node[text=drawColor,anchor=base,inner sep=0pt, outer sep=0pt, scale=  1.00] at (106.82, 20.78) {5000};

\node[text=drawColor,anchor=base,inner sep=0pt, outer sep=0pt, scale=  1.00] at ( 81.27, 20.78) {3000};

\node[text=drawColor,anchor=base,inner sep=0pt, outer sep=0pt, scale=  1.00] at ( 55.71, 20.78) {1000};
\end{scope}
\begin{scope}
\path[clip] (  0.00,  0.00) rectangle (144.54,144.54);
\definecolor{drawColor}{RGB}{0,0,0}

\end{scope}
\begin{scope}
\path[clip] (  0.00,  0.00) rectangle (144.54,144.54);
\definecolor{drawColor}{RGB}{0,0,0}

\end{scope}
\end{tikzpicture}}%
  } \caption{ECDF of computation time for the shortest-path problem (left), the
  knapsack problem (middle), and the portfolio selection problem (right) with
  continuous knapsack uncertainty sets. The numbers in the legend indicate the
  number of instances solved within the time limit of 2 hours out of 100
  instances. Solid \blue{blue}: \blue{bilevel approach}. Dotted \orange{orange}:
  \orange{robust approach}.}
  \label{fig:num-results:cont:knapsack:ecdf-time}
\end{figure}

In Figure~\ref{fig:num-results:cont:knapsack:ecdf-time}, we depict the ECDF of
computation times for all three applications with the continuous knapsack
uncertainty set. Again, the shortest-path problem is shown on the left, the
knapsack problem in the middle, and the portfolio selection problem on the
right. We see that the bilevel approach performs better compared to the budgeted
uncertainty case shown in the previous section. For the knapsack problem, the
bilevel approach is even able to solve $10$ more instances than the robust
approach. However, for the portfolio problem, the robust approach still
outperforms the bilevel approach by solving $10$ more instances. For the
shortest-path problem, both approaches are able to solve all the instances from
the test set. The ECDF however shows that the robust approach is significantly
faster. The ECDF of the number of branch-and-bound nodes depicted in
Figure~\ref{fig:num-results:cont:ecdf-nodes} \rev{in
Appendix~\ref{sec:appendix:plots}} confirms this observation. For the knapsack
instances, the bilevel approach tends to produce smaller search trees than the
robust approach, while for the portfolio problem, the search trees are of
similar size. For the shortest-path problem, the robust approach solves all
instances at the root node, while the bilevel approach produces larger search
trees for some instances. As in the case of the budgeted uncertainty set,
increasing the instance sizes impacts the computation time for both approaches.
As an example, both approaches are able to solve all 10 knapsack instances with
$3000$~items, whereas the robust approach only solves $2$ instances with
$9000$~items and the bilevel approach is not able to solve any instance with
more than $8000$~items within the time limit. Similar observations can again be
made for the other two applications; see Appendix~\ref{sec:appendix:plots}.

\FloatBarrier

\subsubsection{Conclusion}
\label{sec:num-results:cont:conclusion}
For continuous uncertainty sets, where the DDRO problem can be reformulated as a
single-level optimization problem using strong duality, the classic robust
approach generally outperforms the bilevel approach in terms of computation
time. However, for the continuous knapsack uncertainty set, the difference
between the two approaches is substantially smaller than for the continuous
budgeted uncertainty set. In particular, for the knapsack problem, the bilevel
approach is even able to solve more instances than the robust approach.
This behavior can be explained by the fact that the additional primal
lower-level constraints in the reformulation of the bilevel problem are
significantly less in case of knapsack uncertainty than in case of budgeted
uncertainty. The knapsack uncertainty set only introduces one additional primal
lower-level constraint, while the budgeted uncertainty set introduces one budget
constraint and one additional constraint for each uncertain parameter.

In conclusion, the bilevel approach can be a competitive alternative to the
classic robust approach for uncertainty sets with a small number of constraints,
while for uncertainty sets with a large number of constraints, the classic
robust approach is generally superior.

\subsection{Numerical Results for Discrete Uncertainty Sets}
\label{sec:num-results:discrete}

In this section, we study the discrete versions of the budgeted and the knapsack
uncertainty sets for the shortest-path problem and the knapsack problem.

To this end, we solve the bilevel reformulation using the mixed-integer
linear bilevel solver \textsf{MibS} \parencite{mibs}.
We compare these results to those obtained by using the quantified integer
programming solver~\textsf{Yasol} \parencite{yasol2017}.
We do not present numerical results for the portfolio selection
problem, since quadratic constraints can neither be handled by \textsf{MibS}
nor by \textsf{Yasol}.

\subsubsection{Budgeted Uncertainty}
\label{sec:num-results:discrete:budgeted}
We consider the discrete version of the budgeted uncertainty set for the
shortest-path problem and the knapsack problem. The complete derivation of the
reformulations can be found in
Appendix~\ref{sec:appendix:applications:shortest-path:budgeted} for the
shortest-path problem and in
Appendix~\ref{sec:appendix:applications:knapsack:budgeted} for the knapsack
problem. For the QIP reformulations, we refer to Appendix~\ref{sec:qip-models}.
Note that the budgeted uncertainty set $U^{\text{db}}$ is given by
interdiction-like constraints. However, \textsf{MibS} is not able to recognize
this structure in our reformulation and thus cannot apply specialized methods
for solving interdiction problems that would most likely improve the
performance.

\begin{figure}
  \centering
  \makebox[\textwidth][c]{%
    \scalebox{1.0}{\input{plots/optl_rev/sp_disc_budg_ecdf_time.tex}}%
    \scalebox{1.0}{\begin{tikzpicture}[x=1pt,y=1pt]
\definecolor{fillColor}{RGB}{255,255,255}
\path[use as bounding box,fill=fillColor,fill opacity=0.00] (0,0) rectangle (144.54,144.54);
\begin{scope}
\path[clip] ( 38.33, 32.16) rectangle (139.54,139.54);
\definecolor{drawColor}{gray}{0.92}

\path[draw=drawColor,line width= 0.3pt,line join=round] ( 38.33, 49.25) --
	(139.54, 49.25);

\path[draw=drawColor,line width= 0.3pt,line join=round] ( 38.33, 73.65) --
	(139.54, 73.65);

\path[draw=drawColor,line width= 0.3pt,line join=round] ( 38.33, 98.05) --
	(139.54, 98.05);

\path[draw=drawColor,line width= 0.3pt,line join=round] ( 38.33,122.46) --
	(139.54,122.46);

\path[draw=drawColor,line width= 0.3pt,line join=round] (119.60, 32.16) --
	(119.60,139.54);

\path[draw=drawColor,line width= 0.3pt,line join=round] ( 94.05, 32.16) --
	( 94.05,139.54);

\path[draw=drawColor,line width= 0.3pt,line join=round] ( 68.49, 32.16) --
	( 68.49,139.54);

\path[draw=drawColor,line width= 0.5pt,line join=round] ( 38.33, 37.04) --
	(139.54, 37.04);

\path[draw=drawColor,line width= 0.5pt,line join=round] ( 38.33, 61.45) --
	(139.54, 61.45);

\path[draw=drawColor,line width= 0.5pt,line join=round] ( 38.33, 85.85) --
	(139.54, 85.85);

\path[draw=drawColor,line width= 0.5pt,line join=round] ( 38.33,110.26) --
	(139.54,110.26);

\path[draw=drawColor,line width= 0.5pt,line join=round] ( 38.33,134.66) --
	(139.54,134.66);

\path[draw=drawColor,line width= 0.5pt,line join=round] (132.38, 32.16) --
	(132.38,139.54);

\path[draw=drawColor,line width= 0.5pt,line join=round] (106.82, 32.16) --
	(106.82,139.54);

\path[draw=drawColor,line width= 0.5pt,line join=round] ( 81.27, 32.16) --
	( 81.27,139.54);

\path[draw=drawColor,line width= 0.5pt,line join=round] ( 55.71, 32.16) --
	( 55.71,139.54);
\definecolor{drawColor}{RGB}{69,117,180}

\path[draw=drawColor,line width= 1.4pt,line join=round] ( 42.93, 41.92) --
	( 42.93, 41.92) --
	( 42.93, 45.83) --
	( 42.93, 45.83) --
	( 42.93, 47.78) --
	( 42.93, 47.78) --
	( 42.93, 48.76) --
	( 42.93, 48.76) --
	( 42.93, 50.71) --
	( 42.93, 50.71) --
	( 42.93, 52.66) --
	( 42.93, 52.66) --
	( 42.93, 53.64) --
	( 42.95, 53.64) --
	( 42.95, 54.61) --
	( 42.95, 54.61) --
	( 42.95, 55.59) --
	( 42.97, 55.59) --
	( 42.97, 56.57) --
	( 42.97, 56.57) --
	( 42.97, 57.54) --
	( 43.10, 57.54) --
	( 43.10, 58.52) --
	( 43.14, 58.52) --
	( 43.14, 59.49) --
	( 43.70, 59.49) --
	( 43.70, 60.47) --
	( 43.78, 60.47) --
	( 43.78, 61.45) --
	( 43.89, 61.45) --
	( 43.89, 62.42) --
	( 44.11, 62.42) --
	( 44.11, 63.40) --
	( 44.94, 63.40) --
	( 44.94, 64.38) --
	( 45.69, 64.38) --
	( 45.69, 65.35) --
	( 46.03, 65.35) --
	( 46.03, 66.33) --
	( 50.40, 66.33) --
	( 50.40, 67.30) --
	( 57.15, 67.30) --
	( 57.15, 68.28) --
	( 75.24, 68.28) --
	( 75.24, 69.26) --
	( 75.55, 69.26) --
	( 75.55, 70.23) --
	( 80.42, 70.23) --
	( 80.42, 71.21) --
	( 81.51, 71.21) --
	( 81.51, 72.18) --
	(134.94, 72.18) --
	(134.94, 72.18);
\definecolor{drawColor}{RGB}{252,141,89}

\path[draw=drawColor,line width= 1.4pt,dash pattern=on 2pt off 2pt ,line join=round] ( 42.94, 38.02) --
	( 42.94, 38.02) --
	( 42.94, 39.00) --
	( 42.94, 39.00) --
	( 42.94, 39.97) --
	( 42.94, 39.97) --
	( 42.94, 40.95) --
	( 42.94, 40.95) --
	( 42.94, 41.92) --
	( 42.94, 41.92) --
	( 42.94, 42.90) --
	( 42.94, 42.90) --
	( 42.94, 43.88) --
	( 42.94, 43.88) --
	( 42.94, 44.85) --
	( 42.94, 44.85) --
	( 42.94, 45.83) --
	( 42.94, 45.83) --
	( 42.94, 46.80) --
	( 42.94, 46.80) --
	( 42.94, 47.78) --
	( 42.94, 47.78) --
	( 42.94, 48.76) --
	( 42.94, 48.76) --
	( 42.94, 49.73) --
	( 42.95, 49.73) --
	( 42.95, 50.71) --
	( 42.95, 50.71) --
	( 42.95, 51.69) --
	( 42.95, 51.69) --
	( 42.95, 52.66) --
	( 43.05, 52.66) --
	( 43.05, 53.64) --
	( 43.10, 53.64) --
	( 43.10, 54.61) --
	( 43.13, 54.61) --
	( 43.13, 55.59) --
	( 43.17, 55.59) --
	( 43.17, 56.57) --
	( 43.49, 56.57) --
	( 43.49, 57.54) --
	( 44.86, 57.54) --
	( 44.86, 58.52) --
	( 45.36, 58.52) --
	( 45.36, 59.49) --
	( 71.92, 59.49) --
	( 71.92, 60.47) --
	( 79.69, 60.47) --
	( 79.69, 61.45) --
	(134.94, 61.45) --
	(134.94, 61.45);
\definecolor{drawColor}{RGB}{215,25,28}

\path[draw=drawColor,line width= 1.4pt,dash pattern=on 4pt off 2pt ,line join=round] ( 42.94, 38.02) --
	( 42.94, 38.02) --
	( 42.94, 39.00) --
	( 42.94, 39.00) --
	( 42.94, 39.97) --
	( 42.94, 39.97) --
	( 42.94, 40.95) --
	( 42.94, 40.95) --
	( 42.94, 41.92) --
	( 42.94, 41.92) --
	( 42.94, 42.90) --
	( 42.94, 42.90) --
	( 42.94, 43.88) --
	( 42.95, 43.88) --
	( 42.95, 44.85) --
	( 42.95, 44.85) --
	( 42.95, 45.83) --
	( 42.95, 45.83) --
	( 42.95, 46.80) --
	( 42.95, 46.80) --
	( 42.95, 47.78) --
	( 42.96, 47.78) --
	( 42.96, 48.76) --
	( 43.01, 48.76) --
	( 43.01, 49.73) --
	( 43.15, 49.73) --
	( 43.15, 50.71) --
	( 43.22, 50.71) --
	( 43.22, 51.69) --
	( 52.29, 51.69) --
	( 52.29, 52.66) --
	(116.21, 52.66) --
	(116.21, 53.64) --
	(134.94, 53.64) --
	(134.94, 53.64);
\definecolor{drawColor}{RGB}{69,117,180}

\node[text=drawColor,anchor=base west,inner sep=0pt, outer sep=0pt, scale=  1.00] at ( 40,110) {36/100};
\definecolor{drawColor}{RGB}{252,141,89}

\node[text=drawColor,anchor=base west,inner sep=0pt, outer sep=0pt, scale=  1.00] at ( 40,100) {25/100};
\definecolor{drawColor}{RGB}{215,25,28}

\node[text=drawColor,anchor=base west,inner sep=0pt, outer sep=0pt, scale=  1.00] at ( 40,90) {17/100};
\end{scope}
\begin{scope}
\path[clip] (  0.00,  0.00) rectangle (144.54,144.54);
\definecolor{drawColor}{gray}{0.30}

\node[text=drawColor,anchor=base east,inner sep=0pt, outer sep=0pt, scale=  1.00] at ( 33.83, 33.60) {0};

\node[text=drawColor,anchor=base east,inner sep=0pt, outer sep=0pt, scale=  1.00] at ( 33.83, 58.00) {25};

\node[text=drawColor,anchor=base east,inner sep=0pt, outer sep=0pt, scale=  1.00] at ( 33.83, 82.41) {50};

\node[text=drawColor,anchor=base east,inner sep=0pt, outer sep=0pt, scale=  1.00] at ( 33.83,106.81) {75};

\node[text=drawColor,anchor=base east,inner sep=0pt, outer sep=0pt, scale=  1.00] at ( 33.83,131.22) {100};
\end{scope}
\begin{scope}
\path[clip] (  0.00,  0.00) rectangle (144.54,144.54);
\definecolor{drawColor}{gray}{0.30}

\node[text=drawColor,anchor=base,inner sep=0pt, outer sep=0pt, scale=  1.00] at (132.38, 20.78) {7000};

\node[text=drawColor,anchor=base,inner sep=0pt, outer sep=0pt, scale=  1.00] at (106.82, 20.78) {5000};

\node[text=drawColor,anchor=base,inner sep=0pt, outer sep=0pt, scale=  1.00] at ( 81.27, 20.78) {3000};

\node[text=drawColor,anchor=base,inner sep=0pt, outer sep=0pt, scale=  1.00] at ( 55.71, 20.78) {1000};
\end{scope}
\begin{scope}
\path[clip] (  0.00,  0.00) rectangle (144.54,144.54);
\definecolor{drawColor}{RGB}{0,0,0}

\node[text=drawColor,anchor=base,inner sep=0pt, outer sep=0pt, scale=  1.00] at ( 88.93,  6.94) {Computation time (s)};
\end{scope}
\end{tikzpicture}}%
  } \caption{ECDF of computation time for the shortest-path problem (left) and
  the knapsack problem (right) with discrete budgeted uncertainty sets. The
  numbers in the legend indicate the number of instances solved within the time
  limit of 2 hours out of 100 instances. Solid \blue{blue}: \blue{\textsf{MibS}}.
  Dotted \orange{orange}: \orange{\textsf{qip-existeval}}. Dashed \red{red}: \red{\textsf{qip-bilevel}}.}
  \label{fig:num-results:discrete:budgeted:ecdf-time}
\end{figure}

Figure~\ref{fig:num-results:discrete:budgeted:ecdf-time} depicts the ECDF of
computation times for the shortest-path problem on the left and the knapsack
problem on the right. For the shortest-path problem, \textsf{MibS} is able to
solve $\rev{95}$ instances and thus outperforms both QIP-based approaches,
solving only $\rev{29}$ and $\rev{32}$ instances using the \textsf{Yasol}
solver, respectively. \rev{For the knapsack instances, \textsf{MibS} again
performs best, although the advantage is smaller than for the shortest-path
problem. It solves 36 instances, while \textsf{Yasol} solves 25 and 17 instances
for the two reformulations, respectively.} Notably, the \textsf{qip-bilevel}
reformulation is better than the \textsf{qip-existeval} reformulation for the
shortest-path problem, while it is worse for the knapsack problem. The ECDF of
the number of branch-and-bound nodes depicted in
Figure~\ref{fig:num-results:discrete:ecdf-nodes} \rev{in
Appendix~\ref{sec:appendix:plots}} confirms these results, showing that
\textsf{MibS} produces smaller search trees than the QIP-based approaches.

Compared to the continuous case, the instance sizes are significantly smaller.
The largest solved knapsack instance contains $\rev{50}$ items for the bilevel
approach, while the largest solved shortest-path instance contains \rev{$20$}
nodes for \textsf{MibS}. \textsf{Yasol} is only able to solve shortest-path
instances with at most $\rev{8}$ nodes in the graph. For more details, see
Appendix~\ref{sec:appendix:plots}. While especially the sizes of the
shortest-path instances seem to be small, we highlight that the number of
variables as well as the number of constraints are in the order of $|V|^2$. For
example, the largest instances which can be solved with \textsf{MibS} within the
time limit have ``only'' \rev{20} nodes but the bilevel reformulation has
\rev{608} upper- and lower-level variables as well as \rev{40} and \rev{1217}
upper- and lower-level constraints, respectively. We recall that the large
number of constraints in the lower level are due to the McCormick inequalities
used to linearize the products between the upper- and the lower-level variables;
see Appendix~\ref{sec:appendix:mccormick}.

\FloatBarrier

\subsubsection{Knapsack Uncertainty}
\label{sec:num-results:discrete:knapsack}
For the last set of numerical experiments, we consider the discrete version of
the knapsack uncertainty set for the shortest-path problem and the knapsack
problem. The complete derivation of the reformulations can be found in
Appendix~\ref{sec:appendix:applications:shortest-path:knapsack} for the
shortest-path problem and in
Appendix~\ref{sec:appendix:applications:knapsack:knapsack} for the knapsack
problem. For the QIP reformulations, we again refer to
Appendix~\ref{sec:qip-models}.

\begin{figure}
  \centering
  \makebox[\textwidth][c]{%
    \scalebox{1.0}{\input{plots/optl_rev/sp_disc_knap_ecdf_time.tex}}%
    \scalebox{1.0}{\begin{tikzpicture}[x=1pt,y=1pt]
\definecolor{fillColor}{RGB}{255,255,255}
\path[use as bounding box,fill=fillColor,fill opacity=0.00] (0,0) rectangle (144.54,144.54);
\begin{scope}
\path[clip] ( 38.33, 32.16) rectangle (139.54,139.54);
\definecolor{drawColor}{gray}{0.92}

\path[draw=drawColor,line width= 0.3pt,line join=round] ( 38.33, 49.25) --
	(139.54, 49.25);

\path[draw=drawColor,line width= 0.3pt,line join=round] ( 38.33, 73.65) --
	(139.54, 73.65);

\path[draw=drawColor,line width= 0.3pt,line join=round] ( 38.33, 98.05) --
	(139.54, 98.05);

\path[draw=drawColor,line width= 0.3pt,line join=round] ( 38.33,122.46) --
	(139.54,122.46);

\path[draw=drawColor,line width= 0.3pt,line join=round] (119.60, 32.16) --
	(119.60,139.54);

\path[draw=drawColor,line width= 0.3pt,line join=round] ( 94.05, 32.16) --
	( 94.05,139.54);

\path[draw=drawColor,line width= 0.3pt,line join=round] ( 68.49, 32.16) --
	( 68.49,139.54);

\path[draw=drawColor,line width= 0.5pt,line join=round] ( 38.33, 37.04) --
	(139.54, 37.04);

\path[draw=drawColor,line width= 0.5pt,line join=round] ( 38.33, 61.45) --
	(139.54, 61.45);

\path[draw=drawColor,line width= 0.5pt,line join=round] ( 38.33, 85.85) --
	(139.54, 85.85);

\path[draw=drawColor,line width= 0.5pt,line join=round] ( 38.33,110.26) --
	(139.54,110.26);

\path[draw=drawColor,line width= 0.5pt,line join=round] ( 38.33,134.66) --
	(139.54,134.66);

\path[draw=drawColor,line width= 0.5pt,line join=round] (132.38, 32.16) --
	(132.38,139.54);

\path[draw=drawColor,line width= 0.5pt,line join=round] (106.82, 32.16) --
	(106.82,139.54);

\path[draw=drawColor,line width= 0.5pt,line join=round] ( 81.27, 32.16) --
	( 81.27,139.54);

\path[draw=drawColor,line width= 0.5pt,line join=round] ( 55.71, 32.16) --
	( 55.71,139.54);
\definecolor{drawColor}{RGB}{69,117,180}

\path[draw=drawColor,line width= 1.4pt,line join=round] ( 42.93, 38.02) --
	( 42.93, 38.02) --
	( 42.93, 39.97) --
	( 42.93, 39.97) --
	( 42.93, 40.95) --
	( 42.93, 40.95) --
	( 42.93, 41.92) --
	( 42.93, 41.92) --
	( 42.93, 43.88) --
	( 42.93, 43.88) --
	( 42.93, 44.85) --
	( 42.93, 44.85) --
	( 42.93, 45.83) --
	( 42.93, 45.83) --
	( 42.93, 46.80) --
	( 42.93, 46.80) --
	( 42.93, 47.78) --
	( 42.93, 47.78) --
	( 42.93, 49.73) --
	( 42.93, 49.73) --
	( 42.93, 50.71) --
	( 42.93, 50.71) --
	( 42.93, 51.69) --
	( 42.93, 51.69) --
	( 42.93, 52.66) --
	( 42.94, 52.66) --
	( 42.94, 53.64) --
	( 42.94, 53.64) --
	( 42.94, 54.61) --
	( 42.94, 54.61) --
	( 42.94, 55.59) --
	( 42.94, 55.59) --
	( 42.94, 56.57) --
	( 42.94, 56.57) --
	( 42.94, 57.54) --
	( 42.94, 57.54) --
	( 42.94, 58.52) --
	( 42.95, 58.52) --
	( 42.95, 59.49) --
	( 42.95, 59.49) --
	( 42.95, 60.47) --
	( 42.95, 60.47) --
	( 42.95, 61.45) --
	( 42.95, 61.45) --
	( 42.95, 62.42) --
	( 42.95, 62.42) --
	( 42.95, 63.40) --
	( 42.96, 63.40) --
	( 42.96, 64.38) --
	( 42.96, 64.38) --
	( 42.96, 65.35) --
	( 42.97, 65.35) --
	( 42.97, 66.33) --
	( 42.97, 66.33) --
	( 42.97, 67.30) --
	( 42.98, 67.30) --
	( 42.98, 68.28) --
	( 43.00, 68.28) --
	( 43.00, 69.26) --
	( 43.01, 69.26) --
	( 43.01, 70.23) --
	( 43.02, 70.23) --
	( 43.02, 71.21) --
	( 43.02, 71.21) --
	( 43.02, 72.18) --
	( 43.03, 72.18) --
	( 43.03, 73.16) --
	( 43.05, 73.16) --
	( 43.05, 74.14) --
	( 43.09, 74.14) --
	( 43.09, 75.11) --
	( 43.14, 75.11) --
	( 43.14, 76.09) --
	( 43.20, 76.09) --
	( 43.20, 77.07) --
	( 43.23, 77.07) --
	( 43.23, 78.04) --
	( 43.25, 78.04) --
	( 43.25, 79.02) --
	( 43.27, 79.02) --
	( 43.27, 79.99) --
	( 43.31, 79.99) --
	( 43.31, 80.97) --
	( 43.36, 80.97) --
	( 43.36, 81.95) --
	( 43.41, 81.95) --
	( 43.41, 82.92) --
	( 43.44, 82.92) --
	( 43.44, 83.90) --
	( 43.46, 83.90) --
	( 43.46, 84.88) --
	( 43.46, 84.88) --
	( 43.46, 85.85) --
	( 43.53, 85.85) --
	( 43.53, 86.83) --
	( 43.59, 86.83) --
	( 43.59, 87.80) --
	( 44.09, 87.80) --
	( 44.09, 88.78) --
	( 44.39, 88.78) --
	( 44.39, 89.76) --
	( 44.58, 89.76) --
	( 44.58, 91.71) --
	( 44.65, 91.71) --
	( 44.65, 92.68) --
	( 45.59, 92.68) --
	( 45.59, 93.66) --
	( 45.71, 93.66) --
	( 45.71, 94.64) --
	( 45.88, 94.64) --
	( 45.88, 95.61) --
	( 47.86, 95.61) --
	( 47.86, 96.59) --
	( 47.88, 96.59) --
	( 47.88, 97.57) --
	( 48.65, 97.57) --
	( 48.65, 98.54) --
	( 48.73, 98.54) --
	( 48.73, 99.52) --
	( 49.36, 99.52) --
	( 49.36,100.49) --
	( 49.63,100.49) --
	( 49.63,101.47) --
	( 50.20,101.47) --
	( 50.20,102.45) --
	( 50.42,102.45) --
	( 50.42,103.42) --
	( 55.82,103.42) --
	( 55.82,104.40) --
	( 58.67,104.40) --
	( 58.67,105.37) --
	( 63.01,105.37) --
	( 63.01,106.35) --
	( 74.18,106.35) --
	( 74.18,107.33) --
	( 74.41,107.33) --
	( 74.41,108.30) --
	( 76.73,108.30) --
	( 76.73,109.28) --
	( 83.10,109.28) --
	( 83.10,110.26) --
	( 93.67,110.26) --
	( 93.67,111.23) --
	(100.48,111.23) --
	(100.48,112.21) --
	(106.87,112.21) --
	(106.87,113.18) --
	(132.95,113.18) --
	(132.95,114.16) --
	(134.94,114.16) --
	(134.94,114.16);
\definecolor{drawColor}{RGB}{252,141,89}

\path[draw=drawColor,line width= 1.4pt,dash pattern=on 2pt off 2pt ,line join=round] ( 42.96, 38.02) --
	( 42.98, 38.02) --
	( 42.98, 39.00) --
	( 43.02, 39.00) --
	( 43.02, 39.97) --
	( 43.02, 39.97) --
	( 43.02, 40.95) --
	( 43.08, 40.95) --
	( 43.08, 41.92) --
	( 43.10, 41.92) --
	( 43.10, 42.90) --
	( 43.13, 42.90) --
	( 43.13, 43.88) --
	( 43.15, 43.88) --
	( 43.15, 44.85) --
	( 43.19, 44.85) --
	( 43.19, 45.83) --
	( 43.19, 45.83) --
	( 43.19, 46.80) --
	( 43.22, 46.80) --
	( 43.22, 47.78) --
	( 43.33, 47.78) --
	( 43.33, 48.76) --
	( 43.36, 48.76) --
	( 43.36, 49.73) --
	( 43.37, 49.73) --
	( 43.37, 50.71) --
	( 43.38, 50.71) --
	( 43.38, 51.69) --
	( 43.54, 51.69) --
	( 43.54, 52.66) --
	( 43.55, 52.66) --
	( 43.55, 53.64) --
	( 43.81, 53.64) --
	( 43.81, 54.61) --
	( 43.86, 54.61) --
	( 43.86, 55.59) --
	( 44.14, 55.59) --
	( 44.14, 56.57) --
	( 44.15, 56.57) --
	( 44.15, 57.54) --
	( 44.25, 57.54) --
	( 44.25, 58.52) --
	( 44.26, 58.52) --
	( 44.26, 59.49) --
	( 44.41, 59.49) --
	( 44.41, 60.47) --
	( 44.88, 60.47) --
	( 44.88, 61.45) --
	( 45.54, 61.45) --
	( 45.54, 62.42) --
	( 45.58, 62.42) --
	( 45.58, 63.40) --
	( 45.65, 63.40) --
	( 45.65, 64.38) --
	( 46.18, 64.38) --
	( 46.18, 65.35) --
	( 47.07, 65.35) --
	( 47.07, 66.33) --
	( 47.10, 66.33) --
	( 47.10, 67.30) --
	( 47.15, 67.30) --
	( 47.15, 68.28) --
	( 49.19, 68.28) --
	( 49.19, 69.26) --
	( 51.75, 69.26) --
	( 51.75, 70.23) --
	( 54.09, 70.23) --
	( 54.09, 71.21) --
	( 58.84, 71.21) --
	( 58.84, 72.18) --
	( 59.25, 72.18) --
	( 59.25, 73.16) --
	( 59.90, 73.16) --
	( 59.90, 74.14) --
	( 61.15, 74.14) --
	( 61.15, 75.11) --
	( 62.25, 75.11) --
	( 62.25, 76.09) --
	( 69.13, 76.09) --
	( 69.13, 77.07) --
	( 69.78, 77.07) --
	( 69.78, 78.04) --
	( 71.02, 78.04) --
	( 71.02, 79.02) --
	( 92.35, 79.02) --
	( 92.35, 79.99) --
	( 93.44, 79.99) --
	( 93.44, 80.97) --
	( 93.81, 80.97) --
	( 93.81, 81.95) --
	(106.01, 81.95) --
	(106.01, 82.92) --
	(120.57, 82.92) --
	(120.57, 83.90) --
	(133.95, 83.90) --
	(133.95, 84.88) --
	(134.94, 84.88) --
	(134.94, 84.88);
\definecolor{drawColor}{RGB}{215,25,28}

\path[draw=drawColor,line width= 1.4pt,dash pattern=on 4pt off 2pt ,line join=round] ( 43.04, 38.02) --
	( 44.99, 38.02) --
	( 44.99, 39.00) --
	( 49.87, 39.00) --
	( 49.87, 39.97) --
	( 53.22, 39.97) --
	( 53.22, 40.95) --
	( 61.88, 40.95) --
	( 61.88, 41.92) --
	(134.94, 41.92) --
	(134.94, 41.92);
\definecolor{drawColor}{RGB}{69,117,180}

\node[text=drawColor,anchor=base west,inner sep=0pt, outer sep=0pt, scale=  1.00] at ( 100, 70) {79/100};
\definecolor{drawColor}{RGB}{252,141,89}

\node[text=drawColor,anchor=base west,inner sep=0pt, outer sep=0pt, scale=  1.00] at ( 100, 60) {49/100};
\definecolor{drawColor}{RGB}{215,25,28}

\node[text=drawColor,anchor=base west,inner sep=0pt, outer sep=0pt, scale=  1.00] at ( 105,50) {5/100};
\end{scope}
\begin{scope}
\path[clip] (  0.00,  0.00) rectangle (144.54,144.54);
\definecolor{drawColor}{gray}{0.30}

\node[text=drawColor,anchor=base east,inner sep=0pt, outer sep=0pt, scale=  1.00] at ( 33.83, 33.60) {0};

\node[text=drawColor,anchor=base east,inner sep=0pt, outer sep=0pt, scale=  1.00] at ( 33.83, 58.00) {25};

\node[text=drawColor,anchor=base east,inner sep=0pt, outer sep=0pt, scale=  1.00] at ( 33.83, 82.41) {50};

\node[text=drawColor,anchor=base east,inner sep=0pt, outer sep=0pt, scale=  1.00] at ( 33.83,106.81) {75};

\node[text=drawColor,anchor=base east,inner sep=0pt, outer sep=0pt, scale=  1.00] at ( 33.83,131.22) {100};
\end{scope}
\begin{scope}
\path[clip] (  0.00,  0.00) rectangle (144.54,144.54);
\definecolor{drawColor}{gray}{0.30}

\node[text=drawColor,anchor=base,inner sep=0pt, outer sep=0pt, scale=  1.00] at (132.38, 20.78) {7000};

\node[text=drawColor,anchor=base,inner sep=0pt, outer sep=0pt, scale=  1.00] at (106.82, 20.78) {5000};

\node[text=drawColor,anchor=base,inner sep=0pt, outer sep=0pt, scale=  1.00] at ( 81.27, 20.78) {3000};

\node[text=drawColor,anchor=base,inner sep=0pt, outer sep=0pt, scale=  1.00] at ( 55.71, 20.78) {1000};
\end{scope}
\begin{scope}
\path[clip] (  0.00,  0.00) rectangle (144.54,144.54);
\definecolor{drawColor}{RGB}{0,0,0}

\node[text=drawColor,anchor=base,inner sep=0pt, outer sep=0pt, scale=  1.00] at ( 88.93,  6.94) {Computation time (s)};
\end{scope}
\begin{scope}
\path[clip] (  0.00,  0.00) rectangle (144.54,144.54);
\definecolor{drawColor}{RGB}{0,0,0}

\end{scope}
\end{tikzpicture}}%
  } \caption{ECDF of computation time for the shortest-path problem (left) and
  the knapsack problem (right) with discrete knapsack uncertainty sets. The
  numbers in the legend indicate the number of instances solved within the time
  limit of 2 hours out of 100 instances. Solid \blue{blue}: \blue{\textsf{MibS}}.
  Dotted \orange{orange}: \orange{\textsf{qip-existeval}}. Dashed \red{red}: \red{\textsf{qip-bilevel}}.}
  \label{fig:num-results:discrete:knapsack:ecdf-time}
\end{figure}

In Figure~\ref{fig:num-results:discrete:knapsack:ecdf-time}, we depict the ECDF
of computation times for the shortest-path problem (left) and the knapsack
problem (right). For the shortest-path problem, \textsf{MibS} performs better
than the QIP-based approaches, solving \rev{almost} all instances of the test
set, while the \textsf{qip-existeval} reformulation \rev{solves $82$ instances
} and the \textsf{qip-bilevel} reformulation \rev{solves only $27$
instances}. For the knapsack problem, \textsf{MibS} solves $30$ instances more
than the \textsf{qip-existeval} reformulation, while the \textsf{qip-bilevel}
reformulation only solves $5$ instances in total. The ECDF of the number of
branch-and-bound nodes depicted in
Figure~\ref{fig:num-results:discrete:ecdf-nodes} \rev{in
Appendix~\ref{sec:appendix:plots}} confirms these observations, showing that
\textsf{MibS} produces smaller search trees than the QIP-based approaches for
both applications.

The sizes of the instances are comparable to the ones with discrete budgeted
uncertainty, with the largest solved knapsack instance containing $120$ items
for \textsf{MibS} and $90$ items for \textsf{Yasol}. For the shortest-path
problem, \textsf{Yasol} is only able to solve \rev{$1$~instance with $20$ nodes,
whereas \textsf{MibS} solved $6$ instances of the same size. The
\textsf{qip-bilevel} formulation even was only able to solve instances with up
to $6$ nodes.} Again, we note that the number of variables and constraints in
the bilevel reformulation of the shortest-path problem is in the order of
$|V|^2$; see Section~\ref{sec:num-results:discrete:budgeted}. Detailed scatter
plots for different sizes of instances can again be found in
Appendix~\ref{sec:appendix:plots}.

\subsubsection{Conclusion}
\label{sec:num-results:discrete:conclusion}

For DDRO problems with discrete uncertainty sets, the bilevel approach using an
MIBLP solver like \textsf{MibS} and the QIP-based approaches using a solver like
\textsf{Yasol} are among the first methods to ever solve problems of this type.
For the considered applications, \textsf{MibS} generally outperforms
\textsf{Yasol} in terms of computation time and number of branch-and-bound
nodes. Note that this is to be expected, since \textsf{MibS} is a
state-of-the-art solver for mixed-integer linear bilevel problems, whereas
\textsf{Yasol} targets a broader range of quantified integer problems and is not
specifically designed for solving this kind of problems. It is also noteworthy
that the performance of the two QIP-based approaches \rev{for budgeted
uncertainty sets} varies depending on the application, which may indicate that
there is still potential for improvement with more advanced modeling techniques.

Overall, the results show that solving DDRO problems with discrete uncertainty
sets is computationally challenging, even for relatively small instance sizes.
However, the results also demonstrate that the bilevel approach using a
dedicated MIBLP solver can be a powerful tool for solving such problems,
outperforming more general QIP-based approaches in the considered applications.

\ifLongVersion
\section{Conclusion}
\label{sec:conclusion}

Bilevel and robust optimization problems have a rather similar
mathematical structure.
However, only until recently, these similarities have not been studied
and the two respective communities have not been in scientific contact
a lot.
While the connection has been, to the best of our knowledge,
observed by \textcite{Stein:2013} in the context of (generalized)
semi-infinite optimization, the first systematic analysis of the
structural similarities and differences has only recently been
published by \textcite{goerigk2025}.

Using the results from the latter paper, we are the first
exploiting the equivalence of decision-dependent robust optimization and bilevel
optimization in an extensive computational study.
The computational results demonstrate the potential of solving DDRO problems as
bilevel optimization problems, particularly for mixed-integer uncertainty sets.

First, we consider different classic
robust optimization problems with decision-dependent uncertainty sets, which are
given in a continuous and convex way so that classic dualization tricks of
robust optimization can be applied. In these cases, the respective single-level
reformulations of the corresponding bilevel problems are similar but larger,
which usually leads to larger computation times. Second, we also
consider decision-dependent uncertainty sets that cannot be treated via
dualization because they are represented as mixed-integer linear problems, for
which no strong-duality theorem is available in general.
For these cases, we compare the bilevel approach, which leads to
mixed-integer linear bilevel problems, with the
\textsf{Yasol} solver for quantified mixed-integer problems
\parencite{yasol2017}.
While the novel possibility of solving robust
problems as bilevel ones outperforms the quantified mixed-integer
solver, the numerical results also show that, even with the bilevel
approach, only rather small-scale decision-dependent robust problems
can be solved.
Hence, there is quite some room for future research at the interface
of robust and bilevel optimization.
We hope that this paper paves the way for such future contributions.

\fi
\section*{Acknowledgements}

We acknowledge the support by the German
Bundesministerium f\"ur Bildung und Forschung within
the project \enquote{RODES} (F\"orderkennzeichen 05M22UTB).
Moreover, the authors thank the DFG for their support within
the research training group~2126 \enquote{Algorithmic Optimization}.
The computations were executed on the high performance cluster
``Elwetritsch'' at the TU Kaiserslautern, which is part of the
``Alliance of High Performance Computing Rheinland-Pfalz'' (AHRP).
We kindly acknowledge the support of RHRK.
The authors acknowledge the use of DeepL and OpenAI's ChatGPT for
partly editing and polishing the text and figures for spelling,
grammar, and stylistic improvements.
Additionally, ChatGPT was utilized for support in basic coding tasks.
Finally, we are very grateful to Marc Goerigk and Michael
Hartisch, who made us aware of \textsf{Yasol} and greatly helped us
using it.

\section*{Compliance with Ethical Standards}

The authors have no relevant financial or non-financial interests to
disclose.

\section*{Data Availability Statement}
\label{sec:data-avail-stat}

This manuscript has associated data in a data repository
(\url{https://github.com/simstevens/ddro-via-bilevel}), which
includes all code and data that is required to reproduce the numerical
results reported in the manuscript.

\printbibliography
\appendix
\section{Detailed Derivation of the Reformulations for all Applications}
\label{sec:appendix:applications}

We illustrate the reformulation techniques for decision-dependent robust
optimization on the three considered applications. This section follows a common
pattern: for each application, we restate the DDRO model with budgeted and
knapsack uncertainty, present the associated bilevel formulation, and list both
the classic robust single-level reformulation and the strong duality-based
single-level reformulation of the bilevel problem. We begin with the
shortest-path problem (Section~\ref{sec:appendix:applications:shortest-path}),
then turn to the knapsack problem
(Section~\ref{sec:appendix:applications:knapsack}), and close with the portfolio
optimization problem (Section~\ref{sec:appendix:applications:portfolio}). We omit the
reformulation using KKT conditions because early preliminary experiments showed
that this approach is not competitive w.r.t.\ the strong duality-based
reformulation.

\subsection{The Shortest-Path Problem}
\label{sec:appendix:applications:shortest-path}

We start with the decision-dependent robust shortest-path problem described in
Section~\ref{sec:applications:shortest-path} and state the reformulations for
the cases of budgeted and knapsack uncertainty.

\subsubsection{Budgeted Uncertainty}
\label{sec:appendix:applications:shortest-path:budgeted}
Recall the decision-dependent robust shortest-path problem
\eqref{eq:applications:shortest-path:cont-dd-rob-model} from
Section~\ref{sec:applications:shortest-path:budgeted-uncertainty} given as
\begin{subequations}
  \label{eq:appendix:applications:shortest-path:budgeted}
  \begin{align}
    \min_{x,y} \quad
    &\sum_{a \in A} h_a x_a
      + \sum_{a \in A} \bar{d}_a y_a
      + \max_{u \in U_{\text{sp}}(x)} \
      \sum_{a \in A} u_a\hat{d}_a y_a
      \label{eq:appendix:applications:shortest-path:budgeted:objective}
    \\
    \st \quad
    &\sum_{a \in \delta^{\text{in}}(v)} y_a
    - \sum_{a \in \delta^{\text{out}}(v)} y_a =
      \begin{cases}
        1, \ &v = t, \\\
        -1, \ &v = s, \\
        0, \ &\text{else},
      \end{cases}
      \ v \in V,
    \label{eq:appendix:applications:shortest-path:flow-conservation}
    \\
    &x, y \in \set{0,1}^{\Card{A}},
  \end{align}
\end{subequations}
with decision-dependent budgeted uncertainty set $U_{\text{sp}} = U^{\text{cb}}$
or \hbox{$U_{\text{sp}} = U^{\text{db}}$} as defined in
Section~\ref{sec:uncertainty-sets:budgeted}. According to
Section~\ref{sec:general-model}, this problem can be reformulated as the bilevel
optimization problem
\begin{equation}
  \label{eq:appendix:applications:shortest-path:budgeted:bilevel}
  \begin{aligned}
    \min_{x,y} \quad
    &\sum_{a \in A} h_a x_a
      + \sum_{a \in A} \bar{d}_a y_a
      + \sum_{a \in A} u_a\hat{d}_a y_a
    \\
    \st \quad
    &\eqref{eq:appendix:applications:shortest-path:flow-conservation}, \
    x, y \in \set{0,1}^{\Card{A}},
    \\
    &u \in S(y,x),
  \end{aligned}
\end{equation}
where $S(y,x)$ is the set of optimal solutions to the $(y,x)$-parameterized
lower-level problem
\begin{equation}
  \label{eq:appendix:applications:shortest-path:budgeted:bilevel:lower}
  \begin{aligned}
    \max_{u} \quad
    &\sum_{a \in A} u_a\hat{d}_a y_a
    \\
    \st \quad
    &\sum_{a \in A} u_a \leq \Gamma,
    \\
    &0 \leq u_a \leq 1 - \gamma_a x_a, \quad
      a \in A.
  \end{aligned}
\end{equation}
In case of discrete budgeted uncertainty, i.e., $U_{\text{sp}} = U^{\text{db}}$,
we additionally have the integrality constraints $u \in \set{0,1}^{\Card{A}}$ in
the lower level. Then, the bilevel
problem~\eqref{eq:appendix:applications:shortest-path:budgeted:bilevel} cannot
be reformulated as a single-level problem using duality theory. However, it can
be tackled by an MIBLP solver such as \textsf{MibS}.

Note that the lower-level
problem~\eqref{eq:appendix:applications:shortest-path:budgeted:bilevel:lower} is
exactly the inner maximization problem in the objective
function~\eqref{eq:appendix:applications:shortest-path:budgeted:objective}. In
case of continuous budgeted uncertainty, i.e., $U_{\text{sp}} = U^{\text{cb}}$,
we can dualize the lower-level
problem~\eqref{eq:appendix:applications:shortest-path:budgeted:bilevel:lower} to
obtain the dual problem
\begin{equation*}
  \label{eq:appendix:applications:shortest-path:budgeted:inner-dual}
  \begin{aligned}
    \min_{\pi, \lambda} \quad
    &\pi \Gamma + \sum_{a \in A} \lambda_a (1 - \gamma_a x_a)
    \\
    \st \quad
    &\pi + \lambda_a \geq \hat{d}_a y_a, \quad
      a \in A,
    \\
    &\pi, \lambda \geq 0.
  \end{aligned}
\end{equation*}
This can be used to derive the classic robust reformulation as
\begin{equation*}
  \label{eq:appendix:applications:shortest-path:budgeted:classic-robust}
  \begin{aligned}
    \min_{x,y,\pi,\lambda} \quad
    &\sum_{a \in A} h_a x_a
      + \sum_{a \in A} \bar{d}_a y_a
      + \pi \Gamma + \sum_{a \in A} \lambda_a (1 - \gamma_a x_a)
    \\
    \st \quad
    &\eqref{eq:appendix:applications:shortest-path:flow-conservation}, \
    x, y \in \set{0,1}^{\Card{A}},
    \\
    &\pi + \lambda_a \geq \hat{d}_a y_a, \quad
      a \in A,
    \\
    &\pi, \lambda \geq 0.
  \end{aligned}
\end{equation*}
Similarly, one can derive the single-level reformulation of the bilevel
problem~\eqref{eq:appendix:applications:shortest-path:budgeted:bilevel} using
strong duality to obtain
\begin{equation*}
  \label{eq:appendix:applications:shortest-path:budgeted:bilevel-sd}
  \begin{aligned}
    \min_{x,y,u,\pi,\lambda} \quad
    &\sum_{a \in A} h_a x_a
      + \sum_{a \in A} \bar{d}_a y_a
      + \sum_{a \in A} u_a\hat{d}_a y_a
    \\
    \st \quad
    &\eqref{eq:appendix:applications:shortest-path:flow-conservation}, \
    x, y \in \set{0,1}^{\Card{A}},
    \\
    &\sum_{a \in A} u_a \leq \Gamma,
    \\
    &0 \leq u_a \leq 1 - \gamma_a x_a, \quad
      a \in A,
    \\
    &\pi + \lambda_a \geq \hat{d}_a y_a, \quad
      a \in A,
    \\
    &\sum_{a \in A} u_a\hat{d}_a y_a \geq
      \pi \Gamma + \sum_{a \in A} \lambda_a (1 - \gamma_a x_a),
    \\
    &\pi, \lambda \geq 0.
  \end{aligned}
\end{equation*}
The bilinearities $\lambda_a x_a$ and $u_a y_a$ in all reformulations can be
linearized using McCormick envelopes. For a detailed discussion we refer to
Appendix~\ref{sec:appendix:mccormick}.

\subsubsection{Knapsack Uncertainty}
\label{sec:appendix:applications:shortest-path:knapsack}

Recall the decision-dependent robust shortest-path problem
\eqref{eq:applications:shortest-path:cont-dd-rob-model:knapsack} from
Section~\ref{sec:applications:shortest-path:knapsack-uncertainty} given as
\begin{equation*}
  \label{eq:appendix:applications:shortest-path:knapsack}
  \begin{aligned}
    \min_{y} \quad
    &\sum_{a \in A} \bar{d}_a y_a
      + \max_{u \in U_{\text{sp}}(y)} \
      \sum_{a \in A} u_a\hat{d}_a y_a
    \\
    \st \quad
    &\eqref{eq:appendix:applications:shortest-path:flow-conservation}, \
    y \in \set{0,1}^{\Card{A}},
  \end{aligned}
\end{equation*}
with decision-dependent knapsack uncertainty set $U_{\text{sp}} = U^{\text{ck}}$
or \hbox{$U_{\text{sp}} = U^{\text{dk}}$} as defined in
Section~\ref{sec:uncertainty-sets:knapsack}. This problem can be equivalently
reformulated as the bilevel optimization problem
\begin{equation}
  \label{eq:appendix:applications:shortest-path:knapsack:bilevel}
  \begin{aligned}
    \min_{y} \quad
    &\sum_{a \in A} \bar{d}_a y_a
      + \sum_{a \in A} u_a\hat{d}_a y_a
    \\
    \st \quad
    &\eqref{eq:appendix:applications:shortest-path:flow-conservation}, \
    y \in \set{0,1}^{\Card{A}},
    \\
    & u \in S(y),
  \end{aligned}
\end{equation}
where $S(y)$ is the set of optimal solutions to the $y$-parameterized
lower-level problem
\begin{equation}
  \label{eq:appendix:applications:shortest-path:knapsack:bilevel:lower}
  \begin{aligned}
    \max_{u} \quad
    &\sum_{a \in A} u_a\hat{d}_a y_a
    \\
    \st \quad
    &f^\top u \leq b + w^\top y,
    \\
    &0 \leq u_a \leq 1, \quad
      a \in A.
  \end{aligned}
\end{equation}
Again, in case of discrete knapsack uncertainty, i.e., $U_{\text{sp}} =
U^{\text{dk}}$, we additionally have the integrality constraints $u \in
\set{0,1}^{\Card{A}}$ in the lower level and can solve the bilevel problem using
\textsf{MibS}. For continuous knapsack uncertainty, i.e.,
\hbox{$U_{\text{sp}} = U^{\text{ck}}$}, we can dualize the lower-level
problem~\eqref{eq:appendix:applications:shortest-path:knapsack:bilevel:lower} to
obtain the dual
\begin{equation*}
  \label{eq:appendix:applications:shortest-path:knapsack:inner-dual}
  \begin{aligned}
    \min_{\pi, \lambda} \quad
    &\pi (b + w^\top y) + \sum_{a \in A} \lambda_a
      \\
    \st \quad
    &\pi f_a + \lambda_a \geq \hat{d}_a y_a, \quad
      a \in A,
    \\
    &\pi, \lambda \geq 0.
  \end{aligned}
\end{equation*}
This can be used to derive the classic robust reformulation as
\begin{equation*}
  \label{eq:appendix:applications:shortest-path:knapsack:classic-robust}
  \begin{aligned}
    \min_{y,\pi,\lambda} \quad
    &\sum_{a \in A} \bar{d}_a y_a
      + \pi (b + w^\top y) + \sum_{a \in A} \lambda_a
    \\
    \st \quad
    &\eqref{eq:appendix:applications:shortest-path:flow-conservation}, \
    y \in \set{0,1}^{\Card{A}},
    \\
    &\pi f_a + \lambda_a \geq \hat{d}_a y_a, \quad
      a \in A,
    \\
    &\pi, \lambda \geq 0.
  \end{aligned}
\end{equation*}
Similarly, the single-level reformulation of the bilevel
problem~\eqref{eq:appendix:applications:shortest-path:knapsack:bilevel} using
strong duality can be derived as
\begin{equation*}
  \label{eq:appendix:applications:shortest-path:knapsack:bilevel-sd}
  \begin{aligned}
    \min_{y,u,\pi,\lambda} \quad
    &\sum_{a \in A} \bar{d}_a y_a
      + \sum_{a \in A} u_a\hat{d}_a y_a
    \\
    \st \quad
    &\eqref{eq:appendix:applications:shortest-path:flow-conservation}, \
    y \in \set{0,1}^{\Card{A}},
    \\
    &f^\top u \leq b + w^\top y,
    \\
    &0 \leq u_a \leq 1, \quad
      a \in A,
    \\
    &\pi f_a + \lambda_a \geq \hat{d}_a y_a, \quad
      a \in A,
    \\
    &\sum_{a \in A} u_a\hat{d}_a y_a \geq
      \pi (b + w^\top y) + \sum_{a \in A} \lambda_a,
    \\
    &\pi, \lambda \geq 0.
  \end{aligned}
\end{equation*}
Again, the bilinearities $\pi y_a$ and $u_a y_a$ can be linearized using
McCormick envelopes.

\subsection{The Knapsack Problem}
\label{sec:appendix:applications:knapsack}
In this section, we summarize the reformulations of the decision-dependent robust knapsack
models from Section~\ref{sec:applications:knapsack} for budgeted and knapsack
uncertainty.

\subsubsection{Budgeted Uncertainty}
\label{sec:appendix:applications:knapsack:budgeted}

We consider the knapsack
problem~\eqref{eq:applications:knapsack:budgeted-dd-rob-model:budgeted} from
Section~\ref{sec:applications:knapsack:budgeted-uncertainty} given as
\begin{equation*}
  \label{eq:appendix:applications:knapsack:budgeted}
  \begin{aligned}
    \max_{x,y} \quad
    &c^\top y - h^\top x
    \\
    \st \quad
    &\sum_{i \in [n]} \bar{a}_i y_i
      + \max_{u \in U_\text{k}(x)} \
      \sum_{i \in [n]} u_i \hat{a}_i y_i \leq d,
    \\
    &x, y \in \set{0,1}^n,
  \end{aligned}
\end{equation*}
with decision-dependent budgeted uncertainty set $U_\text{k} = U^{\text{cb}}$ or
$U_\text{k} = U^{\text{db}}$ as defined in
Section~\ref{sec:uncertainty-sets:budgeted}. The equivalent bilevel formulation
is
\begin{equation*}
  \label{eq:appendix:applications:knapsack:budgeted:bilevel}
  \begin{aligned}
    \max_{x,y} \quad
    &c^\top y - h^\top x
    \\
    \st \quad
    &\sum_{i \in [n]} \bar{a}_i y_i +
      \sum_{i \in [n]} u_i \hat{a}_i y_i \leq d,
    \\
    &x, y \in \set{0,1}^n,
    \\
    &u \in S(y,x),
  \end{aligned}
\end{equation*}
where $S(y,x)$ is the set of optimal solutions to the $(y,x)$-parameterized
lower-level problem
\begin{equation*}
  \label{eq:appendix:applications:knapsack:budgeted:lower-level}
  \begin{aligned}
    \max_{u} \quad
    &\sum_{i \in [n]} u_i \hat{a}_i y_i
    \\
    \st \quad
    &\sum_{i \in [n]} u_i \leq \Gamma,
    \\
    &0 \leq u_i \leq 1 - \gamma_i x_i, \quad
      i \in [n].
  \end{aligned}
\end{equation*}
The discrete variant can be handled with \textsf{MibS}. For the continuous case, we
dualize the lower-level problem to obtain the dual problem
\begin{equation*}
  \label{eq:appendix:applications:knapsack:budgeted:inner-dual}
  \begin{aligned}
    \min_{\pi, \lambda} \quad
    &\Gamma \pi + \sum_{i \in [n]} (1 - \gamma_i x_i) \lambda_i
    \\
    \st \quad
    &\pi + \lambda_i \geq \hat{a}_i y_i, \quad
      i \in [n],
    \\
    &\pi, \lambda \geq 0.
  \end{aligned}
\end{equation*}
This yields the classic robust reformulation
\begin{equation*}
  \label{eq:appendix:applications:knapsack:budgeted:robust-reformulation}
  \begin{aligned}
    \max_{x,y,\pi,\lambda} \quad
    &c^\top y - h^\top x
    \\
    \st \quad
    &\sum_{i \in [n]} \bar{a}_i y_i +
      \Gamma \pi + \sum_{i \in [n]} (1 - \gamma_i x_i) \lambda_i \leq d,
    \\
    &\pi + \lambda_i \geq \hat{a}_i y_i, \quad
      i \in [n],
    \\
    &\pi, \lambda \geq 0,
    \\
    &x, y \in \set{0,1}^n,
  \end{aligned}
\end{equation*}
and the single-level reformulation of the bilevel problem
\begin{equation*}
  \label{eq:appendix:applications:knapsack:budgeted:single-level}
  \begin{aligned}
    \max_{x,y,u, \pi, \lambda} \quad
    &c^\top y - h^\top x
    \\
    \st \quad
    &\sum_{i \in [n]} \bar{a}_i y_i +
      \sum_{i \in [n]} u_i \hat{a}_i y_i \leq d,
    \\
    &\sum_{i \in [n]} u_i \leq \Gamma,
    \\
    &0 \leq u_i \leq 1 - \gamma_i x_i, \quad
      i \in [n],
    \\
    &\pi + \lambda_i \geq \hat{a}_i y_i, \quad
      i \in [n],
    \\
    &\sum_{i \in [n]} u_i \hat{a}_i y_i \geq
      \Gamma \pi + \sum_{i \in [n]} (1 - \gamma_i x_i) \lambda_i,
    \\
    &\pi, \lambda \geq 0,
    \\
    &x, y \in \set{0,1}^n.
  \end{aligned}
\end{equation*}

\subsubsection{Knapsack Uncertainty}
\label{sec:appendix:applications:knapsack:knapsack}

For knapsack uncertainty sets, the corresponding
problem~\eqref{eq:applications:knapsack:cont-dd-rob-model:knapsack} from
Section~\ref{sec:applications:knapsack:knapsack-uncertainty} is
\begin{equation*}
  \label{eq:appendix:applications:knapsack:knapsack}
  \begin{aligned}
    \max_{y} \quad
    &c^\top y
    \\
    \st \quad
    &\sum_{i \in [n]} \bar{a}_i y_i
      + \max_{u \in U_\text{k}(y)} \
      \sum_{i \in [n]} u_i \hat{a}_i y_i \leq d,
    \\
    &y \in \set{0,1}^n,
  \end{aligned}
\end{equation*}
with decision-dependent knapsack uncertainty set $U_\text{k} =
U^{\text{ck}}$ or $U_\text{k} = U^{\text{dk}}$ as defined in
Section~\ref{sec:uncertainty-sets:knapsack}. The equivalent bilevel problem is
\begin{equation*}
  \label{eq:appendix:applications:knapsack:knapsack:bilevel}
  \begin{aligned}
    \max_{y} \quad
    &c^\top y
    \\
    \st \quad
    &\sum_{i \in [n]} \bar{a}_i y_i
      + \sum_{i \in [n]} u_i \hat{a}_i y_i \leq d,
    \\
    &y \in \set{0,1}^n,
    \\
    &u \in S(y),
  \end{aligned}
\end{equation*}
where $S(y)$ denotes the set of optimal solutions to the $y$-parameterized lower-level
problem
\begin{equation}
  \label{eq:appendix:applications:knapsack:knapsack:lower-level}
  \begin{aligned}
    \max_{u} \quad
    &\sum_{i \in [n]} u_i \hat{a}_i y_i
    \\
    \st \quad
    &f^\top u \leq b + w^\top y,
    \\
    &0 \leq u_i \leq 1, \quad
      i \in [n].
  \end{aligned}
\end{equation}
Again, in case of discrete knapsack uncertainty, i.e., $U_\text{k} =
U^{\text{dk}}$, we additionally have the integrality constraints $u \in
\set{0,1}^n$ in the lower level and can solve the bilevel problem using
\textsf{MibS}. For continuous knapsack uncertainty, i.e., $U_\text{k} =
U^{\text{ck}}$, we can dualize the lower-level
problem~\eqref{eq:appendix:applications:knapsack:knapsack:lower-level} to obtain
\begin{equation*}
  \label{eq:appendix:applications:knapsack:knapsack:inner-dual}
  \begin{aligned}
    \min_{\pi, \lambda} \quad
    &\pi (b + w^\top y) + \sum_{i \in [n]} \lambda_i
    \\
    \st \quad
    &\pi f_i + \lambda_i \geq \hat{a}_i y_i, \quad
      i \in [n],
    \\
    &\pi, \lambda \geq 0.
  \end{aligned}
\end{equation*}
This leads to the classic robust reformulation
\begin{equation*}
  \label{eq:appendix:applications:knapsack:knapsack:robust-reformulation}
  \begin{aligned}
    \max_{y,\pi,\lambda} \quad
    &c^\top y
    \\
    \st \quad
    &\sum_{i \in [n]} \bar{a}_i y_i
      + \pi (b + w^\top y) + \sum_{i \in [n]} \lambda_i \leq d,
    \\
    &\pi f_i + \lambda_i \geq \hat{a}_i y_i, \quad
      i \in [n],
    \\
    &\pi, \lambda \geq 0,
    \\
    &y \in \set{0,1}^n,
  \end{aligned}
\end{equation*}
and the single-level reformulation of the bilevel problem is
\begin{equation*}
  \label{eq:appendix:applications:knapsack:knapsack:single-level}
  \begin{aligned}
    \max_{y,u, \pi, \lambda} \quad
    &c^\top y
    \\
    \st \quad
    &\sum_{i \in [n]} \bar{a}_i y_i +
      \sum_{i \in [n]} u_i \hat{a}_i y_i \leq d,
    \\
    &f^\top u \leq b + w^\top y,
    \\
    &0 \leq u_i \leq 1, \quad
      i \in [n],
    \\
    &\pi f_i + \lambda_i \geq \hat{a}_i y_i, \quad
      i \in [n],
    \\
    &\sum_{i \in [n]} u_i \hat{a}_i y_i \geq
      \pi (b + w^\top y) + \sum_{i \in [n]} \lambda_i,
    \\
    &\pi, \lambda \geq 0,
    \\
    &y \in \set{0,1}^n.
  \end{aligned}
\end{equation*}
Bilinear terms can be linearized with McCormick envelopes.

\subsection{The Portfolio Optimization Problem}
\label{sec:appendix:applications:portfolio}
Lastly, we restate the decision-dependent robust portfolio models from
Section~\ref{sec:applications:portfolio:budgeted-uncertainty} for budgeted and
knapsack uncertainty and collect their reformulations.

\subsubsection{Budgeted Uncertainty}
\label{sec:appendix:applications:portfolio:budgeted}
The robust portfolio optimization
problem~\eqref{eq:applications:portfolio:cont-dd-rob-model} reads
\begin{subequations}
  \label{eq:appendix:applications:portfolio:budgeted}
  \begin{align}
    \max_{y,s,x} \quad &\bar{\mu}^\top y - h^\top x -
    \max_{u \in U_\text{p}(x)} \
    \sum_{i \in [N]} u_i \hat{\mu}_i y_i
    \\
    \st \quad &y^\top \Sigma y \leq V_0, \quad \sum_{i \in [N]} y_i = 1,
    \label{eq:appendix:applications:portfolio:budgeted:c1}
    \\
    &\sum_{i \in [N]} s_i \leq k, \quad y_i \leq s_i, \quad i \in [N],
    \label{eq:appendix:applications:portfolio:budgeted:c2}
    \\
    &y, x \in [0,1]^N, \quad s \in \{0,1\}^N,
    \label{eq:appendix:applications:portfolio:budgeted:c3}
  \end{align}
\end{subequations}
with continuous decision-dependent budgeted uncertainty set $U_\text{p} =
U^\text{cb}$ as defined in Section~\ref{sec:uncertainty-sets:budgeted}. The
equivalent bilevel formulation reads
\begin{equation*}
  \label{eq:appendix:applications:portfolio:budgeted:bilevel}
  \begin{aligned}
    \max_{y,s,x,u} \quad &\bar{\mu}^\top y - h^\top x -
    \sum_{i \in [N]} u_i \hat{\mu}_i y_i
    \\
    \st \quad
    &\eqref{eq:appendix:applications:portfolio:budgeted:c1}\text{--}%
    \eqref{eq:appendix:applications:portfolio:budgeted:c3},
    \\
    &u \in S(y,x),
  \end{aligned}
\end{equation*}
where $S(y,x)$ denotes the set of optimal solutions to the $(y,x)$-parameterized
lower-level problem
\begin{equation}
  \label{eq:appendix:applications:portfolio:budgeted:lower-level}
  \begin{aligned}
    \max_{u} \quad &\sum_{i \in [N]} u_i \hat{\mu}_i y_i
    \\
    \st \quad &\sum_{i \in [N]} u_i \leq \Gamma,
    \\
    &0 \leq u_i \leq 1 - \gamma_i x_i, \quad i \in [N].
  \end{aligned}
\end{equation}
Due to the quadratic term in
Constraint~\eqref{eq:appendix:applications:portfolio:budgeted:c1}, the discrete
variant cannot be handled by MIBLP solvers such as \textsf{MibS}. For the continuous case
we dualize the lower-level
problem~\eqref{eq:appendix:applications:portfolio:budgeted:lower-level} to
obtain the dual problem
\begin{equation*}
  \label{eq:appendix:applications:portfolio:budgeted:dual-inner}
  \begin{aligned}
    \min_{\pi, \lambda} \quad
    &\Gamma \pi + \sum_{i \in [N]} (1 - \gamma_i x_i) \lambda_i
    \\
    \st \quad
    &\pi + \lambda_i \geq \hat{\mu}_i y_i, \quad
      i \in [N],
    \\
    &\pi, \lambda \geq 0.
  \end{aligned}
\end{equation*}
This yields the classic robust reformulation
\begin{equation*}
  \label{eq:appendix:applications:portfolio:budgeted:robust-reformulation}
  \begin{aligned}
    \max_{y,s,x,\pi,\lambda} \quad &\bar{\mu}^\top y - h^\top x -
    \Gamma \pi - \sum_{i \in [N]} (1 - \gamma_i x_i) \lambda_i
    \\
    \st \quad
    &\eqref{eq:appendix:applications:portfolio:budgeted:c1}\text{--}%
    \eqref{eq:appendix:applications:portfolio:budgeted:c3},
    \\
    &\pi + \lambda_i \geq \hat{\mu}_i y_i, \quad
      i \in [N],
    \\
    &\pi, \lambda \geq 0,
  \end{aligned}
\end{equation*}
and the single-level reformulation of the bilevel problem
\begin{equation*}
  \label{eq:appendix:applications:portfolio:budgeted:bilevel-sd}
  \begin{aligned}
    \max_{y,s,x,u,\pi,\lambda} \quad &\bar{\mu}^\top y - h^\top x
    - \sum_{i \in [N]} u_i \hat{\mu}_i y_i
    \\
    \st \quad
    &\eqref{eq:appendix:applications:portfolio:budgeted:c1}\text{--}%
    \eqref{eq:appendix:applications:portfolio:budgeted:c3},
    \\
    &\sum_{i \in [N]} u_i \leq \Gamma,
    \\
    &0 \leq u_i \leq 1 - \gamma_i x_i, \quad i \in [N],
    \\
    &\pi + \lambda_i \geq \hat{\mu}_i y_i, \quad
      i \in [N],
    \\
    &\sum_{i \in [N]} u_i \hat{\mu}_i y_i \geq
    \Gamma \pi + \sum_{i \in [N]} (1 - \gamma_i x_i) \lambda_i,
    \\
    &\pi, \lambda \geq 0.
  \end{aligned}
\end{equation*}

\subsubsection{Knapsack Uncertainty}
\label{sec:appendix:applications:portfolio:knapsack}
For knapsack uncertainty sets, the corresponding portfolio
problem~\eqref{eq:applications:portfolio:cont-dd-rob-model:knapsack} from
Section~\ref{sec:applications:portfolio:knapsack-uncertainty} is
\begin{equation*}
  \label{eq:appendix:applications:portfolio:knapsack}
  \begin{aligned}
    \max_{y,s} \quad &\bar{\mu}^\top y - \max_{u \in U_\text{p}(s)} \
    \sum_{i \in [N]} u_i \hat{\mu}_i y_i
    \\
    \st \quad
    &\eqref{eq:appendix:applications:portfolio:budgeted:c1}\text{--}%
    \eqref{eq:appendix:applications:portfolio:budgeted:c3},
  \end{aligned}
\end{equation*}
with continuous decision-dependent knapsack uncertainty set $U_\text{p} =
U^\text{ck}$ as defined in Section~\ref{sec:uncertainty-sets:knapsack}. The
equivalent bilevel formulation reads
\begin{equation*}
  \label{eq:appendix:applications:portfolio:knapsack:bilevel}
  \begin{aligned}
    \max_{y,s,u} \quad &\bar{\mu}^\top y - \sum_{i \in [N]} u_i \hat{\mu}_i y_i
    \\
    \st \quad
    &\eqref{eq:appendix:applications:portfolio:budgeted:c1}\text{--}%
    \eqref{eq:appendix:applications:portfolio:budgeted:c3},
    \\
    &u \in S(y,s),
  \end{aligned}
\end{equation*}
where $S(y,s)$ denotes the set of optimal solutions to the $(y,s)$-parameterized
lower-level problem
\begin{equation*}
  \label{eq:appendix:applications:portfolio:knapsack:lower-level}
  \begin{aligned}
    \max_{u} \quad &\sum_{i \in [N]} u_i \hat{\mu}_i y_i
    \\
    \st \quad &f^\top u \leq b + w^\top s,
    \\
    &0 \leq u_i \leq 1, \quad i \in [N].
  \end{aligned}
\end{equation*}
Again, due to the quadratic term in the constraints, the discrete version of
this problem cannot be solved with an MIBLP solver. However, for the continuous
version we can dualize the lower-level problem to obtain the dual
\begin{equation*}
  \label{eq:appendix:applications:portfolio:knapsack:inner-dual}
  \begin{aligned}
    \min_{\pi, \lambda} \quad
    &\pi (b + w^\top s) + \sum_{i \in [N]} \lambda_i
    \\
    \st \quad
    &f_i\pi + \lambda_i \geq \hat{\mu}_i y_i, \quad
      i \in [N],
    \\
    &\pi, \lambda \geq 0.
  \end{aligned}
\end{equation*}
This leads to the classic robust reformulation
\begin{equation*}
  \label{eq:appendix:applications:portfolio:knapsack:robust-reformulation}
  \begin{aligned}
    \max_{y,s,\pi,\lambda} \quad &\bar{\mu}^\top y
    - \pi (b + w^\top s) - \sum_{i \in [N]} \lambda_i
    \\
    \st \quad
    &\eqref{eq:appendix:applications:portfolio:budgeted:c1}\text{--}%
    \eqref{eq:appendix:applications:portfolio:budgeted:c3},
    \\
    &f_i\pi + \lambda_i \geq \hat{\mu}_i y_i, \quad
      i \in [N],
    \\
    &\pi, \lambda \geq 0,
  \end{aligned}
\end{equation*}
and the single-level reformulation of the bilevel problem
\begin{equation*}
  \label{eq:appendix:applications:portfolio:knapsack:bilevel-sd}
  \begin{aligned}
    \max_{y,s,u,\pi,\lambda} \quad &\bar{\mu}^\top y
    - \sum_{i \in [N]} u_i \hat{\mu}_i y_i
    \\
    \st \quad
    &\eqref{eq:appendix:applications:portfolio:budgeted:c1}\text{--}%
    \eqref{eq:appendix:applications:portfolio:budgeted:c3},
    \\
    &f^\top u \leq b + w^\top s,
    \\
    &0 \leq u_i \leq 1, \quad i \in [N],
    \\
    &f_i\pi + \lambda_i \geq \hat{\mu}_i y_i, \quad
      i \in [N],
    \\
    &\sum_{i \in [N]} u_i \hat{\mu}_i y_i \geq
    \pi (b + w^\top s) + \sum_{i \in [N]} \lambda_i,
    \\
    &\pi, \lambda \geq 0.
  \end{aligned}
\end{equation*}
Bilinearities are handled by McCormick envelopes; see
Appendix~\ref{sec:appendix:mccormick} for details.

\section{Quantified Integer Programming Models}
\label{sec:qip-models}

The DDRO problems we study in this paper can also be interpreted and solved as
quantified integer programs~(QIPs). A QIP is a natural extension of an integer
linear program~(ILP) that contains both existentially and universally quantified
variables. For more details we refer to, e.g.,
\textcite{goerigk2021} and \textcite{chistikov2017} as well as the references
therein. The QIP solver \textsf{Yasol} is described in detail in the paper by
\textcite{yasol2017}.

In the context of the DDRO problem~\eqref{eq:general-model}, the decision variables $x$
are existentially quantified, while the uncertain parameters $u$ are universally
quantified. An equivalent reformulation of Problem~\eqref{eq:general-model} as a
QIP is thus given by
\begin{subequations}
  \label{eq:qip-models:qip-model}
  \begin{align}
    \min_{x} \quad & c^\top x
    \\
    \st \quad &\exists x \in X \ \forall i \in [m_x] \ \forall u_i \in U_i(x):
    \\
    & a_i^\top x + u_i^\top B_i x \leq \beta_i.
    \label{eq:qip-models:qip-model:existential-constraint}
  \end{align}
\end{subequations}
We solve these QIP formulations using \textsf{Yasol} for the decision-dependent
robust knapsack and shortest-path problems with discrete uncertainty sets. Thus,
in the following, we can assume that the variables $x$ and $u$ are binary since
this is the case in both applications.

To solve the arising QIPs using \textsf{Yasol}, we need to linearize the
bilinear terms~$u_i^\top B_i x$ in
Constraint~\eqref{eq:qip-models:qip-model:existential-constraint}. In
Appendix~\ref{sec:appendix:mccormick} we show that it is sufficient to introduce
McCormick inequalities only in the lower level of the bilevel
reformulation~\eqref{eq:general-model:bilevel-reformulation}. This way, we
obtain a linear QIP of the form
\begin{equation}
  \label{eq:qip-models:qip-bilevel}
  \begin{aligned}
    \min_{x, r} \quad & c^\top x
    \\
    \st \quad &\exists x \in X \ \forall i \in [m_x] \ \forall u_i \in \tilde{U}_i(x):
    \\
    & a_i^\top x + r_i^\top B_i \leq \beta_i,
  \end{aligned}
\end{equation}
where $\tilde{U}_i(x)$ is the uncertainty set $U_i(x)$ augmented by the
McCormick inequalities for the bilinear terms $u_{ij} (B_i x)_j$ with $j \in
[n_x]$ and the auxiliary variables $r_i \in \R^{n_x}$ defined as $r_{ij} =
u_{ij} x_j$. For the exact formulation of the McCormick envelopes, we refer to
Appendix~\ref{sec:appendix:mccormick}. Another method to linearize the bilinear terms is
to add a second existential stage to the QIP that defines the auxiliary
variables $r_i$. With this, we obtain the linear QIP
\begin{equation}
  \label{eq:qip-models:qip-existeval}
  \begin{aligned}
    \min_{x, r} \quad & c^\top x
    \\
    \st \quad &\exists x \in X \ \forall i \in [m_x] \ \forall u_i \in U_i(x) \ \exists r_i \in \set{0,1}^{n_x}:
    \\
    & a_i^\top x + r_i^\top B_i \leq \beta_i,
    \\
    & u_{ij} + x_j - 1 \leq r_{ij}, \quad \forall j \in [n_x]
  \end{aligned}
\end{equation}
For binary variables $x$ and $u_i$, the last constraint ensures that $r_{ij} =
1$ holds if $u_{ij} = 1$ and $x_j = 1$.

In the numerical experiments in Section~\ref{sec:num-results}, we refer
to Reformulation~\eqref{eq:qip-models:qip-bilevel} as \textsf{qip-bilevel} and
to Reformulation~\eqref{eq:qip-models:qip-existeval} as \textsf{qip-existeval}.
All~\textsf{.qlp} instance files used in the experiments are available
at~\url{https://github.com/simstevens/ddro-via-bilevel}.

\section{McCormick envelopes for Bilinearities in Bilevel Problems}
\label{sec:appendix:mccormick}

We now show that we can equivalently linearize products of binary variables in
the lower and upper-level of a given bilevel problem using McCormick
inequalities \parencite{mccormick1976}. To this end, we denote the element-wise
(or Hadamard) product of two vectors $x, u \in \set{0,1}^n$ by $(x \circ u)_i =
x_i u_i$, $i \in \set{1,\dotsc,n}$. We then consider the bilevel problem
\begin{align*}
  \min_{x, u} \quad &c^\top x
  \\
  \st \quad &a^\top x + b^\top (x \circ u) \leq d_u,
  \\
                    &x \in \set{0,1}^n,
  \\
                    &u \in S(x),
\end{align*}
where $S(x)$ is the set of optimal solutions to the $x$-parameterized
lower-level problem
\begin{equation}
  \label{eq:app-a:lower-level-without-mccormick}
  \begin{aligned}
    \min_{u} \quad &-b^\top(x \circ u)
    \\
    \st \quad &Cx + Du \leq d_l,
    \\
                   &u \in \set{0,1}^n.
  \end{aligned}
\end{equation}
To linearize the bilinear terms ``$x \circ u$''  we use the McCormick
inequalities
\begin{subequations}
  \label{eq:app-a:mccormick}
  \begin{align}
    -u_i + r_i^u \leq 0, \
    -x_i + r_i^u \leq 0, \
    x_i + u_i - r_i^u \leq 1,
    & \quad
      i \in \set{1,\dotsc,n},
      \label{eq:app-a:mccormick-upper-level}
    \\
    -u_i + r_i^l \leq 0, \
    -x_i + r_i^l \leq 0,  \
    x_i + u_i - r_i^l \leq 1,
    & \quad i \in \set{1,\dotsc,n}.
      \label{eq:app-a:mccormick-lower-level}
  \end{align}
\end{subequations}
Here, \mbox{$r^u \in \set{0,1}^n$} is the auxiliary variable for the upper-level
problem and \mbox{$r^l \in \set{0,1}^n$} is the auxiliary variable for the
lower-level problem and \mbox{$r_i^l = r_i^u = x_i u_i, \ i \in
\set{1,\dotsc,n},$} holds by construction. Note, that the auxiliary variables
can also be chosen continuously. However, none of the approaches showed a
significant advantage in terms of runtime. Thus, we will only state the
derivation and the results for the binary case.

Using these McCormick inequalities~\eqref{eq:app-a:mccormick}, we obtain the
bilevel problem
\begin{equation}
  \label{eq:app-a:both-levels}
  \begin{split}
    \min_{x, u, r^u, r^l} \quad
    & c^\top x
    \\
    \st \quad
    & a^\top x + b^\top r^u \leq d_u,
    \\
    &\eqref{eq:app-a:mccormick-upper-level},
    \ x, r^u \in \set{0,1}^n,
    \\
    &(u, r^l) \in S(x),
  \end{split}
\end{equation}
where $S(x)$ is the set of optimal solutions to the $x$-parameterized
lower-level problem
\begin{equation}
  \label{eq:app-a:lower-level-problem}
  \begin{split}
    \min_{u, r^l} \quad
    & -b^\top r^l
    \\
    \st \quad
    & Cx + Du \leq d_l,
    \\
    &\eqref{eq:app-a:mccormick-lower-level},
    \ u, r^l \in \set{0,1}^n.
  \end{split}
\end{equation}
We can now show that this problem is equivalent to the following one in which
the McCormick inequalities are only in the lower-level problem
\begin{equation}
  \label{eq:app-a:lower-level}
  \begin{split}
    \min_{x, u, r^l} \quad
    & c^\top x
    \\
    \st \quad
    & a^\top x + b^\top r^l \leq d_u,
    \\
    &x \in \set{0,1}^n,
    \ (u, r^l) \in S(x).
  \end{split}
\end{equation}
Here, $S(x)$ is the set of optimal solutions to the $x$-parameterized
lower-level problem~\eqref{eq:app-a:lower-level-problem}.

\begin{lemma}
  \label{lem:app-a:mccormick-in-lower-level}
  For every bilevel-feasible point $(x, r^u, u, r^l)$ of
  Problem~\eqref{eq:app-a:both-levels}, the point $(x,u,r^l)$ is
  bilevel-feasible for Problem~\eqref{eq:app-a:lower-level} with the same
  objective function value. Moreover, for every bilevel-feasible point
  $(x,u,r^l)$ of Problem~\eqref{eq:app-a:lower-level}, the point $(x,r^l,u,r^l)$
  is bilevel-feasible for Problem~\eqref{eq:app-a:both-levels} with the same
  objective function value.
\end{lemma}

\begin{proof}
  Let $(x,u,r^l)$ be bilevel-feasible for Problem~\eqref{eq:app-a:lower-level}.
  Then, $(x,r^l,u,r^l)$ satisfies the upper- and lower-level constraints of
  Problem~\eqref{eq:app-a:both-levels}. The point $(u,r^l)$ is optimal for the
  lower-level problem~\eqref{eq:app-a:lower-level-problem} due to the
  bilevel-feasibility of $(x,u,r^l)$.

  Let now $(x,r^u,u,r^l)$ be bilevel-feasible for
  Problem~\eqref{eq:app-a:both-levels}. Then, $(x,u,r^l)$ satisfies the upper-
  and lower-level constraints of Problem~\eqref{eq:app-a:lower-level} since
  $r^l=r^u$ by construction of the McCormick inequalities. The point $(u,r^l)$
  is optimal for the lower-level problem~\eqref{eq:app-a:lower-level-problem}
  due to the bilevel-feasibility of $(x,r^u,u,r^l)$.

  We finally note that Problems~\eqref{eq:app-a:both-levels} and
  \eqref{eq:app-a:lower-level} have the same objective functions, which proves
  the claim.
\end{proof}

We note that it is w.l.o.g.\ to assume that the uncertainties and, thus, the
auxiliary variables for the McCormick uncertainties are exclusively in the
constraints. For uncertainties in the objective function we can always use the
epigraph reformulation.

As a result of the presented reformulation, we obtain a mixed-integer linear
bilevel problem~\eqref{eq:app-a:lower-level}  that can be directly solved by
state-of-the-art bilevel solvers. Consequently, we can reformulate the discrete
versions of the bilevel problems in
Section~\ref{sec:appendix:applications:shortest-path} and
\ref{sec:appendix:applications:knapsack} as problems that can be tackled by the
\textsf{MibS} solver.

\begin{remark}
  Since the proof of Lemma~\ref{lem:app-a:mccormick-in-lower-level} does not
  depend on the structure of the upper-level non-coupling constraints, it stays
  valid if (non-)linear non-coupling constraints are added to the upper-level.
  However, there is no solver available for such nonlinear bilevel problems.
\end{remark}

\begin{remark}
  Lemma~\ref{lem:app-a:mccormick-in-lower-level} also holds for any finitely
  bounded continuous lower-level variable~$u_i \in [0,u_i^+]$ and the
  corresponding McCormick inequalities. Since the lower-level
  problem~\eqref{eq:app-a:lower-level-without-mccormick} would then be a
  continuous linear problem, the bilevel problem can be reformulated as a
  single-level problem using the KKT conditions or strong duality.
\end{remark}

\FloatBarrier

\FloatBarrier
\section{Supplementary Figures}
\label{sec:appendix:plots}
\rev{
\begin{figure}[h!]
  \centering
  \makebox[\textwidth][c]{%
    \scalebox{1.0}{\input{plots/optl_rev/sp_hedging_objective_scatter.tex}}%
  } \caption{Scatter plot for the shortest-path problem with continuous budgeted
  uncertainty sets. The x-axis shows the hedging costs~$h$ and the y-axis shows
  the optimal objective function values.}
  \label{fig:num-results:hedging:objective:scatter}
\end{figure}
\begin{figure}
  \centering
  \makebox[\textwidth][c]{%
    \scalebox{1.0}{\begin{tikzpicture}[x=1pt,y=1pt]
\definecolor{fillColor}{RGB}{255,255,255}
\path[use as bounding box,fill=fillColor,fill opacity=0.00] (35,0) rectangle (108.54,144.54);
\begin{scope}
\path[clip] ( 38.33, 32.16) rectangle (139.54,139.54);
\definecolor{drawColor}{gray}{0.92}

\path[draw=drawColor,line width= 0.3pt,line join=round] ( 38.33, 49.25) --
	(139.54, 49.25);

\path[draw=drawColor,line width= 0.3pt,line join=round] ( 38.33, 73.65) --
	(139.54, 73.65);

\path[draw=drawColor,line width= 0.3pt,line join=round] ( 38.33, 98.05) --
	(139.54, 98.05);

\path[draw=drawColor,line width= 0.3pt,line join=round] ( 38.33,122.46) --
	(139.54,122.46);

\path[draw=drawColor,line width= 0.3pt,line join=round] ( 47.94, 32.16) --
	( 47.94,139.54);

\path[draw=drawColor,line width= 0.3pt,line join=round] ( 67.28, 32.16) --
	( 67.28,139.54);

\path[draw=drawColor,line width= 0.3pt,line join=round] ( 86.61, 32.16) --
	( 86.61,139.54);

\path[draw=drawColor,line width= 0.3pt,line join=round] (105.94, 32.16) --
	(105.94,139.54);

\path[draw=drawColor,line width= 0.3pt,line join=round] (125.27, 32.16) --
	(125.27,139.54);

\path[draw=drawColor,line width= 0.5pt,line join=round] ( 38.33, 37.04) --
	(139.54, 37.04);

\path[draw=drawColor,line width= 0.5pt,line join=round] ( 38.33, 61.45) --
	(139.54, 61.45);

\path[draw=drawColor,line width= 0.5pt,line join=round] ( 38.33, 85.85) --
	(139.54, 85.85);

\path[draw=drawColor,line width= 0.5pt,line join=round] ( 38.33,110.26) --
	(139.54,110.26);

\path[draw=drawColor,line width= 0.5pt,line join=round] ( 38.33,134.66) --
	(139.54,134.66);

\path[draw=drawColor,line width= 0.5pt,line join=round] ( 57.61, 32.16) --
	( 57.61,139.54);

\path[draw=drawColor,line width= 0.5pt,line join=round] ( 76.94, 32.16) --
	( 76.94,139.54);

\path[draw=drawColor,line width= 0.5pt,line join=round] ( 96.27, 32.16) --
	( 96.27,139.54);

\path[draw=drawColor,line width= 0.5pt,line join=round] (115.61, 32.16) --
	(115.61,139.54);

\path[draw=drawColor,line width= 0.5pt,line join=round] (134.94, 32.16) --
	(134.94,139.54);
\definecolor{drawColor}{RGB}{69,117,180}

\path[draw=drawColor,line width= 1.4pt,line join=round] ( 56.32, 38.02) --
	( 60.33, 38.02) --
	( 60.33, 39.00) --
	( 60.45, 39.00) --
	( 60.45, 39.97) --
	( 60.71, 39.97) --
	( 60.71, 40.95) --
	( 63.24, 40.95) --
	( 63.24, 41.92) --
	( 67.73, 41.92) --
	( 67.73, 42.90) --
	( 70.23, 42.90) --
	( 70.23, 43.88) --
	( 72.55, 43.88) --
	( 72.55, 44.85) --
	( 72.59, 44.85) --
	( 72.59, 45.83) --
	( 72.62, 45.83) --
	( 72.62, 46.80) --
	( 73.71, 46.80) --
	( 73.71, 47.78) --
	( 73.97, 47.78) --
	( 73.97, 48.76) --
	( 74.42, 48.76) --
	( 74.42, 49.73) --
	( 74.79, 49.73) --
	( 74.79, 50.71) --
	( 75.24, 50.71) --
	( 75.24, 51.69) --
	( 76.05, 51.69) --
	( 76.05, 52.66) --
	( 77.06, 52.66) --
	( 77.06, 53.64) --
	( 77.32, 53.64) --
	( 77.32, 54.61) --
	( 77.73, 54.61) --
	( 77.73, 55.59) --
	( 79.67, 55.59) --
	( 79.67, 56.57) --
	( 79.73, 56.57) --
	( 79.73, 57.54) --
	( 80.08, 57.54) --
	( 80.08, 58.52) --
	( 80.11, 58.52) --
	( 80.11, 59.49) --
	( 80.62, 59.49) --
	( 80.62, 60.47) --
	( 81.10, 60.47) --
	( 81.10, 61.45) --
	( 81.30, 61.45) --
	( 81.30, 62.42) --
	( 81.71, 62.42) --
	( 81.71, 63.40) --
	( 81.98, 63.40) --
	( 81.98, 64.38) --
	( 82.23, 64.38) --
	( 82.23, 65.35) --
	( 82.60, 65.35) --
	( 82.60, 66.33) --
	( 82.67, 66.33) --
	( 82.67, 67.30) --
	( 84.48, 67.30) --
	( 84.48, 68.28) --
	( 85.29, 68.28) --
	( 85.29, 69.26) --
	( 85.41, 69.26) --
	( 85.41, 70.23) --
	( 91.06, 70.23) --
	( 91.06, 71.21) --
	(134.94, 71.21) --
	(134.94, 71.21);
\definecolor{drawColor}{RGB}{252,141,89}

\path[draw=drawColor,line width= 1.4pt,dash pattern=on 2pt off 2pt ,line join=round] ( 42.93, 38.02) --
	( 57.76, 38.02) --
	( 57.76, 39.00) --
	( 59.00, 39.00) --
	( 59.00, 39.97) --
	( 59.77, 39.97) --
	( 59.77, 40.95) --
	( 61.78, 40.95) --
	( 61.78, 41.92) --
	( 64.41, 41.92) --
	( 64.41, 42.90) --
	( 65.81, 42.90) --
	( 65.81, 43.88) --
	( 67.35, 43.88) --
	( 67.35, 44.85) --
	( 68.16, 44.85) --
	( 68.16, 45.83) --
	( 68.98, 45.83) --
	( 68.98, 46.80) --
	( 69.73, 46.80) --
	( 69.73, 47.78) --
	( 69.93, 47.78) --
	( 69.93, 48.76) --
	( 70.53, 48.76) --
	( 70.53, 49.73) --
	( 72.78, 49.73) --
	( 72.78, 50.71) --
	( 72.97, 50.71) --
	( 72.97, 51.69) --
	( 73.09, 51.69) --
	( 73.09, 52.66) --
	( 73.12, 52.66) --
	( 73.12, 53.64) --
	( 73.19, 53.64) --
	( 73.19, 54.61) --
	( 73.92, 54.61) --
	( 73.92, 55.59) --
	( 74.63, 55.59) --
	( 74.63, 56.57) --
	( 74.85, 56.57) --
	( 74.85, 57.54) --
	( 77.10, 57.54) --
	( 77.10, 58.52) --
	( 77.92, 58.52) --
	( 77.92, 59.49) --
	( 78.46, 59.49) --
	( 78.46, 60.47) --
	( 78.93, 60.47) --
	( 78.93, 61.45) --
	( 79.39, 61.45) --
	( 79.39, 62.42) --
	( 79.75, 62.42) --
	( 79.75, 63.40) --
	( 80.22, 63.40) --
	( 80.22, 64.38) --
	( 81.38, 64.38) --
	( 81.38, 65.35) --
	( 81.52, 65.35) --
	( 81.52, 66.33) --
	( 82.04, 66.33) --
	( 82.04, 67.30) --
	( 82.57, 67.30) --
	( 82.57, 68.28) --
	( 82.91, 68.28) --
	( 82.91, 69.26) --
	( 83.55, 69.26) --
	( 83.55, 70.23) --
	( 83.64, 70.23) --
	( 83.64, 71.21) --
	( 83.84, 71.21) --
	( 83.84, 72.18) --
	( 84.06, 72.18) --
	( 84.06, 73.16) --
	( 84.14, 73.16) --
	( 84.14, 74.14) --
	( 84.43, 74.14) --
	( 84.43, 75.11) --
	( 84.43, 75.11) --
	( 84.43, 76.09) --
	( 84.99, 76.09) --
	( 84.99, 77.07) --
	( 85.26, 77.07) --
	( 85.26, 78.04) --
	( 85.91, 78.04) --
	( 85.91, 79.02) --
	( 86.76, 79.02) --
	( 86.76, 79.99) --
	( 86.82, 79.99) --
	( 86.82, 80.97) --
	( 87.35, 80.97) --
	( 87.35, 81.95) --
	( 87.67, 81.95) --
	( 87.67, 82.92) --
	( 88.63, 82.92) --
	( 88.63, 83.90) --
	( 89.57, 83.90) --
	( 89.57, 84.88) --
	( 91.72, 84.88) --
	( 91.72, 85.85) --
	( 92.19, 85.85) --
	( 92.19, 86.83) --
	( 96.82, 86.83) --
	( 96.82, 87.80) --
	(134.94, 87.80) --
	(134.94, 87.80);
\definecolor{drawColor}{RGB}{69,117,180}

\node[text=drawColor,anchor=base west,inner sep=0pt, outer sep=0pt, scale=  1.00] at ( 40,120) {35/100};
\definecolor{drawColor}{RGB}{252,141,89}

\node[text=drawColor,anchor=base west,inner sep=0pt, outer sep=0pt, scale=  1.00] at ( 40,110) {52/100};
\end{scope}
\begin{scope}
\path[clip] (  0.00,  0.00) rectangle (144.54,144.54);
\definecolor{drawColor}{gray}{0.30}

\node[text=drawColor,anchor=base east,inner sep=0pt, outer sep=0pt, scale=  1.00] at ( 33.83, 33.60) {0};

\node[text=drawColor,anchor=base east,inner sep=0pt, outer sep=0pt, scale=  1.00] at ( 33.83, 58.00) {25};

\node[text=drawColor,anchor=base east,inner sep=0pt, outer sep=0pt, scale=  1.00] at ( 33.83, 82.41) {50};

\node[text=drawColor,anchor=base east,inner sep=0pt, outer sep=0pt, scale=  1.00] at ( 33.83,106.81) {75};

\node[text=drawColor,anchor=base east,inner sep=0pt, outer sep=0pt, scale=  1.00] at ( 33.83,131.22) {100};
\end{scope}
\begin{scope}
\path[clip] (  0.00,  0.00) rectangle (144.54,144.54);
\definecolor{drawColor}{gray}{0.30}

\node[text=drawColor,anchor=base,inner sep=0pt, outer sep=0pt, scale=  1.00] at ( 57.61, 20.78) {$10^3$};

\node[text=drawColor,anchor=base,inner sep=0pt, outer sep=0pt, scale=  1.00] at ( 96.27, 20.78) {$10^5$};

\node[text=drawColor,anchor=base,inner sep=0pt, outer sep=0pt, scale=  1.00] at (134.94, 20.78) {$10^7$};
\end{scope}
\begin{scope}
\path[clip] (  0.00,  0.00) rectangle (144.54,144.54);
\definecolor{drawColor}{RGB}{0,0,0}

\node[text=drawColor,rotate= 90.00,anchor=base,inner sep=0pt, outer sep=0pt, scale=  1.00] at ( 15, 85.85) {\# of solved instances};
\end{scope}
\end{tikzpicture}}%
    \scalebox{1.0}{\begin{tikzpicture}[x=1pt,y=1pt]
\definecolor{fillColor}{RGB}{255,255,255}
\path[use as bounding box,fill=fillColor,fill opacity=0.00] (0,0) rectangle (113.54,144.54);
\begin{scope}
\path[clip] ( 38.33, 32.16) rectangle (139.54,139.54);
\definecolor{drawColor}{gray}{0.92}

\path[draw=drawColor,line width= 0.3pt,line join=round] ( 38.33, 49.25) --
	(139.54, 49.25);

\path[draw=drawColor,line width= 0.3pt,line join=round] ( 38.33, 73.65) --
	(139.54, 73.65);

\path[draw=drawColor,line width= 0.3pt,line join=round] ( 38.33, 98.05) --
	(139.54, 98.05);

\path[draw=drawColor,line width= 0.3pt,line join=round] ( 38.33,122.46) --
	(139.54,122.46);

\path[draw=drawColor,line width= 0.3pt,line join=round] ( 42.93, 32.16) --
	( 42.93,139.54);

\path[draw=drawColor,line width= 0.3pt,line join=round] ( 69.22, 32.16) --
	( 69.22,139.54);

\path[draw=drawColor,line width= 0.3pt,line join=round] ( 95.51, 32.16) --
	( 95.51,139.54);

\path[draw=drawColor,line width= 0.3pt,line join=round] (121.79, 32.16) --
	(121.79,139.54);

\path[draw=drawColor,line width= 0.5pt,line join=round] ( 38.33, 37.04) --
	(139.54, 37.04);

\path[draw=drawColor,line width= 0.5pt,line join=round] ( 38.33, 61.45) --
	(139.54, 61.45);

\path[draw=drawColor,line width= 0.5pt,line join=round] ( 38.33, 85.85) --
	(139.54, 85.85);

\path[draw=drawColor,line width= 0.5pt,line join=round] ( 38.33,110.26) --
	(139.54,110.26);

\path[draw=drawColor,line width= 0.5pt,line join=round] ( 38.33,134.66) --
	(139.54,134.66);

\path[draw=drawColor,line width= 0.5pt,line join=round] ( 56.07, 32.16) --
	( 56.07,139.54);

\path[draw=drawColor,line width= 0.5pt,line join=round] ( 82.36, 32.16) --
	( 82.36,139.54);

\path[draw=drawColor,line width= 0.5pt,line join=round] (108.65, 32.16) --
	(108.65,139.54);

\path[draw=drawColor,line width= 0.5pt,line join=round] (134.94, 32.16) --
	(134.94,139.54);
\definecolor{drawColor}{RGB}{69,117,180}

\path[draw=drawColor,line width= 1.4pt,line join=round] ( 42.93, 38.02) --
	( 66.49, 38.02) --
	( 66.49, 39.00) --
	( 67.02, 39.00) --
	( 67.02, 39.97) --
	( 67.26, 39.97) --
	( 67.26, 40.95) --
	( 67.94, 40.95) --
	( 67.94, 41.92) --
	( 68.92, 41.92) --
	( 68.92, 42.90) --
	( 69.91, 42.90) --
	( 69.91, 43.88) --
	( 74.02, 43.88) --
	( 74.02, 44.85) --
	( 75.26, 44.85) --
	( 75.26, 45.83) --
	( 75.91, 45.83) --
	( 75.91, 46.80) --
	( 76.27, 46.80) --
	( 76.27, 47.78) --
	( 78.78, 47.78) --
	( 78.78, 48.76) --
	( 81.36, 48.76) --
	( 81.36, 49.73) --
	( 81.45, 49.73) --
	( 81.45, 50.71) --
	( 81.59, 50.71) --
	( 81.59, 51.69) --
	( 83.29, 51.69) --
	( 83.29, 52.66) --
	( 84.74, 52.66) --
	( 84.74, 53.64) --
	( 84.97, 53.64) --
	( 84.97, 54.61) --
	( 85.34, 54.61) --
	( 85.34, 55.59) --
	( 86.30, 55.59) --
	( 86.30, 56.57) --
	( 86.73, 56.57) --
	( 86.73, 57.54) --
	( 86.98, 57.54) --
	( 86.98, 58.52) --
	( 88.24, 58.52) --
	( 88.24, 59.49) --
	( 88.70, 59.49) --
	( 88.70, 60.47) --
	( 89.24, 60.47) --
	( 89.24, 61.45) --
	( 89.98, 61.45) --
	( 89.98, 62.42) --
	( 90.40, 62.42) --
	( 90.40, 63.40) --
	( 92.75, 63.40) --
	( 92.75, 64.38) --
	( 95.31, 64.38) --
	( 95.31, 65.35) --
	( 99.09, 65.35) --
	( 99.09, 66.33) --
	( 99.37, 66.33) --
	( 99.37, 67.30) --
	( 99.54, 67.30) --
	( 99.54, 68.28) --
	(101.88, 68.28) --
	(101.88, 69.26) --
	(102.26, 69.26) --
	(102.26, 70.23) --
	(102.75, 70.23) --
	(102.75, 71.21) --
	(103.46, 71.21) --
	(103.46, 72.18) --
	(103.71, 72.18) --
	(103.71, 73.16) --
	(103.76, 73.16) --
	(103.76, 74.14) --
	(103.82, 74.14) --
	(103.82, 75.11) --
	(104.20, 75.11) --
	(104.20, 76.09) --
	(104.49, 76.09) --
	(104.49, 77.07) --
	(104.68, 77.07) --
	(104.68, 78.04) --
	(106.22, 78.04) --
	(106.22, 79.02) --
	(106.47, 79.02) --
	(106.47, 79.99) --
	(106.89, 79.99) --
	(106.89, 80.97) --
	(107.23, 80.97) --
	(107.23, 81.95) --
	(108.23, 81.95) --
	(108.23, 82.92) --
	(108.41, 82.92) --
	(108.41, 83.90) --
	(109.63, 83.90) --
	(109.63, 84.88) --
	(125.68, 84.88) --
	(125.68, 85.85) --
	(134.94, 85.85) --
	(134.94, 85.85);
\definecolor{drawColor}{RGB}{252,141,89}

\path[draw=drawColor,line width= 1.4pt,dash pattern=on 2pt off 2pt ,line join=round] ( 42.93, 38.02) --
	( 59.43, 38.02) --
	( 59.43, 39.00) --
	( 60.03, 39.00) --
	( 60.03, 39.97) --
	( 61.95, 39.97) --
	( 61.95, 40.95) --
	( 69.66, 40.95) --
	( 69.66, 41.92) --
	( 69.86, 41.92) --
	( 69.86, 42.90) --
	( 71.49, 42.90) --
	( 71.49, 43.88) --
	( 73.85, 43.88) --
	( 73.85, 44.85) --
	( 76.55, 44.85) --
	( 76.55, 45.83) --
	( 78.21, 45.83) --
	( 78.21, 46.80) --
	( 78.73, 46.80) --
	( 78.73, 47.78) --
	( 79.31, 47.78) --
	( 79.31, 48.76) --
	( 79.94, 48.76) --
	( 79.94, 49.73) --
	( 79.95, 49.73) --
	( 79.95, 50.71) --
	( 81.68, 50.71) --
	( 81.68, 51.69) --
	( 83.06, 51.69) --
	( 83.06, 52.66) --
	( 83.52, 52.66) --
	( 83.52, 53.64) --
	( 86.13, 53.64) --
	( 86.13, 54.61) --
	( 86.92, 54.61) --
	( 86.92, 55.59) --
	( 88.29, 55.59) --
	( 88.29, 56.57) --
	( 90.36, 56.57) --
	( 90.36, 57.54) --
	( 90.49, 57.54) --
	( 90.49, 58.52) --
	( 90.55, 58.52) --
	( 90.55, 59.49) --
	( 93.81, 59.49) --
	( 93.81, 60.47) --
	( 94.96, 60.47) --
	( 94.96, 61.45) --
	( 95.29, 61.45) --
	( 95.29, 62.42) --
	( 95.92, 62.42) --
	( 95.92, 63.40) --
	( 98.49, 63.40) --
	( 98.49, 64.38) --
	( 98.73, 64.38) --
	( 98.73, 65.35) --
	(100.16, 65.35) --
	(100.16, 66.33) --
	(101.34, 66.33) --
	(101.34, 67.30) --
	(101.95, 67.30) --
	(101.95, 68.28) --
	(102.96, 68.28) --
	(102.96, 69.26) --
	(103.67, 69.26) --
	(103.67, 70.23) --
	(104.31, 70.23) --
	(104.31, 71.21) --
	(104.55, 71.21) --
	(104.55, 72.18) --
	(106.49, 72.18) --
	(106.49, 73.16) --
	(106.68, 73.16) --
	(106.68, 74.14) --
	(107.15, 74.14) --
	(107.15, 75.11) --
	(110.56, 75.11) --
	(110.56, 76.09) --
	(110.86, 76.09) --
	(110.86, 77.07) --
	(111.13, 77.07) --
	(111.13, 78.04) --
	(111.18, 78.04) --
	(111.18, 79.02) --
	(111.37, 79.02) --
	(111.37, 79.99) --
	(113.15, 79.99) --
	(113.15, 80.97) --
	(113.59, 80.97) --
	(113.59, 81.95) --
	(114.87, 81.95) --
	(114.87, 82.92) --
	(115.59, 82.92) --
	(115.59, 83.90) --
	(115.93, 83.90) --
	(115.93, 84.88) --
	(117.02, 84.88) --
	(117.02, 85.85) --
	(117.52, 85.85) --
	(117.52, 86.83) --
	(119.47, 86.83) --
	(119.47, 87.80) --
	(119.94, 87.80) --
	(119.94, 88.78) --
	(121.50, 88.78) --
	(121.50, 89.76) --
	(127.48, 89.76) --
	(127.48, 90.73) --
	(134.94, 90.73) --
	(134.94, 90.73);
\definecolor{drawColor}{RGB}{69,117,180}

\node[text=drawColor,anchor=base west,inner sep=0pt, outer sep=0pt, scale=  1.00] at ( 40,120) {50/100};
\definecolor{drawColor}{RGB}{252,141,89}

\node[text=drawColor,anchor=base west,inner sep=0pt, outer sep=0pt, scale=  1.00] at ( 40,110) {55/100};
\end{scope}
\begin{scope}
\path[clip] (  0.00,  0.00) rectangle (144.54,144.54);
\definecolor{drawColor}{gray}{0.30}

\node[text=drawColor,anchor=base,inner sep=0pt, outer sep=0pt, scale=  1.00] at ( 56.07, 20.78) {$10^1$};

\node[text=drawColor,anchor=base,inner sep=0pt, outer sep=0pt, scale=  1.00] at ( 82.36, 20.78) {$10^3$};

\node[text=drawColor,anchor=base,inner sep=0pt, outer sep=0pt, scale=  1.00] at (108.65, 20.78) {$10^5$};

\node[text=drawColor,anchor=base,inner sep=0pt, outer sep=0pt, scale=  1.00] at (134.94, 20.78) {$10^7$};
\end{scope}
\begin{scope}
\path[clip] (  0.00,  0.00) rectangle (144.54,144.54);
\definecolor{drawColor}{RGB}{0,0,0}

\node[text=drawColor,anchor=base,inner sep=0pt, outer sep=0pt, scale=  1.00] at ( 88.93,  6.94) {Branch-and-bound nodes};
\end{scope}
\begin{scope}
\path[clip] (  0.00,  0.00) rectangle (144.54,144.54);
\definecolor{drawColor}{RGB}{0,0,0}

\end{scope}
\end{tikzpicture}}%
    \scalebox{1.0}{\begin{tikzpicture}[x=1pt,y=1pt]
\definecolor{fillColor}{RGB}{255,255,255}
\path[use as bounding box,fill=fillColor,fill opacity=0.00] (0,0) rectangle (113.54,144.54);
\begin{scope}
\path[clip] ( 38.33, 32.16) rectangle (139.54,139.54);
\definecolor{drawColor}{gray}{0.92}

\path[draw=drawColor,line width= 0.3pt,line join=round] ( 38.33, 49.25) --
	(139.54, 49.25);

\path[draw=drawColor,line width= 0.3pt,line join=round] ( 38.33, 73.65) --
	(139.54, 73.65);

\path[draw=drawColor,line width= 0.3pt,line join=round] ( 38.33, 98.05) --
	(139.54, 98.05);

\path[draw=drawColor,line width= 0.3pt,line join=round] ( 38.33,122.46) --
	(139.54,122.46);

\path[draw=drawColor,line width= 0.3pt,line join=round] ( 42.93, 32.16) --
	( 42.93,139.54);

\path[draw=drawColor,line width= 0.3pt,line join=round] ( 69.22, 32.16) --
	( 69.22,139.54);

\path[draw=drawColor,line width= 0.3pt,line join=round] ( 95.51, 32.16) --
	( 95.51,139.54);

\path[draw=drawColor,line width= 0.3pt,line join=round] (121.79, 32.16) --
	(121.79,139.54);

\path[draw=drawColor,line width= 0.5pt,line join=round] ( 38.33, 37.04) --
	(139.54, 37.04);

\path[draw=drawColor,line width= 0.5pt,line join=round] ( 38.33, 61.45) --
	(139.54, 61.45);

\path[draw=drawColor,line width= 0.5pt,line join=round] ( 38.33, 85.85) --
	(139.54, 85.85);

\path[draw=drawColor,line width= 0.5pt,line join=round] ( 38.33,110.26) --
	(139.54,110.26);

\path[draw=drawColor,line width= 0.5pt,line join=round] ( 38.33,134.66) --
	(139.54,134.66);

\path[draw=drawColor,line width= 0.5pt,line join=round] ( 56.07, 32.16) --
	( 56.07,139.54);

\path[draw=drawColor,line width= 0.5pt,line join=round] ( 82.36, 32.16) --
	( 82.36,139.54);

\path[draw=drawColor,line width= 0.5pt,line join=round] (108.65, 32.16) --
	(108.65,139.54);

\path[draw=drawColor,line width= 0.5pt,line join=round] (134.94, 32.16) --
	(134.94,139.54);
\definecolor{drawColor}{RGB}{69,117,180}

\path[draw=drawColor,line width= 1.4pt,line join=round] ( 42.93, 52.66) --
	( 49.20, 52.66) --
	( 49.20, 53.64) --
	( 54.04, 53.64) --
	( 54.04, 54.61) --
	( 57.57, 54.61) --
	( 57.57, 55.59) --
	( 61.53, 55.59) --
	( 61.53, 56.57) --
	( 65.37, 56.57) --
	( 65.37, 57.54) --
	( 69.66, 57.54) --
	( 69.66, 58.52) --
	( 72.35, 58.52) --
	( 72.35, 59.49) --
	( 72.44, 59.49) --
	( 72.44, 60.47) --
	( 74.17, 60.47) --
	( 74.17, 61.45) --
	( 75.14, 61.45) --
	( 75.14, 62.42) --
	( 77.27, 62.42) --
	( 77.27, 63.40) --
	( 77.40, 63.40) --
	( 77.40, 64.38) --
	( 77.53, 64.38) --
	( 77.53, 65.35) --
	( 77.98, 65.35) --
	( 77.98, 66.33) --
	( 78.57, 66.33) --
	( 78.57, 67.30) --
	( 78.88, 67.30) --
	( 78.88, 68.28) --
	( 79.08, 68.28) --
	( 79.08, 69.26) --
	( 79.09, 69.26) --
	( 79.09, 70.23) --
	( 79.60, 70.23) --
	( 79.60, 71.21) --
	( 79.95, 71.21) --
	( 79.95, 72.18) --
	( 80.00, 72.18) --
	( 80.00, 73.16) --
	( 80.54, 73.16) --
	( 80.54, 74.14) --
	( 80.67, 74.14) --
	( 80.67, 75.11) --
	( 81.32, 75.11) --
	( 81.32, 76.09) --
	( 81.34, 76.09) --
	( 81.34, 77.07) --
	( 81.53, 77.07) --
	( 81.53, 78.04) --
	( 82.08, 78.04) --
	( 82.08, 79.02) --
	( 82.42, 79.02) --
	( 82.42, 79.99) --
	( 82.68, 79.99) --
	( 82.68, 80.97) --
	( 82.71, 80.97) --
	( 82.71, 81.95) --
	( 82.83, 81.95) --
	( 82.83, 82.92) --
	( 82.84, 82.92) --
	( 82.84, 83.90) --
	( 83.05, 83.90) --
	( 83.05, 84.88) --
	( 84.05, 84.88) --
	( 84.05, 85.85) --
	( 86.30, 85.85) --
	( 86.30, 86.83) --
	( 87.16, 86.83) --
	( 87.16, 87.80) --
	( 87.24, 87.80) --
	( 87.24, 88.78) --
	( 87.64, 88.78) --
	( 87.64, 89.76) --
	( 88.37, 89.76) --
	( 88.37, 90.73) --
	( 89.00, 90.73) --
	( 89.00, 91.71) --
	( 89.37, 91.71) --
	( 89.37, 92.68) --
	( 89.88, 92.68) --
	( 89.88, 93.66) --
	( 90.58, 93.66) --
	( 90.58, 94.64) --
	( 91.30, 94.64) --
	( 91.30, 95.61) --
	( 91.48, 95.61) --
	( 91.48, 96.59) --
	( 91.88, 96.59) --
	( 91.88, 97.57) --
	( 91.94, 97.57) --
	( 91.94, 98.54) --
	( 92.82, 98.54) --
	( 92.82, 99.52) --
	( 93.36, 99.52) --
	( 93.36,100.49) --
	( 93.92,100.49) --
	( 93.92,101.47) --
	( 94.36,101.47) --
	( 94.36,102.45) --
	( 96.11,102.45) --
	( 96.11,103.42) --
	( 96.13,103.42) --
	( 96.13,104.40) --
	( 98.08,104.40) --
	( 98.08,105.37) --
	( 99.76,105.37) --
	( 99.76,106.35) --
	(100.07,106.35) --
	(100.07,107.33) --
	(111.97,107.33) --
	(111.97,108.30) --
	(113.73,108.30) --
	(113.73,109.28) --
	(115.16,109.28) --
	(115.16,110.26) --
	(134.94,110.26) --
	(134.94,110.26);
\definecolor{drawColor}{RGB}{252,141,89}

\path[draw=drawColor,line width= 1.4pt,dash pattern=on 2pt off 2pt ,line join=round] ( 75.01, 38.02) --
	( 75.03, 38.02) --
	( 75.03, 39.00) --
	( 75.12, 39.00) --
	( 75.12, 39.97) --
	( 75.45, 39.97) --
	( 75.45, 40.95) --
	( 75.66, 40.95) --
	( 75.66, 41.92) --
	( 75.71, 41.92) --
	( 75.71, 42.90) --
	( 75.75, 42.90) --
	( 75.75, 43.88) --
	( 75.82, 43.88) --
	( 75.82, 44.85) --
	( 75.87, 44.85) --
	( 75.87, 45.83) --
	( 76.05, 45.83) --
	( 76.05, 46.80) --
	( 76.08, 46.80) --
	( 76.08, 47.78) --
	( 76.10, 47.78) --
	( 76.10, 48.76) --
	( 76.22, 48.76) --
	( 76.22, 49.73) --
	( 76.81, 49.73) --
	( 76.81, 50.71) --
	( 77.26, 50.71) --
	( 77.26, 53.64) --
	( 77.33, 53.64) --
	( 77.33, 55.59) --
	( 77.35, 55.59) --
	( 77.35, 56.57) --
	( 77.44, 56.57) --
	( 77.44, 57.54) --
	( 77.45, 57.54) --
	( 77.45, 58.52) --
	( 77.56, 58.52) --
	( 77.56, 59.49) --
	( 77.74, 59.49) --
	( 77.74, 60.47) --
	( 77.79, 60.47) --
	( 77.79, 61.45) --
	( 77.90, 61.45) --
	( 77.90, 62.42) --
	( 78.04, 62.42) --
	( 78.04, 63.40) --
	( 78.06, 63.40) --
	( 78.06, 64.38) --
	( 78.15, 64.38) --
	( 78.15, 65.35) --
	( 78.42, 65.35) --
	( 78.42, 66.33) --
	( 78.51, 66.33) --
	( 78.51, 67.30) --
	( 78.64, 67.30) --
	( 78.64, 69.26) --
	( 78.69, 69.26) --
	( 78.69, 70.23) --
	( 78.70, 70.23) --
	( 78.70, 71.21) --
	( 78.73, 71.21) --
	( 78.73, 72.18) --
	( 79.33, 72.18) --
	( 79.33, 73.16) --
	( 79.46, 73.16) --
	( 79.46, 74.14) --
	( 79.53, 74.14) --
	( 79.53, 75.11) --
	( 79.58, 75.11) --
	( 79.58, 76.09) --
	( 79.62, 76.09) --
	( 79.62, 77.07) --
	( 79.64, 77.07) --
	( 79.64, 78.04) --
	( 79.65, 78.04) --
	( 79.65, 79.02) --
	( 79.67, 79.02) --
	( 79.67, 79.99) --
	( 79.77, 79.99) --
	( 79.77, 80.97) --
	( 80.01, 80.97) --
	( 80.01, 81.95) --
	( 80.25, 81.95) --
	( 80.25, 82.92) --
	( 80.40, 82.92) --
	( 80.40, 84.88) --
	( 80.70, 84.88) --
	( 80.70, 85.85) --
	( 80.82, 85.85) --
	( 80.82, 86.83) --
	( 81.01, 86.83) --
	( 81.01, 87.80) --
	( 81.06, 87.80) --
	( 81.06, 89.76) --
	( 81.17, 89.76) --
	( 81.17, 90.73) --
	( 81.18, 90.73) --
	( 81.18, 91.71) --
	( 81.19, 91.71) --
	( 81.19, 92.68) --
	( 81.26, 92.68) --
	( 81.26, 93.66) --
	( 81.28, 93.66) --
	( 81.28, 94.64) --
	( 81.53, 94.64) --
	( 81.53, 95.61) --
	( 81.57, 95.61) --
	( 81.57, 96.59) --
	( 81.63, 96.59) --
	( 81.63, 97.57) --
	( 81.82, 97.57) --
	( 81.82, 99.52) --
	( 81.83, 99.52) --
	( 81.83,100.49) --
	( 81.93,100.49) --
	( 81.93,101.47) --
	( 82.01,101.47) --
	( 82.01,102.45) --
	( 82.19,102.45) --
	( 82.19,103.42) --
	( 82.26,103.42) --
	( 82.26,104.40) --
	( 82.30,104.40) --
	( 82.30,105.37) --
	( 82.41,105.37) --
	( 82.41,106.35) --
	( 82.44,106.35) --
	( 82.44,107.33) --
	( 82.47,107.33) --
	( 82.47,108.30) --
	( 82.51,108.30) --
	( 82.51,109.28) --
	( 82.52,109.28) --
	( 82.52,110.26) --
	( 82.54,110.26) --
	( 82.54,111.23) --
	( 82.67,111.23) --
	( 82.67,112.21) --
	( 82.69,112.21) --
	( 82.69,113.18) --
	( 82.86,113.18) --
	( 82.86,114.16) --
	( 82.87,114.16) --
	( 82.87,115.14) --
	( 85.79,115.14) --
	( 85.79,116.11) --
	( 88.18,116.11) --
	( 88.18,117.09) --
	( 88.19,117.09) --
	( 88.19,118.06) --
	( 88.63,118.06) --
	( 88.63,119.04) --
	( 90.81,119.04) --
	( 90.81,120.02) --
	( 94.13,120.02) --
	( 94.13,120.99) --
	( 94.64,120.99) --
	( 94.64,121.97) --
	( 96.40,121.97) --
	( 96.40,122.95) --
	( 99.64,122.95) --
	( 99.64,123.92) --
	( 99.98,123.92) --
	( 99.98,124.90) --
	(100.07,124.90) --
	(100.07,125.87) --
	(100.27,125.87) --
	(100.27,126.85) --
	(105.86,126.85) --
	(105.86,127.83) --
	(106.79,127.83) --
	(106.79,128.80) --
	(106.95,128.80) --
	(106.95,129.78) --
	(108.36,129.78) --
	(108.36,130.75) --
	(110.60,130.75) --
	(110.60,131.73) --
	(113.10,131.73) --
	(113.10,132.71) --
	(116.91,132.71) --
	(116.91,133.68) --
	(134.94,133.68) --
	(134.94,133.68);
\definecolor{drawColor}{RGB}{69,117,180}

\node[text=drawColor,anchor=base west,inner sep=0pt, outer sep=0pt, scale=  1.00] at ( 40,120) {75/100};
\definecolor{drawColor}{RGB}{252,141,89}

\node[text=drawColor,anchor=base west,inner sep=0pt, outer sep=0pt, scale=  1.00] at ( 40,110) {99/100};
\end{scope}
\begin{scope}
\path[clip] (  0.00,  0.00) rectangle (144.54,144.54);
\definecolor{drawColor}{gray}{0.30}

\node[text=drawColor,anchor=base,inner sep=0pt, outer sep=0pt, scale=  1.00] at ( 56.07, 20.78) {$10^1$};

\node[text=drawColor,anchor=base,inner sep=0pt, outer sep=0pt, scale=  1.00] at ( 82.36, 20.78) {$10^3$};

\node[text=drawColor,anchor=base,inner sep=0pt, outer sep=0pt, scale=  1.00] at (108.65, 20.78) {$10^5$};

\node[text=drawColor,anchor=base,inner sep=0pt, outer sep=0pt, scale=  1.00] at (134.94, 20.78) {$10^7$};
\end{scope}

\end{tikzpicture}}%
  }
  \makebox[\textwidth][c]{%
    \scalebox{1.0}{\begin{tikzpicture}[x=1pt,y=1pt]
\definecolor{fillColor}{RGB}{255,255,255}
\path[use as bounding box,fill=fillColor,fill opacity=0.00] (35,0) rectangle (108.54,144.54);
\begin{scope}
\path[clip] ( 38.33, 32.16) rectangle (139.54,139.54);
\definecolor{drawColor}{gray}{0.92}

\path[draw=drawColor,line width= 0.3pt,line join=round] ( 38.33, 49.25) --
	(139.54, 49.25);

\path[draw=drawColor,line width= 0.3pt,line join=round] ( 38.33, 73.65) --
	(139.54, 73.65);

\path[draw=drawColor,line width= 0.3pt,line join=round] ( 38.33, 98.05) --
	(139.54, 98.05);

\path[draw=drawColor,line width= 0.3pt,line join=round] ( 38.33,122.46) --
	(139.54,122.46);

\path[draw=drawColor,line width= 0.3pt,line join=round] ( 42.93, 32.16) --
	( 42.93,139.54);

\path[draw=drawColor,line width= 0.3pt,line join=round] ( 69.22, 32.16) --
	( 69.22,139.54);

\path[draw=drawColor,line width= 0.3pt,line join=round] ( 95.51, 32.16) --
	( 95.51,139.54);

\path[draw=drawColor,line width= 0.3pt,line join=round] (121.79, 32.16) --
	(121.79,139.54);

\path[draw=drawColor,line width= 0.5pt,line join=round] ( 38.33, 37.04) --
	(139.54, 37.04);

\path[draw=drawColor,line width= 0.5pt,line join=round] ( 38.33, 61.45) --
	(139.54, 61.45);

\path[draw=drawColor,line width= 0.5pt,line join=round] ( 38.33, 85.85) --
	(139.54, 85.85);

\path[draw=drawColor,line width= 0.5pt,line join=round] ( 38.33,110.26) --
	(139.54,110.26);

\path[draw=drawColor,line width= 0.5pt,line join=round] ( 38.33,134.66) --
	(139.54,134.66);

\path[draw=drawColor,line width= 0.5pt,line join=round] ( 56.07, 32.16) --
	( 56.07,139.54);

\path[draw=drawColor,line width= 0.5pt,line join=round] ( 82.36, 32.16) --
	( 82.36,139.54);

\path[draw=drawColor,line width= 0.5pt,line join=round] (108.65, 32.16) --
	(108.65,139.54);

\path[draw=drawColor,line width= 0.5pt,line join=round] (134.94, 32.16) --
	(134.94,139.54);
\definecolor{drawColor}{RGB}{69,117,180}

\path[draw=drawColor,line width= 1.4pt,line join=round] ( 42.93,104.40) --
	( 63.99,104.40) --
	( 63.99,105.37) --
	( 65.80,105.37) --
	( 65.80,106.35) --
	( 68.80,106.35) --
	( 68.80,107.33) --
	( 69.10,107.33) --
	( 69.10,108.30) --
	( 69.60,108.30) --
	( 69.60,109.28) --
	( 70.01,109.28) --
	( 70.01,111.23) --
	( 70.11,111.23) --
	( 70.11,112.21) --
	( 70.40,112.21) --
	( 70.40,114.16) --
	( 71.53,114.16) --
	( 71.53,115.14) --
	( 71.79,115.14) --
	( 71.79,116.11) --
	( 72.41,116.11) --
	( 72.41,117.09) --
	( 72.79,117.09) --
	( 72.79,118.06) --
	( 72.91,118.06) --
	( 72.91,119.04) --
	( 73.72,119.04) --
	( 73.72,120.02) --
	( 74.67,120.02) --
	( 74.67,120.99) --
	( 79.00,120.99) --
	( 79.00,121.97) --
	( 79.68,121.97) --
	( 79.68,122.95) --
	( 80.67,122.95) --
	( 80.67,123.92) --
	( 82.15,123.92) --
	( 82.15,124.90) --
	( 83.98,124.90) --
	( 83.98,125.87) --
	( 84.86,125.87) --
	( 84.86,126.85) --
	( 85.51,126.85) --
	( 85.51,127.83) --
	( 85.76,127.83) --
	( 85.76,128.80) --
	( 86.11,128.80) --
	( 86.11,129.78) --
	( 87.55,129.78) --
	( 87.55,130.75) --
	( 87.93,130.75) --
	( 87.93,131.73) --
	( 88.05,131.73) --
	( 88.05,132.71) --
	( 90.25,132.71) --
	( 90.25,133.68) --
	( 92.21,133.68) --
	( 92.21,134.66) --
	(134.94,134.66) --
	(134.94,134.66);
\definecolor{drawColor}{RGB}{252,141,89}

\path[draw=drawColor,line width= 1.4pt,dash pattern=on 2pt off 2pt ,line join=round] ( 42.93,134.66) --
	(134.94,134.66) --
	(134.94,134.66);
\definecolor{drawColor}{RGB}{69,117,180}

\node[text=drawColor,anchor=base west,inner sep=0pt, outer sep=0pt, scale=  1.00] at ( 100,50) {100/100};
\definecolor{drawColor}{RGB}{252,141,89}

\node[text=drawColor,anchor=base west,inner sep=0pt, outer sep=0pt, scale=  1.00] at ( 100,40) {100/100};
\end{scope}
\begin{scope}
\path[clip] (  0.00,  0.00) rectangle (144.54,144.54);
\definecolor{drawColor}{gray}{0.30}

\node[text=drawColor,anchor=base east,inner sep=0pt, outer sep=0pt, scale=  1.00] at ( 33.83, 33.60) {0};

\node[text=drawColor,anchor=base east,inner sep=0pt, outer sep=0pt, scale=  1.00] at ( 33.83, 58.00) {25};

\node[text=drawColor,anchor=base east,inner sep=0pt, outer sep=0pt, scale=  1.00] at ( 33.83, 82.41) {50};

\node[text=drawColor,anchor=base east,inner sep=0pt, outer sep=0pt, scale=  1.00] at ( 33.83,106.81) {75};

\node[text=drawColor,anchor=base east,inner sep=0pt, outer sep=0pt, scale=  1.00] at ( 33.83,131.22) {100};
\end{scope}
\begin{scope}
\path[clip] (  0.00,  0.00) rectangle (144.54,144.54);
\definecolor{drawColor}{gray}{0.30}

\node[text=drawColor,anchor=base,inner sep=0pt, outer sep=0pt, scale=  1.00] at ( 56.07, 20.78) {$10^1$};

\node[text=drawColor,anchor=base,inner sep=0pt, outer sep=0pt, scale=  1.00] at ( 82.36, 20.78) {$10^3$};

\node[text=drawColor,anchor=base,inner sep=0pt, outer sep=0pt, scale=  1.00] at (108.65, 20.78) {$10^5$};
\node[text=drawColor,anchor=base,inner sep=0pt, outer sep=0pt, scale=  1.00] at (134.94, 20.78) {$10^7$};
\end{scope}
\begin{scope}
\path[clip] (  0.00,  0.00) rectangle (144.54,144.54);
\definecolor{drawColor}{RGB}{0,0,0}

\end{scope}
\begin{scope}
\path[clip] (  0.00,  0.00) rectangle (144.54,144.54);
\definecolor{drawColor}{RGB}{0,0,0}

\node[text=drawColor,rotate= 90.00,anchor=base,inner sep=0pt, outer sep=0pt, scale=  1.00] at ( 15, 85.85) {\# of solved instances};
\end{scope}
\end{tikzpicture}}%
    \scalebox{1.0}{\begin{tikzpicture}[x=1pt,y=1pt]
\definecolor{fillColor}{RGB}{255,255,255}
\path[use as bounding box,fill=fillColor,fill opacity=0.00] (0,0) rectangle (113.54,144.54);
\begin{scope}
\path[clip] ( 38.33, 32.16) rectangle (139.54,139.54);
\definecolor{drawColor}{gray}{0.92}

\path[draw=drawColor,line width= 0.3pt,line join=round] ( 38.33, 49.25) --
	(139.54, 49.25);

\path[draw=drawColor,line width= 0.3pt,line join=round] ( 38.33, 73.65) --
	(139.54, 73.65);

\path[draw=drawColor,line width= 0.3pt,line join=round] ( 38.33, 98.05) --
	(139.54, 98.05);

\path[draw=drawColor,line width= 0.3pt,line join=round] ( 38.33,122.46) --
	(139.54,122.46);

\path[draw=drawColor,line width= 0.3pt,line join=round] ( 46.13, 32.16) --
	( 46.13,139.54);

\path[draw=drawColor,line width= 0.3pt,line join=round] ( 81.66, 32.16) --
	( 81.66,139.54);

\path[draw=drawColor,line width= 0.3pt,line join=round] (117.18, 32.16) --
	(117.18,139.54);

\path[draw=drawColor,line width= 0.5pt,line join=round] ( 38.33, 37.04) --
	(139.54, 37.04);

\path[draw=drawColor,line width= 0.5pt,line join=round] ( 38.33, 61.45) --
	(139.54, 61.45);

\path[draw=drawColor,line width= 0.5pt,line join=round] ( 38.33, 85.85) --
	(139.54, 85.85);

\path[draw=drawColor,line width= 0.5pt,line join=round] ( 38.33,110.26) --
	(139.54,110.26);

\path[draw=drawColor,line width= 0.5pt,line join=round] ( 38.33,134.66) --
	(139.54,134.66);

\path[draw=drawColor,line width= 0.5pt,line join=round] ( 63.89, 32.16) --
	( 63.89,139.54);

\path[draw=drawColor,line width= 0.5pt,line join=round] ( 99.42, 32.16) --
	( 99.42,139.54);

\path[draw=drawColor,line width= 0.5pt,line join=round] (134.94, 32.16) --
	(134.94,139.54);
\definecolor{drawColor}{RGB}{69,117,180}

\path[draw=drawColor,line width= 1.4pt,line join=round] ( 42.93, 38.02) --
	( 54.63, 38.02) --
	( 54.63, 39.00) --
	( 56.20, 39.00) --
	( 56.20, 39.97) --
	( 60.43, 39.97) --
	( 60.43, 40.95) --
	( 60.71, 40.95) --
	( 60.71, 41.92) --
	( 60.91, 41.92) --
	( 60.91, 42.90) --
	( 62.35, 42.90) --
	( 62.35, 43.88) --
	( 62.46, 43.88) --
	( 62.46, 44.85) --
	( 62.59, 44.85) --
	( 62.59, 45.83) --
	( 63.79, 45.83) --
	( 63.79, 46.80) --
	( 65.98, 46.80) --
	( 65.98, 47.78) --
	( 67.57, 47.78) --
	( 67.57, 48.76) --
	( 68.96, 48.76) --
	( 68.96, 49.73) --
	( 69.24, 49.73) --
	( 69.24, 50.71) --
	( 69.54, 50.71) --
	( 69.54, 51.69) --
	( 69.67, 51.69) --
	( 69.67, 52.66) --
	( 69.79, 52.66) --
	( 69.79, 53.64) --
	( 69.83, 53.64) --
	( 69.83, 54.61) --
	( 70.16, 54.61) --
	( 70.16, 55.59) --
	( 70.28, 55.59) --
	( 70.28, 56.57) --
	( 70.33, 56.57) --
	( 70.33, 57.54) --
	( 71.34, 57.54) --
	( 71.34, 58.52) --
	( 72.00, 58.52) --
	( 72.00, 59.49) --
	( 72.15, 59.49) --
	( 72.15, 60.47) --
	( 72.27, 60.47) --
	( 72.27, 61.45) --
	( 73.34, 61.45) --
	( 73.34, 62.42) --
	( 73.35, 62.42) --
	( 73.35, 63.40) --
	( 73.71, 63.40) --
	( 73.71, 64.38) --
	( 74.02, 64.38) --
	( 74.02, 65.35) --
	( 74.12, 65.35) --
	( 74.12, 66.33) --
	( 74.46, 66.33) --
	( 74.46, 67.30) --
	( 75.25, 67.30) --
	( 75.25, 68.28) --
	( 75.94, 68.28) --
	( 75.94, 69.26) --
	( 75.96, 69.26) --
	( 75.96, 70.23) --
	( 75.99, 70.23) --
	( 75.99, 71.21) --
	( 76.72, 71.21) --
	( 76.72, 72.18) --
	( 77.02, 72.18) --
	( 77.02, 73.16) --
	( 77.83, 73.16) --
	( 77.83, 74.14) --
	( 77.93, 74.14) --
	( 77.93, 75.11) --
	( 77.98, 75.11) --
	( 77.98, 76.09) --
	( 78.02, 76.09) --
	( 78.02, 77.07) --
	( 78.43, 77.07) --
	( 78.43, 78.04) --
	( 78.46, 78.04) --
	( 78.46, 79.02) --
	( 78.57, 79.02) --
	( 78.57, 79.99) --
	( 78.65, 79.99) --
	( 78.65, 80.97) --
	( 81.03, 80.97) --
	( 81.03, 81.95) --
	( 81.58, 81.95) --
	( 81.58, 82.92) --
	( 81.82, 82.92) --
	( 81.82, 83.90) --
	( 81.91, 83.90) --
	( 81.91, 84.88) --
	( 82.88, 84.88) --
	( 82.88, 85.85) --
	( 82.94, 85.85) --
	( 82.94, 86.83) --
	( 83.06, 86.83) --
	( 83.06, 87.80) --
	( 84.23, 87.80) --
	( 84.23, 88.78) --
	( 84.98, 88.78) --
	( 84.98, 89.76) --
	( 85.05, 89.76) --
	( 85.05, 90.73) --
	( 85.19, 90.73) --
	( 85.19, 91.71) --
	( 86.17, 91.71) --
	( 86.17, 92.68) --
	( 86.36, 92.68) --
	( 86.36, 93.66) --
	( 87.16, 93.66) --
	( 87.16, 94.64) --
	( 87.17, 94.64) --
	( 87.17, 95.61) --
	( 87.18, 95.61) --
	( 87.18, 96.59) --
	( 87.19, 96.59) --
	( 87.19, 97.57) --
	( 87.21, 97.57) --
	( 87.21, 98.54) --
	( 87.69, 98.54) --
	( 87.69, 99.52) --
	( 88.19, 99.52) --
	( 88.19,100.49) --
	( 92.16,100.49) --
	( 92.16,101.47) --
	( 93.89,101.47) --
	( 93.89,102.45) --
	( 94.18,102.45) --
	( 94.18,103.42) --
	( 94.42,103.42) --
	( 94.42,104.40) --
	(134.94,104.40) --
	(134.94,104.40);
\definecolor{drawColor}{RGB}{252,141,89}

\path[draw=drawColor,line width= 1.4pt,dash pattern=on 2pt off 2pt ,line join=round] ( 59.75, 38.02) --
	( 60.51, 38.02) --
	( 60.51, 39.00) --
	( 61.10, 39.00) --
	( 61.10, 39.97) --
	( 61.51, 39.97) --
	( 61.51, 40.95) --
	( 63.55, 40.95) --
	( 63.55, 41.92) --
	( 64.27, 41.92) --
	( 64.27, 42.90) --
	( 64.87, 42.90) --
	( 64.87, 43.88) --
	( 65.04, 43.88) --
	( 65.04, 44.85) --
	( 66.20, 44.85) --
	( 66.20, 45.83) --
	( 66.21, 45.83) --
	( 66.21, 46.80) --
	( 66.95, 46.80) --
	( 66.95, 47.78) --
	( 67.08, 47.78) --
	( 67.08, 48.76) --
	( 68.18, 48.76) --
	( 68.18, 49.73) --
	( 69.33, 49.73) --
	( 69.33, 50.71) --
	( 69.46, 50.71) --
	( 69.46, 51.69) --
	( 69.68, 51.69) --
	( 69.68, 52.66) --
	( 70.00, 52.66) --
	( 70.00, 53.64) --
	( 71.33, 53.64) --
	( 71.33, 54.61) --
	( 71.76, 54.61) --
	( 71.76, 55.59) --
	( 72.54, 55.59) --
	( 72.54, 56.57) --
	( 73.01, 56.57) --
	( 73.01, 57.54) --
	( 73.15, 57.54) --
	( 73.15, 58.52) --
	( 74.27, 58.52) --
	( 74.27, 59.49) --
	( 75.66, 59.49) --
	( 75.66, 60.47) --
	( 76.26, 60.47) --
	( 76.26, 61.45) --
	( 76.47, 61.45) --
	( 76.47, 62.42) --
	( 77.57, 62.42) --
	( 77.57, 63.40) --
	( 77.66, 63.40) --
	( 77.66, 64.38) --
	( 80.51, 64.38) --
	( 80.51, 65.35) --
	( 80.95, 65.35) --
	( 80.95, 66.33) --
	( 81.73, 66.33) --
	( 81.73, 67.30) --
	( 82.03, 67.30) --
	( 82.03, 68.28) --
	( 82.05, 68.28) --
	( 82.05, 69.26) --
	( 82.45, 69.26) --
	( 82.45, 70.23) --
	( 82.69, 70.23) --
	( 82.69, 71.21) --
	( 82.80, 71.21) --
	( 82.80, 72.18) --
	( 82.97, 72.18) --
	( 82.97, 73.16) --
	( 82.99, 73.16) --
	( 82.99, 74.14) --
	( 84.02, 74.14) --
	( 84.02, 75.11) --
	( 84.85, 75.11) --
	( 84.85, 76.09) --
	( 85.05, 76.09) --
	( 85.05, 77.07) --
	( 85.72, 77.07) --
	( 85.72, 78.04) --
	( 86.28, 78.04) --
	( 86.28, 79.02) --
	( 88.10, 79.02) --
	( 88.10, 79.99) --
	( 88.36, 79.99) --
	( 88.36, 80.97) --
	( 88.49, 80.97) --
	( 88.49, 81.95) --
	( 89.47, 81.95) --
	( 89.47, 82.92) --
	( 89.58, 82.92) --
	( 89.58, 83.90) --
	( 91.92, 83.90) --
	( 91.92, 84.88) --
	( 91.97, 84.88) --
	( 91.97, 85.85) --
	( 92.87, 85.85) --
	( 92.87, 86.83) --
	( 93.57, 86.83) --
	( 93.57, 87.80) --
	( 99.37, 87.80) --
	( 99.37, 88.78) --
	(100.03, 88.78) --
	(100.03, 89.76) --
	(100.41, 89.76) --
	(100.41, 90.73) --
	(101.87, 90.73) --
	(101.87, 91.71) --
	(102.25, 91.71) --
	(102.25, 92.68) --
	(103.02, 92.68) --
	(103.02, 93.66) --
	(105.82, 93.66) --
	(105.82, 94.64) --
	(134.94, 94.64) --
	(134.94, 94.64);
\definecolor{drawColor}{RGB}{69,117,180}

\node[text=drawColor,anchor=base west,inner sep=0pt, outer sep=0pt, scale=  1.00] at ( 40,120) {69/100};
\definecolor{drawColor}{RGB}{252,141,89}

\node[text=drawColor,anchor=base west,inner sep=0pt, outer sep=0pt, scale=  1.00] at ( 40,110) {59/100};
\end{scope}
\begin{scope}
\path[clip] (  0.00,  0.00) rectangle (144.54,144.54);
\definecolor{drawColor}{gray}{0.30}

\end{scope}
\begin{scope}
\path[clip] (  0.00,  0.00) rectangle (144.54,144.54);
\definecolor{drawColor}{gray}{0.30}

\node[text=drawColor,anchor=base,inner sep=0pt, outer sep=0pt, scale=  1.00] at ( 63.89, 20.78) {$10^3$};

\node[text=drawColor,anchor=base,inner sep=0pt, outer sep=0pt, scale=  1.00] at ( 99.42, 20.78) {$10^5$};

\node[text=drawColor,anchor=base,inner sep=0pt, outer sep=0pt, scale=  1.00] at (134.94, 20.78) {$10^7$};
\end{scope}
\begin{scope}
\path[clip] (  0.00,  0.00) rectangle (144.54,144.54);
\definecolor{drawColor}{RGB}{0,0,0}

\node[text=drawColor,anchor=base,inner sep=0pt, outer sep=0pt, scale=  1.00] at ( 88.93,  6.94) {Branch-and-bound nodes};
\end{scope}
\begin{scope}
\path[clip] (  0.00,  0.00) rectangle (144.54,144.54);
\definecolor{drawColor}{RGB}{0,0,0}

\end{scope}
\end{tikzpicture}}%
    \scalebox{1.0}{\begin{tikzpicture}[x=1pt,y=1pt]
\definecolor{fillColor}{RGB}{255,255,255}
\path[use as bounding box,fill=fillColor,fill opacity=0.00] (0,0) rectangle (113.54,144.54);
\begin{scope}
\path[clip] ( 38.33, 32.16) rectangle (139.54,139.54);
\definecolor{drawColor}{gray}{0.92}

\path[draw=drawColor,line width= 0.3pt,line join=round] ( 38.33, 49.25) --
	(139.54, 49.25);

\path[draw=drawColor,line width= 0.3pt,line join=round] ( 38.33, 73.65) --
	(139.54, 73.65);

\path[draw=drawColor,line width= 0.3pt,line join=round] ( 38.33, 98.05) --
	(139.54, 98.05);

\path[draw=drawColor,line width= 0.3pt,line join=round] ( 38.33,122.46) --
	(139.54,122.46);

\path[draw=drawColor,line width= 0.3pt,line join=round] ( 42.93, 32.16) --
	( 42.93,139.54);

\path[draw=drawColor,line width= 0.3pt,line join=round] ( 69.22, 32.16) --
	( 69.22,139.54);

\path[draw=drawColor,line width= 0.3pt,line join=round] ( 95.51, 32.16) --
	( 95.51,139.54);

\path[draw=drawColor,line width= 0.3pt,line join=round] (121.79, 32.16) --
	(121.79,139.54);

\path[draw=drawColor,line width= 0.5pt,line join=round] ( 38.33, 37.04) --
	(139.54, 37.04);

\path[draw=drawColor,line width= 0.5pt,line join=round] ( 38.33, 61.45) --
	(139.54, 61.45);

\path[draw=drawColor,line width= 0.5pt,line join=round] ( 38.33, 85.85) --
	(139.54, 85.85);

\path[draw=drawColor,line width= 0.5pt,line join=round] ( 38.33,110.26) --
	(139.54,110.26);

\path[draw=drawColor,line width= 0.5pt,line join=round] ( 38.33,134.66) --
	(139.54,134.66);

\path[draw=drawColor,line width= 0.5pt,line join=round] ( 56.07, 32.16) --
	( 56.07,139.54);

\path[draw=drawColor,line width= 0.5pt,line join=round] ( 82.36, 32.16) --
	( 82.36,139.54);

\path[draw=drawColor,line width= 0.5pt,line join=round] (108.65, 32.16) --
	(108.65,139.54);

\path[draw=drawColor,line width= 0.5pt,line join=round] (134.94, 32.16) --
	(134.94,139.54);
\definecolor{drawColor}{RGB}{69,117,180}

\path[draw=drawColor,line width= 1.4pt,line join=round] ( 42.93, 79.99) --
	( 55.47, 79.99) --
	( 55.47, 80.97) --
	( 57.11, 80.97) --
	( 57.11, 81.95) --
	( 57.57, 81.95) --
	( 57.57, 82.92) --
	( 59.10, 82.92) --
	( 59.10, 83.90) --
	( 61.95, 83.90) --
	( 61.95, 86.83) --
	( 62.53, 86.83) --
	( 62.53, 87.80) --
	( 66.67, 87.80) --
	( 66.67, 88.78) --
	( 67.02, 88.78) --
	( 67.02, 89.76) --
	( 67.34, 89.76) --
	( 67.34, 90.73) --
	( 67.73, 90.73) --
	( 67.73, 91.71) --
	( 73.80, 91.71) --
	( 73.80, 92.68) --
	( 73.90, 92.68) --
	( 73.90, 93.66) --
	( 74.05, 93.66) --
	( 74.05, 94.64) --
	( 75.41, 94.64) --
	( 75.41, 95.61) --
	( 75.71, 95.61) --
	( 75.71, 96.59) --
	( 75.75, 96.59) --
	( 75.75, 97.57) --
	( 76.30, 97.57) --
	( 76.30, 98.54) --
	( 76.34, 98.54) --
	( 76.34, 99.52) --
	( 76.47, 99.52) --
	( 76.47,100.49) --
	( 77.30,100.49) --
	( 77.30,101.47) --
	( 78.81,101.47) --
	( 78.81,102.45) --
	( 79.69,102.45) --
	( 79.69,103.42) --
	( 93.24,103.42) --
	( 93.24,104.40) --
	( 94.36,104.40) --
	( 94.36,105.37) --
	( 94.37,105.37) --
	( 94.37,106.35) --
	(101.19,106.35) --
	(101.19,107.33) --
	(101.49,107.33) --
	(101.49,108.30) --
	(103.12,108.30) --
	(103.12,109.28) --
	(109.66,109.28) --
	(109.66,110.26) --
	(114.45,110.26) --
	(114.45,111.23) --
	(116.56,111.23) --
	(116.56,112.21) --
	(125.49,112.21) --
	(125.49,113.18) --
	(134.94,113.18) --
	(134.94,113.18);
\definecolor{drawColor}{RGB}{252,141,89}

\path[draw=drawColor,line width= 1.4pt,dash pattern=on 2pt off 2pt ,line join=round] ( 42.93, 82.92) --
	( 46.88, 82.92) --
	( 46.88, 83.90) --
	( 49.20, 83.90) --
	( 49.20, 84.88) --
	( 52.12, 84.88) --
	( 52.12, 86.83) --
	( 53.16, 86.83) --
	( 53.16, 88.78) --
	( 54.80, 88.78) --
	( 54.80, 89.76) --
	( 56.07, 89.76) --
	( 56.07, 90.73) --
	( 56.62, 90.73) --
	( 56.62, 91.71) --
	( 57.57, 91.71) --
	( 57.57, 92.68) --
	( 57.99, 92.68) --
	( 57.99, 93.66) --
	( 61.74, 93.66) --
	( 61.74, 94.64) --
	( 61.95, 94.64) --
	( 61.95, 95.61) --
	( 62.53, 95.61) --
	( 62.53, 96.59) --
	( 66.40, 96.59) --
	( 66.40, 97.57) --
	( 67.18, 97.57) --
	( 67.18, 98.54) --
	( 67.94, 98.54) --
	( 67.94, 99.52) --
	( 68.98, 99.52) --
	( 68.98,100.49) --
	( 69.66,100.49) --
	( 69.66,101.47) --
	( 73.45,101.47) --
	( 73.45,102.45) --
	( 75.31,102.45) --
	( 75.31,103.42) --
	( 76.00,103.42) --
	( 76.00,104.40) --
	( 77.09,104.40) --
	( 77.09,105.37) --
	( 78.66,105.37) --
	( 78.66,106.35) --
	( 93.55,106.35) --
	( 93.55,107.33) --
	( 95.88,107.33) --
	( 95.88,108.30) --
	( 97.66,108.30) --
	( 97.66,109.28) --
	( 99.43,109.28) --
	( 99.43,110.26) --
	(105.69,110.26) --
	(105.69,111.23) --
	(107.27,111.23) --
	(107.27,112.21) --
	(107.65,112.21) --
	(107.65,113.18) --
	(110.13,113.18) --
	(110.13,114.16) --
	(110.61,114.16) --
	(110.61,115.14) --
	(111.74,115.14) --
	(111.74,116.11) --
	(113.66,116.11) --
	(113.66,117.09) --
	(114.24,117.09) --
	(114.24,118.06) --
	(115.60,118.06) --
	(115.60,119.04) --
	(120.69,119.04) --
	(120.69,120.02) --
	(123.26,120.02) --
	(123.26,120.99) --
	(129.76,120.99) --
	(129.76,121.97) --
	(134.84,121.97) --
	(134.84,122.95) --
	(134.94,122.95) --
	(134.94,122.95);
\definecolor{drawColor}{RGB}{69,117,180}

\node[text=drawColor,anchor=base west,inner sep=0pt, outer sep=0pt, scale=  1.00] at ( 100,50) {78/100};
\definecolor{drawColor}{RGB}{252,141,89}

\node[text=drawColor,anchor=base west,inner sep=0pt, outer sep=0pt, scale=  1.00] at ( 100,40) {88/100};
\end{scope}
\begin{scope}
\path[clip] (  0.00,  0.00) rectangle (144.54,144.54);
\definecolor{drawColor}{gray}{0.30}

\end{scope}
\begin{scope}
\path[clip] (  0.00,  0.00) rectangle (144.54,144.54);
\definecolor{drawColor}{gray}{0.30}

\node[text=drawColor,anchor=base,inner sep=0pt, outer sep=0pt, scale=  1.00] at ( 56.07, 20.78) {$10^1$};

\node[text=drawColor,anchor=base,inner sep=0pt, outer sep=0pt, scale=  1.00] at ( 82.36, 20.78) {$10^3$};

\node[text=drawColor,anchor=base,inner sep=0pt, outer sep=0pt, scale=  1.00] at (108.65, 20.78) {$10^5$};
\node[text=drawColor,anchor=base,inner sep=0pt, outer sep=0pt, scale=  1.00] at (134.94, 20.78) {$10^7$};
\end{scope}
\begin{scope}
\path[clip] (  0.00,  0.00) rectangle (144.54,144.54);
\definecolor{drawColor}{RGB}{0,0,0}

\end{scope}
\begin{scope}
\path[clip] (  0.00,  0.00) rectangle (144.54,144.54);
\definecolor{drawColor}{RGB}{0,0,0}

\end{scope}
\end{tikzpicture}}%
  } \caption{ECDF of BnB nodes for the shortest-path problem (left), the
  knapsack problem (middle), and the portfolio selection problem (right) with
  continuous budgeted uncertainty sets (top) and continuous knapsack uncertainty
  sets (bottom). The numbers in the legend indicate the number of instances
  solved within the time limit of 2 hours out of 100 instances. Solid
  \blue{blue}: \blue{bilevel approach}. Dotted \orange{orange}: \orange{robust
  approach}.}
  \label{fig:num-results:cont:ecdf-nodes}
\end{figure}
\begin{figure}
  \centering
  \makebox[\textwidth][c]{%
    \scalebox{1.0}{\begin{tikzpicture}[x=1pt,y=1pt]
\definecolor{fillColor}{RGB}{255,255,255}
\path[use as bounding box,fill=fillColor,fill opacity=0.00] (0,0) rectangle (144.54,144.54);
\begin{scope}
\path[clip] ( 38.33, 32.16) rectangle (139.54,139.54);
\definecolor{drawColor}{gray}{0.92}

\path[draw=drawColor,line width= 0.3pt,line join=round] ( 38.33, 49.25) --
	(139.54, 49.25);

\path[draw=drawColor,line width= 0.3pt,line join=round] ( 38.33, 73.65) --
	(139.54, 73.65);

\path[draw=drawColor,line width= 0.3pt,line join=round] ( 38.33, 98.05) --
	(139.54, 98.05);

\path[draw=drawColor,line width= 0.3pt,line join=round] ( 38.33,122.46) --
	(139.54,122.46);

\path[draw=drawColor,line width= 0.3pt,line join=round] ( 46.21, 32.16) --
	( 46.21,139.54);

\path[draw=drawColor,line width= 0.3pt,line join=round] ( 75.79, 32.16) --
	( 75.79,139.54);

\path[draw=drawColor,line width= 0.3pt,line join=round] (105.36, 32.16) --
	(105.36,139.54);

\path[draw=drawColor,line width= 0.3pt,line join=round] (134.94, 32.16) --
	(134.94,139.54);

\path[draw=drawColor,line width= 0.5pt,line join=round] ( 38.33, 37.04) --
	(139.54, 37.04);

\path[draw=drawColor,line width= 0.5pt,line join=round] ( 38.33, 61.45) --
	(139.54, 61.45);

\path[draw=drawColor,line width= 0.5pt,line join=round] ( 38.33, 85.85) --
	(139.54, 85.85);

\path[draw=drawColor,line width= 0.5pt,line join=round] ( 38.33,110.26) --
	(139.54,110.26);

\path[draw=drawColor,line width= 0.5pt,line join=round] ( 38.33,134.66) --
	(139.54,134.66);

\path[draw=drawColor,line width= 0.5pt,line join=round] ( 61.00, 32.16) --
	( 61.00,139.54);

\path[draw=drawColor,line width= 0.5pt,line join=round] ( 90.57, 32.16) --
	( 90.57,139.54);

\path[draw=drawColor,line width= 0.5pt,line join=round] (120.15, 32.16) --
	(120.15,139.54);
\definecolor{drawColor}{RGB}{69,117,180}

\path[draw=drawColor,line width= 1.4pt,line join=round] ( 38.33, 47.78) --
	( 51.83, 47.78) --
	( 51.83, 48.76) --
	( 52.35, 48.76) --
	( 52.35, 49.73) --
	( 52.59, 49.73) --
	( 52.59, 50.71) --
	( 54.78, 50.71) --
	( 54.78, 51.69) --
	( 55.43, 51.69) --
	( 55.43, 52.66) --
	( 55.72, 52.66) --
	( 55.72, 53.64) --
	( 56.01, 53.64) --
	( 56.01, 54.61) --
	( 56.67, 54.61) --
	( 56.67, 55.59) --
	( 56.92, 55.59) --
	( 56.92, 57.54) --
	( 60.18, 57.54) --
	( 60.18, 58.52) --
	( 60.53, 58.52) --
	( 60.53, 59.49) --
	( 63.94, 59.49) --
	( 63.94, 60.47) --
	( 64.77, 60.47) --
	( 64.77, 61.45) --
	( 65.09, 61.45) --
	( 65.09, 62.42) --
	( 65.51, 62.42) --
	( 65.51, 63.40) --
	( 67.23, 63.40) --
	( 67.23, 64.38) --
	( 67.26, 64.38) --
	( 67.26, 65.35) --
	( 68.51, 65.35) --
	( 68.51, 66.33) --
	( 69.52, 66.33) --
	( 69.52, 67.30) --
	( 70.73, 67.30) --
	( 70.73, 68.28) --
	( 71.24, 68.28) --
	( 71.24, 69.26) --
	( 71.71, 69.26) --
	( 71.71, 70.23) --
	( 72.38, 70.23) --
	( 72.38, 71.21) --
	( 73.14, 71.21) --
	( 73.14, 72.18) --
	( 73.95, 72.18) --
	( 73.95, 73.16) --
	( 75.02, 73.16) --
	( 75.02, 74.14) --
	( 75.14, 74.14) --
	( 75.14, 75.11) --
	( 75.19, 75.11) --
	( 75.19, 76.09) --
	( 75.37, 76.09) --
	( 75.37, 77.07) --
	( 75.82, 77.07) --
	( 75.82, 78.04) --
	( 77.08, 78.04) --
	( 77.08, 79.02) --
	( 77.36, 79.02) --
	( 77.36, 79.99) --
	( 78.61, 79.99) --
	( 78.61, 80.97) --
	( 79.60, 80.97) --
	( 79.60, 81.95) --
	( 79.92, 81.95) --
	( 79.92, 82.92) --
	( 80.00, 82.92) --
	( 80.00, 83.90) --
	( 80.32, 83.90) --
	( 80.32, 84.88) --
	( 80.84, 84.88) --
	( 80.84, 85.85) --
	( 80.93, 85.85) --
	( 80.93, 86.83) --
	( 81.40, 86.83) --
	( 81.40, 87.80) --
	( 82.41, 87.80) --
	( 82.41, 88.78) --
	( 82.50, 88.78) --
	( 82.50, 89.76) --
	( 82.51, 89.76) --
	( 82.51, 90.73) --
	( 82.82, 90.73) --
	( 82.82, 91.71) --
	( 82.85, 91.71) --
	( 82.85, 92.68) --
	( 83.45, 92.68) --
	( 83.45, 93.66) --
	( 84.46, 93.66) --
	( 84.46, 94.64) --
	( 85.18, 94.64) --
	( 85.18, 95.61) --
	( 85.71, 95.61) --
	( 85.71, 96.59) --
	( 85.73, 96.59) --
	( 85.73, 97.57) --
	( 86.65, 97.57) --
	( 86.65, 98.54) --
	( 86.68, 98.54) --
	( 86.68, 99.52) --
	( 86.84, 99.52) --
	( 86.84,100.49) --
	( 87.08,100.49) --
	( 87.08,101.47) --
	( 87.13,101.47) --
	( 87.13,102.45) --
	( 87.46,102.45) --
	( 87.46,103.42) --
	( 88.42,103.42) --
	( 88.42,104.40) --
	( 89.56,104.40) --
	( 89.56,105.37) --
	( 89.58,105.37) --
	( 89.58,106.35) --
	( 89.60,106.35) --
	( 89.60,107.33) --
	( 90.43,107.33) --
	( 90.43,108.30) --
	( 91.14,108.30) --
	( 91.14,109.28) --
	( 91.47,109.28) --
	( 91.47,110.26) --
	( 91.77,110.26) --
	( 91.77,111.23) --
	( 92.99,111.23) --
	( 92.99,112.21) --
	( 93.54,112.21) --
	( 93.54,113.18) --
	( 94.55,113.18) --
	( 94.55,114.16) --
	( 95.03,114.16) --
	( 95.03,115.14) --
	( 95.21,115.14) --
	( 95.21,116.11) --
	( 95.87,116.11) --
	( 95.87,117.09) --
	( 96.04,117.09) --
	( 96.04,118.06) --
	( 96.78,118.06) --
	( 96.78,119.04) --
	( 98.45,119.04) --
	( 98.45,120.02) --
	(100.10,120.02) --
	(100.10,120.99) --
	(100.44,120.99) --
	(100.44,121.97) --
	(100.57,121.97) --
	(100.57,122.95) --
	(100.92,122.95) --
	(100.92,123.92) --
	(101.30,123.92) --
	(101.30,124.90) --
	(101.43,124.90) --
	(101.43,125.87) --
	(102.68,125.87) --
	(102.68,126.85) --
	(103.01,126.85) --
	(103.01,127.83) --
	(114.29,127.83) --
	(114.29,128.80) --
	(116.22,128.80) --
	(116.22,129.78) --
	(134.94,129.78) --
	(134.94,129.78);
\definecolor{drawColor}{RGB}{252,141,89}

\path[draw=drawColor,line width= 1.4pt,dash pattern=on 2pt off 2pt ,line join=round] ( 42.93, 46.80) --
	( 69.85, 46.80) --
	( 69.85, 47.78) --
	( 70.48, 47.78) --
	( 70.48, 48.76) --
	( 82.28, 48.76) --
	( 82.28, 49.73) --
	( 84.05, 49.73) --
	( 84.05, 50.71) --
	( 84.36, 50.71) --
	( 84.36, 51.69) --
	( 86.86, 51.69) --
	( 86.86, 52.66) --
	( 87.03, 52.66) --
	( 87.03, 53.64) --
	( 87.14, 53.64) --
	( 87.14, 54.61) --
	( 87.22, 54.61) --
	( 87.22, 55.59) --
	( 87.32, 55.59) --
	( 87.32, 56.57) --
	( 87.43, 56.57) --
	( 87.43, 57.54) --
	( 88.74, 57.54) --
	( 88.74, 58.52) --
	( 95.38, 58.52) --
	( 95.38, 59.49) --
	(114.47, 59.49) --
	(114.47, 60.47) --
	(114.50, 60.47) --
	(114.50, 61.45) --
	(119.14, 61.45) --
	(119.14, 62.42) --
	(122.07, 62.42) --
	(122.07, 63.40) --
	(127.29, 63.40) --
	(127.29, 64.38) --
	(127.36, 64.38) --
	(127.36, 65.35) --
	(134.94, 65.35) --
	(134.94, 65.35);
\definecolor{drawColor}{RGB}{215,25,28}

\path[draw=drawColor,line width= 1.4pt,dash pattern=on 4pt off 2pt ,line join=round] ( 46.21, 46.80) --
	( 71.54, 46.80) --
	( 71.54, 47.78) --
	( 71.64, 47.78) --
	( 71.64, 48.76) --
	( 72.10, 48.76) --
	( 72.10, 49.73) --
	( 72.63, 49.73) --
	( 72.63, 50.71) --
	( 74.43, 50.71) --
	( 74.43, 51.69) --
	( 74.74, 51.69) --
	( 74.74, 52.66) --
	( 75.09, 52.66) --
	( 75.09, 53.64) --
	( 75.21, 53.64) --
	( 75.21, 54.61) --
	( 77.17, 54.61) --
	( 77.17, 55.59) --
	( 88.51, 55.59) --
	( 88.51, 56.57) --
	( 95.04, 56.57) --
	( 95.04, 57.54) --
	( 97.58, 57.54) --
	( 97.58, 58.52) --
	( 99.11, 58.52) --
	( 99.11, 59.49) --
	(105.15, 59.49) --
	(105.15, 60.47) --
	(110.51, 60.47) --
	(110.51, 61.45) --
	(110.90, 61.45) --
	(110.90, 62.42) --
	(111.60, 62.42) --
	(111.60, 63.40) --
	(120.94, 63.40) --
	(120.94, 64.38) --
	(121.65, 64.38) --
	(121.65, 65.35) --
	(122.41, 65.35) --
	(122.41, 66.33) --
	(126.93, 66.33) --
	(126.93, 67.30) --
	(129.60, 67.30) --
	(129.60, 68.28) --
	(134.94, 68.28) --
	(134.94, 68.28);
\definecolor{drawColor}{RGB}{69,117,180}

\node[text=drawColor,anchor=base west,inner sep=0pt, outer sep=0pt, scale=  1.00] at ( 40,120) {95/100};
\definecolor{drawColor}{RGB}{252,141,89}

\node[text=drawColor,anchor=base west,inner sep=0pt, outer sep=0pt, scale=  1.00] at ( 40,110) {29/100};
\definecolor{drawColor}{RGB}{215,25,28}

\node[text=drawColor,anchor=base west,inner sep=0pt, outer sep=0pt, scale=  1.00] at ( 40,100) {32/100};
\end{scope}
\begin{scope}
\path[clip] (  0.00,  0.00) rectangle (144.54,144.54);
\definecolor{drawColor}{gray}{0.30}

\node[text=drawColor,anchor=base east,inner sep=0pt, outer sep=0pt, scale=  1.00] at ( 33.83, 33.60) {0};

\node[text=drawColor,anchor=base east,inner sep=0pt, outer sep=0pt, scale=  1.00] at ( 33.83, 58.00) {25};

\node[text=drawColor,anchor=base east,inner sep=0pt, outer sep=0pt, scale=  1.00] at ( 33.83, 82.41) {50};

\node[text=drawColor,anchor=base east,inner sep=0pt, outer sep=0pt, scale=  1.00] at ( 33.83,106.81) {75};

\node[text=drawColor,anchor=base east,inner sep=0pt, outer sep=0pt, scale=  1.00] at ( 33.83,131.22) {100};
\end{scope}
\begin{scope}
\path[clip] (  0.00,  0.00) rectangle (144.54,144.54);
\definecolor{drawColor}{gray}{0.30}

\node[text=drawColor,anchor=base,inner sep=0pt, outer sep=0pt, scale=  1.00] at ( 61.00, 20.78) {$10^2$};

\node[text=drawColor,anchor=base,inner sep=0pt, outer sep=0pt, scale=  1.00] at ( 90.57, 20.78) {$10^4$};

\node[text=drawColor,anchor=base,inner sep=0pt, outer sep=0pt, scale=  1.00] at (120.15, 20.78) {$10^6$};
\end{scope}
\begin{scope}
\path[clip] (  0.00,  0.00) rectangle (144.54,144.54);
\definecolor{drawColor}{RGB}{0,0,0}

\node[text=drawColor,anchor=base,inner sep=0pt, outer sep=0pt, scale=  1.00] at ( 88.93,  6.94) {Branch-and-bound nodes};
\end{scope}
\begin{scope}
\path[clip] (  0.00,  0.00) rectangle (144.54,144.54);
\definecolor{drawColor}{RGB}{0,0,0}

\node[text=drawColor,rotate= 90.00,anchor=base,inner sep=0pt, outer sep=0pt, scale=  1.00] at ( 11.89, 85.85) {\# of solved instances};
\end{scope}
\end{tikzpicture}}%
    \scalebox{1.0}{\begin{tikzpicture}[x=1pt,y=1pt]
\definecolor{fillColor}{RGB}{255,255,255}
\path[use as bounding box,fill=fillColor,fill opacity=0.00] (0,0) rectangle (144.54,144.54);
\begin{scope}
\path[clip] ( 38.33, 32.16) rectangle (139.54,139.54);
\definecolor{drawColor}{gray}{0.92}

\path[draw=drawColor,line width= 0.3pt,line join=round] ( 38.33, 49.25) --
	(139.54, 49.25);

\path[draw=drawColor,line width= 0.3pt,line join=round] ( 38.33, 73.65) --
	(139.54, 73.65);

\path[draw=drawColor,line width= 0.3pt,line join=round] ( 38.33, 98.05) --
	(139.54, 98.05);

\path[draw=drawColor,line width= 0.3pt,line join=round] ( 38.33,122.46) --
	(139.54,122.46);

\path[draw=drawColor,line width= 0.3pt,line join=round] ( 48.65, 32.16) --
	( 48.65,139.54);

\path[draw=drawColor,line width= 0.3pt,line join=round] ( 77.41, 32.16) --
	( 77.41,139.54);

\path[draw=drawColor,line width= 0.3pt,line join=round] (106.18, 32.16) --
	(106.18,139.54);

\path[draw=drawColor,line width= 0.3pt,line join=round] (134.94, 32.16) --
	(134.94,139.54);

\path[draw=drawColor,line width= 0.5pt,line join=round] ( 38.33, 37.04) --
	(139.54, 37.04);

\path[draw=drawColor,line width= 0.5pt,line join=round] ( 38.33, 61.45) --
	(139.54, 61.45);

\path[draw=drawColor,line width= 0.5pt,line join=round] ( 38.33, 85.85) --
	(139.54, 85.85);

\path[draw=drawColor,line width= 0.5pt,line join=round] ( 38.33,110.26) --
	(139.54,110.26);

\path[draw=drawColor,line width= 0.5pt,line join=round] ( 38.33,134.66) --
	(139.54,134.66);

\path[draw=drawColor,line width= 0.5pt,line join=round] ( 63.03, 32.16) --
	( 63.03,139.54);

\path[draw=drawColor,line width= 0.5pt,line join=round] ( 91.80, 32.16) --
	( 91.80,139.54);

\path[draw=drawColor,line width= 0.5pt,line join=round] (120.56, 32.16) --
	(120.56,139.54);
\definecolor{drawColor}{RGB}{69,117,180}

\path[draw=drawColor,line width= 1.4pt,line join=round] ( 38.33, 46.80) --
	( 42.93, 46.80) --
	( 42.93, 48.76) --
	( 44.32, 48.76) --
	( 44.32, 49.73) --
	( 48.65, 49.73) --
	( 48.65, 50.71) --
	( 53.58, 50.71) --
	( 53.58, 51.69) --
	( 55.72, 51.69) --
	( 55.72, 52.66) --
	( 56.11, 52.66) --
	( 56.11, 53.64) --
	( 74.20, 53.64) --
	( 74.20, 54.61) --
	( 76.32, 54.61) --
	( 76.32, 55.59) --
	( 77.03, 55.59) --
	( 77.03, 56.57) --
	( 82.72, 56.57) --
	( 82.72, 57.54) --
	( 87.23, 57.54) --
	( 87.23, 58.52) --
	( 87.36, 58.52) --
	( 87.36, 59.49) --
	( 95.27, 59.49) --
	( 95.27, 60.47) --
	( 95.68, 60.47) --
	( 95.68, 61.45) --
	( 97.26, 61.45) --
	( 97.26, 62.42) --
	(102.41, 62.42) --
	(102.41, 63.40) --
	(106.27, 63.40) --
	(106.27, 64.38) --
	(106.41, 64.38) --
	(106.41, 65.35) --
	(106.85, 65.35) --
	(106.85, 66.33) --
	(109.92, 66.33) --
	(109.92, 67.30) --
	(114.34, 67.30) --
	(114.34, 68.28) --
	(118.19, 68.28) --
	(118.19, 69.26) --
	(121.07, 69.26) --
	(121.07, 70.23) --
	(122.61, 70.23) --
	(122.61, 71.21) --
	(123.00, 71.21) --
	(123.00, 72.18) --
	(134.94, 72.18) --
	(134.94, 72.18);
\definecolor{drawColor}{RGB}{252,141,89}

\path[draw=drawColor,line width= 1.4pt,dash pattern=on 2pt off 2pt ,line join=round] ( 67.84, 38.02) --
	( 68.29, 38.02) --
	( 68.29, 39.00) --
	( 68.85, 39.00) --
	( 68.85, 39.97) --
	( 69.21, 39.97) --
	( 69.21, 40.95) --
	( 69.46, 40.95) --
	( 69.46, 41.92) --
	( 69.49, 41.92) --
	( 69.49, 42.90) --
	( 69.62, 42.90) --
	( 69.62, 43.88) --
	( 70.28, 43.88) --
	( 70.28, 44.85) --
	( 71.69, 44.85) --
	( 71.69, 45.83) --
	( 77.46, 45.83) --
	( 77.46, 46.80) --
	( 77.82, 46.80) --
	( 77.82, 47.78) --
	( 78.56, 47.78) --
	( 78.56, 48.76) --
	( 79.21, 48.76) --
	( 79.21, 49.73) --
	( 80.53, 49.73) --
	( 80.53, 50.71) --
	( 81.19, 50.71) --
	( 81.19, 51.69) --
	( 83.54, 51.69) --
	( 83.54, 52.66) --
	( 85.82, 52.66) --
	( 85.82, 53.64) --
	( 97.57, 53.64) --
	( 97.57, 54.61) --
	( 98.24, 54.61) --
	( 98.24, 55.59) --
	(100.44, 55.59) --
	(100.44, 56.57) --
	(106.29, 56.57) --
	(106.29, 57.54) --
	(113.69, 57.54) --
	(113.69, 58.52) --
	(115.73, 58.52) --
	(115.73, 59.49) --
	(131.02, 59.49) --
	(131.02, 60.47) --
	(131.65, 60.47) --
	(131.65, 61.45) --
	(134.94, 61.45) --
	(134.94, 61.45);
\definecolor{drawColor}{RGB}{215,25,28}

\path[draw=drawColor,line width= 1.4pt,dash pattern=on 4pt off 2pt ,line join=round] ( 69.07, 38.02) --
	( 69.40, 38.02) --
	( 69.40, 39.00) --
	( 70.22, 39.00) --
	( 70.22, 39.97) --
	( 70.56, 39.97) --
	( 70.56, 40.95) --
	( 70.82, 40.95) --
	( 70.82, 41.92) --
	( 70.87, 41.92) --
	( 70.87, 42.90) --
	( 71.69, 42.90) --
	( 71.69, 43.88) --
	( 71.94, 43.88) --
	( 71.94, 44.85) --
	( 75.47, 44.85) --
	( 75.47, 45.83) --
	( 79.69, 45.83) --
	( 79.69, 46.80) --
	( 80.58, 46.80) --
	( 80.58, 47.78) --
	( 80.61, 47.78) --
	( 80.61, 48.76) --
	( 82.53, 48.76) --
	( 82.53, 49.73) --
	( 85.99, 49.73) --
	( 85.99, 50.71) --
	( 91.75, 50.71) --
	( 91.75, 51.69) --
	(121.75, 51.69) --
	(121.75, 52.66) --
	(132.61, 52.66) --
	(132.61, 53.64) --
	(134.94, 53.64) --
	(134.94, 53.64);
\definecolor{drawColor}{RGB}{69,117,180}

\node[text=drawColor,anchor=base west,inner sep=0pt, outer sep=0pt, scale=  1.00] at ( 40,120) {36/100};
\definecolor{drawColor}{RGB}{252,141,89}

\node[text=drawColor,anchor=base west,inner sep=0pt, outer sep=0pt, scale=  1.00] at ( 40,110) {25/100};
\definecolor{drawColor}{RGB}{215,25,28}

\node[text=drawColor,anchor=base west,inner sep=0pt, outer sep=0pt, scale=  1.00] at ( 40,100) {17/100};
\end{scope}
\begin{scope}
\path[clip] (  0.00,  0.00) rectangle (144.54,144.54);
\definecolor{drawColor}{gray}{0.30}

\node[text=drawColor,anchor=base east,inner sep=0pt, outer sep=0pt, scale=  1.00] at ( 33.83, 33.60) {0};

\node[text=drawColor,anchor=base east,inner sep=0pt, outer sep=0pt, scale=  1.00] at ( 33.83, 58.00) {25};

\node[text=drawColor,anchor=base east,inner sep=0pt, outer sep=0pt, scale=  1.00] at ( 33.83, 82.41) {50};

\node[text=drawColor,anchor=base east,inner sep=0pt, outer sep=0pt, scale=  1.00] at ( 33.83,106.81) {75};

\node[text=drawColor,anchor=base east,inner sep=0pt, outer sep=0pt, scale=  1.00] at ( 33.83,131.22) {100};
\end{scope}
\begin{scope}
\path[clip] (  0.00,  0.00) rectangle (144.54,144.54);
\definecolor{drawColor}{gray}{0.30}

\node[text=drawColor,anchor=base,inner sep=0pt, outer sep=0pt, scale=  1.00] at ( 63.03, 20.78) {$10^2$};

\node[text=drawColor,anchor=base,inner sep=0pt, outer sep=0pt, scale=  1.00] at ( 91.80, 20.78) {$10^4$};

\node[text=drawColor,anchor=base,inner sep=0pt, outer sep=0pt, scale=  1.00] at (120.56, 20.78) {$10^6$};
\end{scope}
\begin{scope}
\path[clip] (  0.00,  0.00) rectangle (144.54,144.54);
\definecolor{drawColor}{RGB}{0,0,0}

\node[text=drawColor,anchor=base,inner sep=0pt, outer sep=0pt, scale=  1.00] at ( 88.93,  6.94) {Branch-and-bound nodes};
\end{scope}
\end{tikzpicture}}%
  }
  \makebox[\textwidth][c]{%
    \scalebox{1.0}{\begin{tikzpicture}[x=1pt,y=1pt]
\definecolor{fillColor}{RGB}{255,255,255}
\path[use as bounding box,fill=fillColor,fill opacity=0.00] (0,0) rectangle (144.54,144.54);
\begin{scope}
\path[clip] ( 38.33, 32.16) rectangle (139.54,139.54);
\definecolor{drawColor}{gray}{0.92}

\path[draw=drawColor,line width= 0.3pt,line join=round] ( 38.33, 49.25) --
	(139.54, 49.25);

\path[draw=drawColor,line width= 0.3pt,line join=round] ( 38.33, 73.65) --
	(139.54, 73.65);

\path[draw=drawColor,line width= 0.3pt,line join=round] ( 38.33, 98.05) --
	(139.54, 98.05);

\path[draw=drawColor,line width= 0.3pt,line join=round] ( 38.33,122.46) --
	(139.54,122.46);

\path[draw=drawColor,line width= 0.3pt,line join=round] ( 42.93, 32.16) --
	( 42.93,139.54);

\path[draw=drawColor,line width= 0.3pt,line join=round] ( 69.18, 32.16) --
	( 69.18,139.54);

\path[draw=drawColor,line width= 0.3pt,line join=round] ( 95.44, 32.16) --
	( 95.44,139.54);

\path[draw=drawColor,line width= 0.3pt,line join=round] (121.69, 32.16) --
	(121.69,139.54);

\path[draw=drawColor,line width= 0.5pt,line join=round] ( 38.33, 37.04) --
	(139.54, 37.04);

\path[draw=drawColor,line width= 0.5pt,line join=round] ( 38.33, 61.45) --
	(139.54, 61.45);

\path[draw=drawColor,line width= 0.5pt,line join=round] ( 38.33, 85.85) --
	(139.54, 85.85);

\path[draw=drawColor,line width= 0.5pt,line join=round] ( 38.33,110.26) --
	(139.54,110.26);

\path[draw=drawColor,line width= 0.5pt,line join=round] ( 38.33,134.66) --
	(139.54,134.66);

\path[draw=drawColor,line width= 0.5pt,line join=round] ( 56.05, 32.16) --
	( 56.05,139.54);

\path[draw=drawColor,line width= 0.5pt,line join=round] ( 82.31, 32.16) --
	( 82.31,139.54);

\path[draw=drawColor,line width= 0.5pt,line join=round] (108.56, 32.16) --
	(108.56,139.54);

\path[draw=drawColor,line width= 0.5pt,line join=round] (134.82, 32.16) --
	(134.82,139.54);
\definecolor{drawColor}{RGB}{69,117,180}

\path[draw=drawColor,line width= 1.4pt,line join=round] ( 38.33,128.80) --
	( 69.78,128.80) --
	( 69.78,129.78) --
	(134.82,129.78) --
	(134.82,129.78);
\definecolor{drawColor}{RGB}{252,141,89}

\path[draw=drawColor,line width= 1.4pt,dash pattern=on 2pt off 2pt ,line join=round] ( 49.19, 46.80) --
	( 75.03, 46.80) --
	( 75.03, 47.78) --
	( 75.76, 47.78) --
	( 75.76, 48.76) --
	( 76.07, 48.76) --
	( 76.07, 49.73) --
	( 76.24, 49.73) --
	( 76.24, 50.71) --
	( 76.29, 50.71) --
	( 76.29, 51.69) --
	( 76.31, 51.69) --
	( 76.31, 52.66) --
	( 76.58, 52.66) --
	( 76.58, 53.64) --
	( 76.61, 53.64) --
	( 76.61, 54.61) --
	( 76.67, 54.61) --
	( 76.67, 55.59) --
	( 77.03, 55.59) --
	( 77.03, 56.57) --
	( 87.31, 56.57) --
	( 87.31, 57.54) --
	( 87.66, 57.54) --
	( 87.66, 58.52) --
	( 91.53, 58.52) --
	( 91.53, 59.49) --
	( 92.17, 59.49) --
	( 92.17, 60.47) --
	( 92.19, 60.47) --
	( 92.19, 61.45) --
	( 92.19, 61.45) --
	( 92.19, 62.42) --
	( 92.45, 62.42) --
	( 92.45, 63.40) --
	( 92.84, 63.40) --
	( 92.84, 64.38) --
	( 93.05, 64.38) --
	( 93.05, 65.35) --
	( 93.12, 65.35) --
	( 93.12, 66.33) --
	( 93.22, 66.33) --
	( 93.22, 67.30) --
	( 93.48, 67.30) --
	( 93.48, 68.28) --
	( 93.69, 68.28) --
	( 93.69, 69.26) --
	( 94.88, 69.26) --
	( 94.88, 70.23) --
	( 95.13, 70.23) --
	( 95.13, 71.21) --
	( 96.24, 71.21) --
	( 96.24, 72.18) --
	( 97.10, 72.18) --
	( 97.10, 73.16) --
	( 98.18, 73.16) --
	( 98.18, 74.14) --
	( 98.36, 74.14) --
	( 98.36, 75.11) --
	(100.14, 75.11) --
	(100.14, 76.09) --
	(103.20, 76.09) --
	(103.20, 77.07) --
	(103.87, 77.07) --
	(103.87, 78.04) --
	(104.96, 78.04) --
	(104.96, 79.02) --
	(106.27, 79.02) --
	(106.27, 79.99) --
	(106.40, 79.99) --
	(106.40, 80.97) --
	(107.12, 80.97) --
	(107.12, 81.95) --
	(109.17, 81.95) --
	(109.17, 82.92) --
	(109.36, 82.92) --
	(109.36, 83.90) --
	(109.57, 83.90) --
	(109.57, 84.88) --
	(109.60, 84.88) --
	(109.60, 85.85) --
	(110.33, 85.85) --
	(110.33, 86.83) --
	(111.67, 86.83) --
	(111.67, 87.80) --
	(112.36, 87.80) --
	(112.36, 88.78) --
	(113.25, 88.78) --
	(113.25, 89.76) --
	(115.94, 89.76) --
	(115.94, 90.73) --
	(116.28, 90.73) --
	(116.28, 91.71) --
	(116.61, 91.71) --
	(116.61, 92.68) --
	(117.41, 92.68) --
	(117.41, 93.66) --
	(118.52, 93.66) --
	(118.52, 94.64) --
	(118.72, 94.64) --
	(118.72, 95.61) --
	(119.09, 95.61) --
	(119.09, 96.59) --
	(119.37, 96.59) --
	(119.37, 97.57) --
	(119.68, 97.57) --
	(119.68, 98.54) --
	(121.63, 98.54) --
	(121.63, 99.52) --
	(123.28, 99.52) --
	(123.28,100.49) --
	(123.68,100.49) --
	(123.68,101.47) --
	(126.45,101.47) --
	(126.45,102.45) --
	(126.57,102.45) --
	(126.57,103.42) --
	(127.29,103.42) --
	(127.29,104.40) --
	(127.67,104.40) --
	(127.67,105.37) --
	(128.24,105.37) --
	(128.24,106.35) --
	(128.48,106.35) --
	(128.48,107.33) --
	(128.60,107.33) --
	(128.60,108.30) --
	(130.27,108.30) --
	(130.27,109.28) --
	(131.04,109.28) --
	(131.04,110.26) --
	(132.81,110.26) --
	(132.81,111.23) --
	(133.38,111.23) --
	(133.38,112.21) --
	(133.56,112.21) --
	(133.56,113.18) --
	(134.49,113.18) --
	(134.49,114.16) --
	(134.69,114.16) --
	(134.69,115.14) --
	(134.82,115.14) --
	(134.82,117.09) --
	(134.92,117.09) --
	(134.92,116.11) --
	(134.94,116.11) --
	(134.94,117.09);
\definecolor{drawColor}{RGB}{215,25,28}

\path[draw=drawColor,line width= 1.4pt,dash pattern=on 4pt off 2pt ,line join=round] ( 42.93, 38.02) --
	( 46.88, 38.02) --
	( 46.88, 46.80) --
	( 75.43, 46.80) --
	( 75.43, 47.78) --
	( 75.48, 47.78) --
	( 75.48, 48.76) --
	( 75.54, 48.76) --
	( 75.54, 49.73) --
	( 75.61, 49.73) --
	( 75.61, 50.71) --
	( 75.80, 50.71) --
	( 75.80, 51.69) --
	( 75.83, 51.69) --
	( 75.83, 52.66) --
	( 76.24, 52.66) --
	( 76.24, 53.64) --
	( 76.32, 53.64) --
	( 76.32, 54.61) --
	( 76.39, 54.61) --
	( 76.39, 55.59) --
	( 76.85, 55.59) --
	( 76.85, 56.57) --
	( 91.97, 56.57) --
	( 91.97, 57.54) --
	( 92.20, 57.54) --
	( 92.20, 58.52) --
	( 92.35, 58.52) --
	( 92.35, 59.49) --
	( 92.45, 59.49) --
	( 92.45, 60.47) --
	( 92.78, 60.47) --
	( 92.78, 61.45) --
	( 92.91, 61.45) --
	( 92.91, 62.42) --
	(101.67, 62.42) --
	(101.67, 63.40) --
	(134.82, 63.40) --
	(134.82, 63.40);
\definecolor{drawColor}{RGB}{69,117,180}

\node[text=drawColor,anchor=base west,inner sep=0pt, outer sep=0pt, scale=  1.00] at ( 40,120) {95/100};
\definecolor{drawColor}{RGB}{252,141,89}

\node[text=drawColor,anchor=base west,inner sep=0pt, outer sep=0pt, scale=  1.00] at ( 40,110) {82/100};
\definecolor{drawColor}{RGB}{215,25,28}

\node[text=drawColor,anchor=base west,inner sep=0pt, outer sep=0pt, scale=  1.00] at ( 40,100) {27/100};
\end{scope}
\begin{scope}
\path[clip] (  0.00,  0.00) rectangle (144.54,144.54);
\definecolor{drawColor}{gray}{0.30}

\node[text=drawColor,anchor=base east,inner sep=0pt, outer sep=0pt, scale=  1.00] at ( 33.83, 33.60) {0};

\node[text=drawColor,anchor=base east,inner sep=0pt, outer sep=0pt, scale=  1.00] at ( 33.83, 58.00) {25};

\node[text=drawColor,anchor=base east,inner sep=0pt, outer sep=0pt, scale=  1.00] at ( 33.83, 82.41) {50};

\node[text=drawColor,anchor=base east,inner sep=0pt, outer sep=0pt, scale=  1.00] at ( 33.83,106.81) {75};

\node[text=drawColor,anchor=base east,inner sep=0pt, outer sep=0pt, scale=  1.00] at ( 33.83,131.22) {100};
\end{scope}
\begin{scope}
\path[clip] (  0.00,  0.00) rectangle (144.54,144.54);
\definecolor{drawColor}{gray}{0.30}

\node[text=drawColor,anchor=base,inner sep=0pt, outer sep=0pt, scale=  1.00] at ( 56.05, 20.78) {$10^1$};

\node[text=drawColor,anchor=base,inner sep=0pt, outer sep=0pt, scale=  1.00] at ( 82.31, 20.78) {$10^3$};

\node[text=drawColor,anchor=base,inner sep=0pt, outer sep=0pt, scale=  1.00] at (108.56, 20.78) {$10^5$};

\node[text=drawColor,anchor=base,inner sep=0pt, outer sep=0pt, scale=  1.00] at (134.82, 20.78) {$10^7$};
\end{scope}
\begin{scope}
\path[clip] (  0.00,  0.00) rectangle (144.54,144.54);
\definecolor{drawColor}{RGB}{0,0,0}

\node[text=drawColor,anchor=base,inner sep=0pt, outer sep=0pt, scale=  1.00] at ( 88.93,  6.94) {Branch-and-bound nodes};
\end{scope}
\begin{scope}
\path[clip] (  0.00,  0.00) rectangle (144.54,144.54);
\definecolor{drawColor}{RGB}{0,0,0}

\node[text=drawColor,rotate= 90.00,anchor=base,inner sep=0pt, outer sep=0pt, scale=  1.00] at ( 11.89, 85.85) {\# of solved instances};
\end{scope}
\end{tikzpicture}}%
    \scalebox{1.0}{\begin{tikzpicture}[x=1pt,y=1pt]
\definecolor{fillColor}{RGB}{255,255,255}
\path[use as bounding box,fill=fillColor,fill opacity=0.00] (0,0) rectangle (144.54,144.54);
\begin{scope}
\path[clip] ( 38.33, 32.16) rectangle (139.54,139.54);
\definecolor{drawColor}{gray}{0.92}

\path[draw=drawColor,line width= 0.3pt,line join=round] ( 38.33, 49.25) --
	(139.54, 49.25);

\path[draw=drawColor,line width= 0.3pt,line join=round] ( 38.33, 73.65) --
	(139.54, 73.65);

\path[draw=drawColor,line width= 0.3pt,line join=round] ( 38.33, 98.05) --
	(139.54, 98.05);

\path[draw=drawColor,line width= 0.3pt,line join=round] ( 38.33,122.46) --
	(139.54,122.46);

\path[draw=drawColor,line width= 0.3pt,line join=round] ( 42.42, 32.16) --
	( 42.42,139.54);

\path[draw=drawColor,line width= 0.3pt,line join=round] ( 77.32, 32.16) --
	( 77.32,139.54);

\path[draw=drawColor,line width= 0.3pt,line join=round] (112.23, 32.16) --
	(112.23,139.54);

\path[draw=drawColor,line width= 0.5pt,line join=round] ( 38.33, 37.04) --
	(139.54, 37.04);

\path[draw=drawColor,line width= 0.5pt,line join=round] ( 38.33, 61.45) --
	(139.54, 61.45);

\path[draw=drawColor,line width= 0.5pt,line join=round] ( 38.33, 85.85) --
	(139.54, 85.85);

\path[draw=drawColor,line width= 0.5pt,line join=round] ( 38.33,110.26) --
	(139.54,110.26);

\path[draw=drawColor,line width= 0.5pt,line join=round] ( 38.33,134.66) --
	(139.54,134.66);

\path[draw=drawColor,line width= 0.5pt,line join=round] ( 59.87, 32.16) --
	( 59.87,139.54);

\path[draw=drawColor,line width= 0.5pt,line join=round] ( 94.78, 32.16) --
	( 94.78,139.54);

\path[draw=drawColor,line width= 0.5pt,line join=round] (129.69, 32.16) --
	(129.69,139.54);
\definecolor{drawColor}{RGB}{69,117,180}

\path[draw=drawColor,line width= 1.4pt,line join=round] ( 42.93, 38.02) --
	( 45.44, 38.02) --
	( 45.44, 39.00) --
	( 46.53, 39.00) --
	( 46.53, 39.97) --
	( 49.42, 39.97) --
	( 49.42, 40.95) --
	( 51.14, 40.95) --
	( 51.14, 41.92) --
	( 51.47, 41.92) --
	( 51.47, 42.90) --
	( 52.65, 42.90) --
	( 52.65, 43.88) --
	( 56.07, 43.88) --
	( 56.07, 44.85) --
	( 57.74, 44.85) --
	( 57.74, 45.83) --
	( 59.95, 45.83) --
	( 59.95, 46.80) --
	( 60.36, 46.80) --
	( 60.36, 47.78) --
	( 62.63, 47.78) --
	( 62.63, 48.76) --
	( 63.64, 48.76) --
	( 63.64, 49.73) --
	( 63.70, 49.73) --
	( 63.70, 50.71) --
	( 64.19, 50.71) --
	( 64.19, 51.69) --
	( 64.44, 51.69) --
	( 64.44, 52.66) --
	( 64.83, 52.66) --
	( 64.83, 53.64) --
	( 66.41, 53.64) --
	( 66.41, 54.61) --
	( 68.68, 54.61) --
	( 68.68, 55.59) --
	( 68.94, 55.59) --
	( 68.94, 56.57) --
	( 70.71, 56.57) --
	( 70.71, 57.54) --
	( 71.53, 57.54) --
	( 71.53, 58.52) --
	( 72.85, 58.52) --
	( 72.85, 59.49) --
	( 73.98, 59.49) --
	( 73.98, 60.47) --
	( 75.89, 60.47) --
	( 75.89, 61.45) --
	( 76.59, 61.45) --
	( 76.59, 62.42) --
	( 76.95, 62.42) --
	( 76.95, 63.40) --
	( 78.01, 63.40) --
	( 78.01, 64.38) --
	( 79.08, 64.38) --
	( 79.08, 65.35) --
	( 79.40, 65.35) --
	( 79.40, 66.33) --
	( 80.17, 66.33) --
	( 80.17, 67.30) --
	( 80.66, 67.30) --
	( 80.66, 68.28) --
	( 84.67, 68.28) --
	( 84.67, 69.26) --
	( 85.12, 69.26) --
	( 85.12, 70.23) --
	( 85.89, 70.23) --
	( 85.89, 71.21) --
	( 86.68, 71.21) --
	( 86.68, 72.18) --
	( 87.00, 72.18) --
	( 87.00, 73.16) --
	( 89.94, 73.16) --
	( 89.94, 74.14) --
	( 90.62, 74.14) --
	( 90.62, 75.11) --
	( 92.05, 75.11) --
	( 92.05, 76.09) --
	( 94.15, 76.09) --
	( 94.15, 77.07) --
	( 94.81, 77.07) --
	( 94.81, 78.04) --
	( 95.02, 78.04) --
	( 95.02, 79.02) --
	( 95.87, 79.02) --
	( 95.87, 79.99) --
	( 96.03, 79.99) --
	( 96.03, 80.97) --
	( 96.12, 80.97) --
	( 96.12, 81.95) --
	( 97.64, 81.95) --
	( 97.64, 82.92) --
	( 97.75, 82.92) --
	( 97.75, 83.90) --
	( 98.67, 83.90) --
	( 98.67, 84.88) --
	( 99.74, 84.88) --
	( 99.74, 85.85) --
	( 99.85, 85.85) --
	( 99.85, 86.83) --
	(100.22, 86.83) --
	(100.22, 87.80) --
	(103.84, 87.80) --
	(103.84, 88.78) --
	(104.02, 88.78) --
	(104.02, 89.76) --
	(105.41, 89.76) --
	(105.41, 90.73) --
	(106.19, 90.73) --
	(106.19, 91.71) --
	(109.15, 91.71) --
	(109.15, 92.68) --
	(109.70, 92.68) --
	(109.70, 93.66) --
	(111.25, 93.66) --
	(111.25, 94.64) --
	(112.42, 94.64) --
	(112.42, 95.61) --
	(113.25, 95.61) --
	(113.25, 96.59) --
	(114.86, 96.59) --
	(114.86, 97.57) --
	(114.88, 97.57) --
	(114.88, 98.54) --
	(115.42, 98.54) --
	(115.42, 99.52) --
	(115.99, 99.52) --
	(115.99,100.49) --
	(116.27,100.49) --
	(116.27,101.47) --
	(116.47,101.47) --
	(116.47,102.45) --
	(116.61,102.45) --
	(116.61,103.42) --
	(121.20,103.42) --
	(121.20,104.40) --
	(121.20,104.40) --
	(121.20,105.37) --
	(122.19,105.37) --
	(122.19,106.35) --
	(126.44,106.35) --
	(126.44,107.33) --
	(126.48,107.33) --
	(126.48,108.30) --
	(127.72,108.30) --
	(127.72,109.28) --
	(128.74,109.28) --
	(128.74,110.26) --
	(128.91,110.26) --
	(128.91,111.23) --
	(129.69,111.23) --
	(129.69,114.16) --
	(132.09,114.16) --
	(132.09,112.21) --
	(132.58,112.21) --
	(132.58,113.18) --
	(134.94,113.18) --
	(134.94,114.16);
\definecolor{drawColor}{RGB}{252,141,89}

\path[draw=drawColor,line width= 1.4pt,dash pattern=on 2pt off 2pt ,line join=round] ( 77.96, 38.02) --
	( 82.11, 38.02) --
	( 82.11, 39.00) --
	( 84.65, 39.00) --
	( 84.65, 39.97) --
	( 85.88, 39.97) --
	( 85.88, 40.95) --
	( 87.56, 40.95) --
	( 87.56, 41.92) --
	( 88.23, 41.92) --
	( 88.23, 42.90) --
	( 89.18, 42.90) --
	( 89.18, 43.88) --
	( 90.91, 43.88) --
	( 90.91, 44.85) --
	( 95.79, 44.85) --
	( 95.79, 45.83) --
	( 96.14, 45.83) --
	( 96.14, 46.80) --
	( 96.91, 46.80) --
	( 96.91, 47.78) --
	( 97.24, 47.78) --
	( 97.24, 48.76) --
	( 98.69, 48.76) --
	( 98.69, 49.73) --
	( 99.50, 49.73) --
	( 99.50, 50.71) --
	(100.80, 50.71) --
	(100.80, 51.69) --
	(101.27, 51.69) --
	(101.27, 52.66) --
	(101.71, 52.66) --
	(101.71, 53.64) --
	(105.53, 53.64) --
	(105.53, 54.61) --
	(105.73, 54.61) --
	(105.73, 55.59) --
	(105.79, 55.59) --
	(105.79, 56.57) --
	(107.26, 56.57) --
	(107.26, 57.54) --
	(107.39, 57.54) --
	(107.39, 58.52) --
	(107.48, 58.52) --
	(107.48, 59.49) --
	(107.59, 59.49) --
	(107.59, 60.47) --
	(107.71, 60.47) --
	(107.71, 61.45) --
	(108.17, 61.45) --
	(108.17, 62.42) --
	(109.51, 62.42) --
	(109.51, 63.40) --
	(110.05, 63.40) --
	(110.05, 64.38) --
	(110.11, 64.38) --
	(110.11, 65.35) --
	(111.20, 65.35) --
	(111.20, 66.33) --
	(111.21, 66.33) --
	(111.21, 67.30) --
	(112.56, 67.30) --
	(112.56, 68.28) --
	(113.36, 68.28) --
	(113.36, 69.26) --
	(115.30, 69.26) --
	(115.30, 70.23) --
	(117.76, 70.23) --
	(117.76, 71.21) --
	(117.84, 71.21) --
	(117.84, 72.18) --
	(117.88, 72.18) --
	(117.88, 73.16) --
	(118.90, 73.16) --
	(118.90, 74.14) --
	(119.88, 74.14) --
	(119.88, 75.11) --
	(120.68, 75.11) --
	(120.68, 76.09) --
	(121.67, 76.09) --
	(121.67, 77.07) --
	(122.27, 77.07) --
	(122.27, 78.04) --
	(123.02, 78.04) --
	(123.02, 79.02) --
	(124.99, 79.02) --
	(124.99, 79.99) --
	(125.78, 79.99) --
	(125.78, 80.97) --
	(126.74, 80.97) --
	(126.74, 81.95) --
	(127.90, 81.95) --
	(127.90, 82.92) --
	(129.69, 82.92) --
	(129.69, 84.88) --
	(129.79, 84.88) --
	(129.79, 83.90) --
	(130.06, 83.90) --
	(130.06, 84.88);
\definecolor{drawColor}{RGB}{215,25,28}

\path[draw=drawColor,line width= 1.4pt,dash pattern=on 4pt off 2pt ,line join=round] ( 87.81, 38.02) --
	(102.06, 38.02) --
	(102.06, 39.00) --
	(112.34, 39.00) --
	(112.34, 39.97) --
	(113.25, 39.97) --
	(113.25, 40.95) --
	(119.27, 40.95) --
	(119.27, 41.92) --
	(129.69, 41.92) --
	(129.69, 41.92);
\definecolor{drawColor}{RGB}{69,117,180}

\node[text=drawColor,anchor=base west,inner sep=0pt, outer sep=0pt, scale=  1.00] at ( 40,120) {79/100};
\definecolor{drawColor}{RGB}{252,141,89}

\node[text=drawColor,anchor=base west,inner sep=0pt, outer sep=0pt, scale=  1.00] at ( 40,110) {49/100};
\definecolor{drawColor}{RGB}{215,25,28}

\node[text=drawColor,anchor=base west,inner sep=0pt, outer sep=0pt, scale=  1.00] at ( 45,100) {5/100};
\end{scope}
\begin{scope}
\path[clip] (  0.00,  0.00) rectangle (144.54,144.54);
\definecolor{drawColor}{gray}{0.30}

\node[text=drawColor,anchor=base east,inner sep=0pt, outer sep=0pt, scale=  1.00] at ( 33.83, 33.60) {0};

\node[text=drawColor,anchor=base east,inner sep=0pt, outer sep=0pt, scale=  1.00] at ( 33.83, 58.00) {25};

\node[text=drawColor,anchor=base east,inner sep=0pt, outer sep=0pt, scale=  1.00] at ( 33.83, 82.41) {50};

\node[text=drawColor,anchor=base east,inner sep=0pt, outer sep=0pt, scale=  1.00] at ( 33.83,106.81) {75};

\node[text=drawColor,anchor=base east,inner sep=0pt, outer sep=0pt, scale=  1.00] at ( 33.83,131.22) {100};
\end{scope}
\begin{scope}
\path[clip] (  0.00,  0.00) rectangle (144.54,144.54);
\definecolor{drawColor}{gray}{0.30}

\node[text=drawColor,anchor=base,inner sep=0pt, outer sep=0pt, scale=  1.00] at ( 59.87, 20.78) {$10^3$};

\node[text=drawColor,anchor=base,inner sep=0pt, outer sep=0pt, scale=  1.00] at ( 94.78, 20.78) {$10^5$};

\node[text=drawColor,anchor=base,inner sep=0pt, outer sep=0pt, scale=  1.00] at (129.69, 20.78) {$10^7$};
\end{scope}
\begin{scope}
\path[clip] (  0.00,  0.00) rectangle (144.54,144.54);
\definecolor{drawColor}{RGB}{0,0,0}

\node[text=drawColor,anchor=base,inner sep=0pt, outer sep=0pt, scale=  1.00] at ( 88.93,  6.94) {Branch-and-bound nodes};
\end{scope}
\begin{scope}
\path[clip] (  0.00,  0.00) rectangle (144.54,144.54);
\definecolor{drawColor}{RGB}{0,0,0}

\end{scope}
\end{tikzpicture}}%
  }\caption{ECDF of BnB nodes for the shortest-path problem (left) and the
  knapsack problem (right) with discrete budgeted uncertainty sets (top) and
  discrete knapsack uncertainty sets (bottom). The numbers in the legend
  indicate the number of instances solved within the time limit of 2 hours out
  of 100 instances. Solid \blue{blue}: \blue{\textsf{MibS}}. Dotted
  \orange{orange}: \orange{\textsf{qip-existeval}}. Dashed \red{red}:
  \red{\textsf{qip-bilevel}}.}
  \label{fig:num-results:discrete:ecdf-nodes}
\end{figure}
}

\begin{figure}
  \centering
    \scalebox{1.0}{\input{plots/optl_rev/sp_cont_budg_scatter.tex}}
    \scalebox{1.0}{\input{plots/optl_rev/knapsack_cont_budg_scatter.tex}}
    \scalebox{1.0}{\input{plots/optl_rev/portfolio_cont_budg_scatter.tex}}
    \caption{Scatter plots for the shortest-path problem (top), the knapsack
    problem (center), and the portfolio selection problem (bottom) with
    continuous budgeted uncertainty sets. The numbers in the legend indicate the
    number of instances solved within the time limit of 2 hours out of 10
    instances per size. \blue{Blue} dots: \blue{bilevel approach}.
    \orange{Orange} diamonds: \orange{robust approach}.}
  \label{fig:num-results:cont:budgeted:scatter}
\end{figure}
\begin{figure}
  \centering
    \scalebox{1.0}{\input{plots/sp_cont_knap_scatter.tex}}
    \scalebox{1.0}{\input{plots/knapsack_cont_knap_scatter.tex}}
    \scalebox{1.0}{\input{plots/portfolio_cont_knap_scatter.tex}}
    \caption{Scatter plots for the shortest-path problem (top), the knapsack
    problem (center), and the portfolio selection problem (bottom) with
    continuous knapsack uncertainty sets. The numbers in the legend indicate the
    number of instances solved within the time limit of 2 hours out of 10
    instances per size. \blue{Blue} dots: \blue{bilevel approach}.
    \orange{Orange} diamonds: \orange{robust approach}.}
  \label{fig:num-results:cont:knapsack:scatter}
\end{figure}
\begin{figure}
  \centering
    \scalebox{1.0}{\input{plots/optl_rev/sp_disc_budg_scatter.tex}}
    \scalebox{1.0}{\input{plots/optl_rev/knapsack_disc_budg_scatter.tex}}
    \caption{Scatter plots for the shortest-path problem (top) and the knapsack
    problem (bottom) with discrete budgeted uncertainty sets. The numbers in the
    legend indicate the number of instances solved within the time limit of 2
    hours out of 10 instances per size. \blue{Blue} dots: \blue{\textsf{MibS}}.
    \orange{Orange} diamonds: \orange{\textsf{qip-existeval}}. \red{Red}
    triangles: \red{\textsf{qip-bilevel}}.}
  \label{fig:num-results:discrete:budgeted:scatter}
\end{figure}
\begin{figure}
  \centering
    \scalebox{1.0}{\input{plots/optl_rev/sp_disc_knap_scatter.tex}}
    \scalebox{1.0}{\input{plots/knapsack_disc_knap_scatter.tex}}
    \caption{Scatter plots for the shortest-path problem (top) and the knapsack
    problem (bottom) with discrete knapsack uncertainty sets. The numbers in the
    legend indicate the number of instances solved within the time limit of 2
    hours out of 10 instances per size. \blue{Blue} dots: \blue{\textsf{MibS}}.
    \orange{Orange} diamonds: \orange{\textsf{qip-existeval}}. \red{Red}
    triangles: \red{\textsf{qip-bilevel}}.}
  \label{fig:num-results:discrete:knapsack:scatter}
\end{figure}

\FloatBarrier

\end{document}